\documentclass[11pt]{amsart}
\usepackage[T1]{fontenc}
\usepackage[utf8]{inputenc}
\usepackage[english]{babel}
\usepackage{amsmath, amsthm, amssymb, amsfonts, enumerate, mathrsfs, dsfont, comment, bm, mathtools}
\usepackage{algorithm}
\usepackage{algpseudocode}
\usepackage{enumitem}
\usepackage[colorlinks=true,linkcolor=blue,urlcolor=blue]{hyperref}
\usepackage{color, geometry, todonotes, indentfirst, graphicx, pdfpages}
\usepackage[title]{appendix}
\usepackage{fancyhdr}
\usepackage{polski}

\allowdisplaybreaks
\newcommand{\jo}[1]{{\textcolor{blue}{#1}}}

\newtheorem{theorem}{Theorem}[section]
\newtheorem{remark}[theorem]{Remark}
\newtheorem{assumption}[theorem]{Assumption}
\newtheorem{lemma}[theorem]{Lemma}
\newtheorem{proposition}[theorem]{Proposition}

\theoremstyle{plain}
\def \dd{{\rm d}}
\def \R{\mathbb{R}}

\def \P{\mathbb{P}}
\def \E{{\mathbb{E}}}
\def \F{\mathbb{F}}

\def \cF{{\mathcal F}}
\def \A{\mathcal{A}}
\def \and{\quad \text{and} \quad}

\DeclareMathOperator*{\tr}{tr}

\DeclareMathOperator*{\esssup}{ess\,sup}

\title[Exploratory OS]{Exploratory optimal stopping: A singular control formulation}
\author[Dianetti]{Jodi Dianetti}
\author[Ferrari]{Giorgio Ferrari}
\author[Xu]{Renyuan Xu}
\thanks{J. Dianetti: Department of Economics and Finance, University of Rome Tor Vergata. Email: \href{mailto:jodi.dianetti@uniroma2.it}{jodi.dianetti@uniroma2.it}}

 \thanks{G. Ferrari: Center for Mathematical Economics, Bielefeld University. Email: \href{mailto:giorgio.ferrari@uni-bielefeld.de}{giorgio.ferrari@uni-bielefeld.de}}
\thanks{R. Xu: Department of Management Science and Engineering, Stanford University. Email: \href{mailto:ryxu@stanford.edu}{ryxu@stanford.edu
}}

\date{\today}

\numberwithin{equation}{section}

\begin{document}

\begin{abstract}
This paper explores continuous-time and state-space optimal stopping problems from a reinforcement learning perspective. We begin by formulating the stopping problem using randomized stopping times, where the decision maker's control is represented by the probability of stopping within a given time—specifically, a bounded, non-decreasing, càdlàg control process. 
To encourage exploration and facilitate learning, we introduce a regularized version of the problem by penalizing the performance criterion with the cumulative residual entropy of the randomized stopping time.  The regularized problem takes the form of an $(n+1)$-dimensional degenerate singular stochastic control with finite-fuel, where the regularized free boundary becomes the graph of a function mapping the state variable of the original stopping problem into the probability of stopping. 
We address this singular control problem through the dynamic programming principle, which enables us to identify the unique optimal exploratory strategy.  
Finally, we propose both model-based and model-free reinforcement learning algorithms tailored for exploratory optimal stopping problems. We establish policy improvement guarantees for the proposed algorithms. Moreover, the model-free method is of actor–critic type and it is scalable in high-dimensions under neural network parameterization.

\end{abstract} 
\maketitle  

\smallskip 
{\textbf{AMS subject classification}}: 
%91A15, %Stochastic games, stochastic differential games 
%49N70, %Differential games and control
60G40,  %Stopping times; optimal stopping problems; gambling theory
93E20, %Optimal stochastic control
35R35,  %Free boundary problems for PDEs
68T01  %General topics in artificial intelligence
%60H10 %Stochastic ordinary differential equations (aspects of stochastic analysis) 
\tableofcontents

\section{Introduction} 
In optimal stopping (OS) problems, a decision maker chooses a time to
take a given action based on an adaptively observed stochastic process in order to optimize an expected performance criterion. Problems of this type are found in the area of Statistics, where the action taken may be to test a hypothesis or to estimate a parameter; in the area of Operations Research, where the action may be to replace a machine, hire a secretary, or reorder stock; and in Mathematical Finance, 
in the context of pricing American contingent claims or in irreversible investment problems
(we refer to the textbook \cite{peskir.shiryav.2006optimal} and to the references therein).
%{The study of OS problems lead to significant advancement in probability theory, and still presents many open challanges related to the characterization of the optimal strategies (see e.g. \cite{deangelis.peskir.2020.global.regularity}) or to the problem of finding equilibria in the multi-agent framework (see \cite{de2022value} and the references therein).}

Despite the wide-ranging applications of optimal stopping across various fields, most of the existing literature relies on the {\it full knowledge} of the system, including both the underlying process and the reward function. Little is known about addressing these problems in a model-free context, where decision-makers interact with an unknown system and learn to improve their decisions over time. In this regard, machine learning techniques, particularly reinforcement learning (RL), offer a promising framework. However, the existing literature has largely focused on ``frequently controlling the system'', leaving the RL framework for optimal stopping largely unexplored.

Generally speaking, RL aims to learn optimal strategies to control an unknown dynamical system (or random environment) by
interacting with the system through exploration and exploitation \cite{sutton2018reinforcement}.  
In recent years, the availability of vast amounts of data and advances in computational resources have spurred extensive research in both theoretical developments and practical applications of RL. 
While much of the existing literature on RL has focused on discrete-time problems, 
%real-world systems in natural sciences and engineering—such as robotics \cite{spong2020robot}, biology \cite{lenhart2007optimal}, medicine \cite{panetta2003optimal}, and finance \cite{merton1992continuous}—are fundamentally {\it continuous in time}. Compared to the discrete-time setting, 
continuous-time RL is rapidly growing and presents unique challenges and technical difficulties in algorithmic design, stability analysis, and efficiency guarantees. 
A key question is how to balance the trade-off between exploring unknown information in the continuous-time system and exploiting current information to take near-optimal actions. 
In their seminal work, \cite{wang2020reinforcement} propose adding Shannon's entropy to the objective function of a continuous-time relaxed linear-quadratic control problem to encourage exploration, leading to an optimal distribution policy of Gaussian-type with the mean capturing the exploited policy and the variance quantifying the degree of exploration. 
Since then, there have been studies on continuous-time RL through the lens of partial differential equations and control theory under entropy regularization, primarily focusing on {\it regular controls}, such as controlling the drift and volatility coefficients of the underlying stochastic processes \cite{bai2023reinforcement,huang2022convergence,jia2022policy,jia2023q,ma2024convergence,tang2022exploratory,tran2024policy,wang2020reinforcement}.

One key challenge that distinguishes OS from regular control is that OS involves a non-smooth decision—whether to stop or continue—while regular controls gradually change the dynamics through drift and/or volatility coefficients. 
As a result, gradient-based RL algorithms, although popular for regular controls, cannot be directly applied to the ``stop-or-continue'' decisions in our setting \cite{reppen2022neural,soner2023stopping}. 
To overcome this difficulty, we replace these sharp rules based on the hitting times by stopping probabilities, leading to a {\it randomized stopping time} that follows the stopping probability. 
Unlike the heuristic of a fuzzy boundary considered in \cite{reppen2022neural}, where the agent stops with a probability proportional to the distance to the boundary, our approach to randomized stopping time is based on principled guidance—regularized objective function with cumulative residual entropy.  
Another advantage of randomized stopping time is its exploratory nature; it stops at different scenarios according to certain probability, thereby collecting more information from the unknown environment in the RL regime. 
Compared to classic RL settings, exploration is particularly essential for optimal stopping problems because the terminal reward can only be collected upon making the stopping decision, contributing to the {\it reward sparsity} challenge in RL \cite{devidze2022exploration,hare2019dealing}.
Therefore, randomized stopping time helps the decision-maker learn more about the unknown terminal reward.

Following the above idea, this paper proposes a framework to address continuous time and state-space OS problems via continuous-time RL. 
We first embed the perspective of exploration-exploitation into OS, and then use this construction to design a new type of RL algorithm for which we provide a convergence analysis. 
 %{\color{red}[$\leftarrow$ We do have an algorithm with convergence  rate analysis, not just the insights. I don't think we need to downgrade this part :-). Also, I think it might be more suitable to put this "summary sentence" after the following paragraph."]}

\subsection{Our work and contributions}
In this section, we provide an informal discussion on the key ideas and contributions of this work. 
The precise assumptions and statements of the results can be found in Sections \ref{section 2}-\ref{sec:general_algorithm}.

The classical OS problem (without regularization) is given, for any initial state $x \in \R^n$, as
\begin{equation*} 
\begin{aligned}
    V(x) &:=  \sup _{\tau} \mathbb{E}\left[\int_{0}^{\tau} e^{- \rho s} \pi (X^x_{s}) \dd s +e ^{- \rho \tau} G(X^x_{\tau})\right], \\
   &\text{subject to} \quad \dd X^x_{t}  =b(X^x_{t}) \dd t+\sigma(X^x_{t}) d W_{t}, \quad X^x_{0}=x, 
\end{aligned}
\end{equation*} 
where the maximization is performed over stopping times $\tau$ of the (augmented) filtration generated by the Brownian motion $W$, $\rho>0$, and the functions $b, \sigma,  \pi, G $ are the data of the problem (satisfying the assumptions in Section \ref{section 2}).

To model exploration, we follow the ideas in \cite{touzi.vieille.2002continuous} and let the agent randomize the choice of the stopping time.
The resulting strategies are chacterized by nondecreasing processes $\xi=(\xi_t)_t$ with $0 \leq \xi_t \leq 1$,
where $\xi_{t}$ is the probability that the agent stops before time $t$ according to 
$\xi_{t}=\mathbb{P}( \tau \leq t | \cF _{t}^{W} )$.
We refer to such processes $\xi$ as \emph{singular controls}. 
However, allowing for randomized stopping times {\it does not change} the optimal value function and, more importantly, it does not imply that optimal actions are necessarily randomized/exploratory (see Proposition \ref{prop xi does not explore}). 
From the RL point of view, this non-exploratory behavior of the optimal control suggests that, without proper modification of the objective function, {\it optimizing} and {\it collecting information} do not naturally come together. 

To overcome the above-mentioned non-exploratory phenomenon and to incentivize exploration, we reward intermediate values of $\xi_t$  by perturbing the original OS problem via the
cumulative residual entropy (CRE) \cite{rao.chen.vemuri.wang.2004cumulative} of the probability measures $\xi$
\begin{eqnarray*}
\begin{aligned}
    {\rm CRE} (\xi) &:= - \int_{0}^{\infty} e^{- \rho t}\left(1-{\xi_{t}}\right) \log \left(1-{\xi_{t}}\right) \dd t.
\end{aligned}
\end{eqnarray*} 
The adoption of CRE in RL is {\it novel} in the literature, as  Shannon's entropy and Kullback–Leibler (KL) divergence are more often used. It is worth noting that this CRE criterion is more suitable for encouraging randomized optimal stopping time by postponing the exercise time via a large cumulative probability. 

The resulting \emph{entropy regularized OS problem} takes the form of the singular control problem:
\begin{equation*}
     V^\lambda (x):=\sup _{\xi} \mathbb E \bigg[ \int _0^\infty e ^{-\rho t} \underbrace{\big( \pi (X^x_t) (1-\xi_t) \dd t  + G(X^x_t) \dd \xi_t \big)}_{\text{exploitation}} -\lambda \int_{0}^{\infty} e^{- \rho t} \underbrace {\left(1-\xi_{t}\right) \log \left( 1-\xi_{t} \right)}_{\text{exploration}} \dd t \bigg],
\end{equation*}
where $\lambda>0$ is a temperature parameter that balances the exploitation and exploration.
For any $\lambda >0$, such a problem admits a unique optimal control $\xi^\lambda$.
%It is clear that the regularized problem \eqref{eq:value_v_lambda}-\eqref{eq:dynamics_intro}-\eqref{eq:dynamics_y_intro} is a singular control problem of non-standard type, with $X$ a (uncontrolled) high-dimensional diffusion process and $Y$ a controlled  one-dimensional process that is degenerate.
%It is well-known in the stochastic control literature that singular control problems and optimal stopping problems are connected in the sense that the value function of a singular control problem can bewritten as the derivative (in the state variable) of the value function of an optimal problem \cite{bather1966continuous,boetius1998connections,de2018stochastic,guo2009class,Karatzas&Shreve84}. Our reformulation \eqref{eq:value_v_lambda} builds  a {\it new connection} between optimal stopping and singular control problems, through the exploratory RL viewpoint. 
The proposed regularization naturally approximates the original OS problem. 
Indeed, we show that (see formal statements in Propositions \ref{proposition vanishing entropy Value} and  \ref{proposition vanishing entropy Control}):
 \begin{itemize}
     \item  $
\sup _x | V^\lambda (x) - V (x)| \leq   
  \lambda (\rho e)^{-1} \to 0 $ as $\lambda \to 0$.
  \item For any sequence $(\lambda_k)_k \to 0$, the sequence $(\xi^{\lambda_k})_k$ converges (up to subsequences) to an optimal singular control $\xi^*$ of the original OS with randomized stopping times.
 \end{itemize}
These convergence results have also been discovered in \cite{wang2020reinforcement} for the linear-quadratic control with Shannon's entropy, using explicit formulas for the value functions and optimal controls.

A detailed study of the entropy regularized OS problem is then performed by applying the dynamic programming principle (DPP).  
To this end, we introduce the additional controlled state variable $Y^{y,\xi}_t = y - \xi_t$, $y \in [0,1]$, and define, for $x \in \R^n$ and $\lambda >0$, the extended problem
\begin{equation*}
     V^\lambda (x,y):=\sup _{\xi_t \leq y} \mathbb E \bigg[ \int _0^\infty e ^{-\rho t} \big( \pi (X^x_t) Y^{y,\xi}_t - \lambda Y^{y,\xi}_t \log (Y^{y,\xi}_t) \big)\dd t  +  \int _0^\infty e ^{-\rho t} G(X^x_t) \dd \xi_t \bigg].
\end{equation*}
This is an $(n+1)$-dimensional {\it degenerate} singular stochastic control problem with finite-fuel.
Notice that $V^\lambda (x,1) = V^\lambda (x)$.
In particular, we show that (see Theorems \ref{theorem value function} and \ref{theorem optimal control} for the formal statements):
\begin{itemize}
\item The value function ${{V^\lambda}}$ is $ W_{loc}^{2, 2}( \R^n \times (0,1))$ and it is the unique solution to the Hamilton-Jacobi-Bellman (HJB) variational inequality 
$$
\max \big\{ \left(\mathcal{L}_{x} -\rho \right) V^\lambda(x, y) + \pi (x) {y - \lambda y \log y}, - V^\lambda_{y}(x, y) + G(x) \big\}=0, 
$$
 with $V^\lambda(x, 0)=0$,  where
$
\mathcal L _x f (x,y) := b (x) D_x f(x,y)  + \frac12 \tr \big( \sigma  \sigma^* (x) D_x^2 f (x,y) \big).
$
\item The optimal control $\xi^{\lambda}$ is of reflecting type and is given by
$
{\xi^\lambda_t := \sup_{s\leq t } \big( y- g_\lambda(X^x_s) \big)^+},
$
with the reflecting free boundary  defined as ${{g_\lambda}} (x) := \sup \big\{ y \in[0,1]  \, | \, - V^\lambda_y (x,y) +  G(x) < 0 \big\}$.
\end{itemize}
These results illustrate that the optimal control of the entropy regularized OS problem is not related to any strict stopping time; that is, by introducing the entropy regularization, we obtain optimal strategies of exploratory type.
From the RL point of view, the exploratory behavior of the optimal control is desirable, as it suggests that {\it optimizing} and {\it collecting information} from the environment can actually be achieved together. 
{Moreover, we obtain a characterization of the minimal optimal stopping time $\tau^*$ in terms of the optimal control $\xi^\lambda$, given by 
$$
\tau ^* = \inf \{ t \, | \, \xi^\lambda_t \geq 1-e^{-1} \}, \quad \text{for any $\lambda \geq 0$.}
$$
Thus, learning the optimal reflecting strategy $\xi^\lambda$ provides the optimal stopping time of the original problem.
Finally, it is important to notice that, while the free boundary of the original stopping problem can be (only) locally expressed as a graph of function, the regularized free boundary becomes the graph of a global function mapping the state variable of the original stopping problem into the probability of stopping.}

From the technical point of view, showing the regularity of $V^\lambda$ (required for the characterization of $\xi^\lambda$) involves several challenges. 
In particular, while the regularity in $x$ is studied via semi-convexity estimates and PDE arguments, the regularity in $y$ is obtained via a probabilistic connection with a different but related OS problem. 
 
%In a benchmark real option example, we can solve the entropy regularized OS problem semi-explicitly (see Theorem \ref{theorem verification real option} for the details).
%For this particular example, we can characterize the free boundary $g_\lambda$ and show its convergence, as $\lambda \to 0$, to the free boundary of the original OS problem.
%Inspired by the fact that the optimal control $\xi^\lambda$ can be fully characterized by the boundary ${{g_\lambda}}$, we develop an RL algorithm based on updating the boundary iteratively. 

Our theoretical analysis suggests a new approach to design RL algorithms for OS problems, which aims at learning the boundary of the $\lambda$-entropy regularized problem rather than the boundary of the original OS problem.
Notice that the value mismatch introduced by the temperature parameter $\lambda$ is of order $\mathcal{O}(\lambda)$.
The proposed RL framework consists of two versions: a model-based numerical version when all model parameters are known, and a model-free learning version when model parameters are not known. For the first version, where all model parameters are known, we design a Policy Iteration algorithm to numerically search for $g_\lambda$. For the second version, where the model parameters are unknown, we provide a sample-based Policy Iteration algorithm. Notably, we do not directly estimate model parameters in the design, which enhances robustness against model misspecification and environmental shifts \cite{agarwal2021theory,fazel2018global,hambly2021policy}. {Our new approach also overcomes the instability issue often encountered when directly learning the boundary of the original OS problem \cite{reppen2022neural}.}

We establish the theoretical foundation of the algorithm for the real option example. Specifically, in the $k$th iteration of the first step, assuming sufficient regularities, we use the second order derivative of $V^\lambda_{g_k}$ (i.e., the value function associated with the policy using reflecting boundary $g_k$) with respect to $y$ to update $g_{k+1}$, which lies in a region where $V^\lambda_{g_k}$ has a {\it better regularity}.  Mathematically, define the new boundary ${g}_{k+1}(x)$ as:
  \begin{equation}\label{eq:updated_policy_intro}
      \begin{cases}
       & \max \Big\{0\leq y < g_{k}(x)\,\Big|\, \partial_{yy} V^\lambda_{g_{k}}(x,y) =0     \Big\} \,\,\text{if} \,\, \partial_{yy}^- V^\lambda_{g_{k}}(x,g_{k}(x))>0, \\
    &  {g}_{k+1}(x) = g_{k} (x)  \qquad \text{otherwise},  
      \end{cases}
  \end{equation}
  where we define the notation $\partial_{yy}^- f(x,y):= \lim_{h\rightarrow0-}\frac{\partial_y f(x,y+h)-\partial_y f(x,y)}{h}$ as the left $y$-derivative of $\partial_y f$ (if exists). With such an updating scheme, we have the following results (see Theorem \ref{prop:policy_improvement} for the formal statement):
\begin{itemize}
    \item \textbf{Policy improvement:} $V^\lambda_{g_{k+1}}(x,y)\ge V^\lambda_{g_k}(x,y)$ for all $k\in \mathbb{N}$.
\end{itemize}
Policy improvement follows from the weak form of It\^o's formula. Related results appear in \cite{wang2020reinforcement} for linear--quadratic control and in \cite{dianetti2026reinforcement} for a one-dimensional real options problem. However, the updating schemes in \eqref{eq:updated_policy_intro} are different.

We then develop an actor–critic algorithm in an unknown environment: the agent observes only discretely sampled reward increments and changes of dynamics  along trajectories (Algorithm \ref{al:agent-environment}). Specifically, the critic (value network) is trained by minimizing a TD(0)-style squared TD error, and the actor (policy network) is updated according to \eqref{eq:updated_policy_intro}, with the critic replacing the unknown value function (see Algorithm \ref{alg:actor-critic}). The algorithm is tested in numerical experiments and demonstrates promising performance (see Section \ref{sec:numerics}).

\subsection{Related literature} 

Our work is related to different streams of literature.
\vspace{0.2cm}

\paragraph{\bf Randomized optimal stopping and singular control.} 
First of all, we relate to the literature on randomized stopping times. 
These typically appear when proving the existence of equilibria in zero-sum or nonzero-sum optimal stopping games (see \cite{de2022value, kifer2013dynkin, riedel2017subgame, solan2012random, touzi.vieille.2002continuous}, among others) and can be thought of as the ``mixed strategies'' in stopping times. Loosely speaking, according to a randomized stopping time, a decision maker chooses 
a probability of stopping and at any time decides or not to stop according to that distribution. As discussed in \cite{solan2012random} and \cite{touzi.vieille.2002continuous}, different but equivalent definitions of randomized stopping times can be developed. In this paper, we follow the approach in \cite{touzi.vieille.2002continuous}, which identifies randomized stopping times with nondecreasing c\`adl\`ag processes starting from zero at the initial time and bounded above by one. 
This in turn leads to the fact that the exploratory formulation of the optimal stopping problem is mathematically equivalent to an $(n+1)$-dimensional degenerate singular stochastic control with finite-fuel (see, e.g., \cite{angelis2019solvable, karatzas2000finite,karoui1988probabilistic}).
\vspace{0.2cm}

\paragraph{\bf Continuous-time RL for regular control.}
Initial studies on continuous-time RL primarily focused on algorithm design and the potential for time discretization \cite{doya2000reinforcement,palanisamy2014continuous,tallec2019making}. 

Following the seminal work of \cite{wang2020reinforcement}, which provided a mathematical perspective within the framework of linear-quadratic control, there has been a recent surge in employing stochastic control theory and partial differential equations (PDEs) to establish the theoretical foundation of continuous-time RL \cite{bai2023reinforcement,huang2022convergence,jia2023q,jia2022policy,ma2024convergence,tang2022exploratory,tran2024policy}. 
In particular, \cite{tang2022exploratory} generalized the framework in \cite{wang2020reinforcement} and studied the exploratory Hamilton--Jacobi--Bellman (HJB) equation arising from a general entropy-regularized stochastic control problem; \cite{jia2023q} developed q-Learning algorithm in continuous time; \cite{jia2022policy} established policy gradient and actor-critic learning frameworks in continuous time and space.  Despite these developments, the overall convergence results of control-inspired or PDE-inspired algorithms, such as the Policy Improvement Algorithm (PIA), remain largely unexplored.
Very recently, \cite{huang2022convergence} proved a qualitative convergence result for the PIA with bounded coefficients when the diffusion term is not controlled.  \cite{tran2024policy} generalized the result to unbounded coefficients or controlled diffusion terms by a uniform estimate for the value sequence generated by PIA. \cite{ma2024convergence}
 provides a much simpler proof of the convergence of the PIA using Feynman-Kac-type probabilistic representation instead of sophisticated PDE estimates.
Moreover, \cite{bai2023reinforcement} showed the convergence of the PIA for an optimal dividend problem, taking advantage that the state process is one dimensional and taking non-negative values. In addition to the PIA method, another popular continuous-time  RL algorithm is the policy gradient method. \cite{reisinger2023linear} demonstrated a linear convergence of this method for certain finite-horizon non-linear control problems, while \cite{giegrich2024convergence} showed its convergence for finite-horizon exploratory linear-quadratic control problems.

Another important aspect of quantifying the effectiveness of continuous-time RL algorithms is regret analysis. For theoretical developments in this area, see \cite{basei2022logarithmic,guo2023reinforcement,szpruch2021exploration,neuman2023offline,szpruch2024optimal}.

We note that all the works mentioned above focus on regular controls.

\vspace{0.2cm}

\paragraph{\bf Machine learning for optimal stopping.}
RL for optimal stopping is closely related to the challenging RL scenario of sparse rewards \cite{devidze2022exploration,hare2019dealing}. More specifically, the terminal reward 
$G(X_\tau)$ can only be collected upon making the stopping decision, contributing to the sparsity of the rewards. This sparsity significantly complicates the learning process compared to the more straightforward regular control problems in continuous-time settings or classical Markov decision processes in discrete time.

When model parameters are fully known to the decision maker, \cite{reppen2022neural} and \cite{soner2023stopping} developed deep-learning-based learning algorithms to learn the stopping boundaries. \cite{bayer2021randomized} considered randomized optimal stopping problems and
established convergence rates for the corresponding forward and backward Monte Carlo optimization algorithms. See also \cite{ata2023singular}  for solving high-dimensional singular control problems using deep learning methods. 

It is worth noting that \cite{denkert2024control} proposed a comprehensive framework for policy gradient methods tailored to continuous-time RL. This framework leverages the connection between stochastic control problems and randomized problems, allowing for applications beyond diffusion models, such as regular, impulse, and optimal stopping/switching problems. However, no theoretical convergence results have been established yet. The most relevant paper to our setting is \cite{dong2024randomized}, which considers a regularized American Put Option problem under Shannon's entropy in one dimension, which involves a single risky asset that follows a geometric Brownian motion. The author employs an intensity control formulation to encourage exploration and demonstrates the convergence of PIA for the regularized problem under a given temperature parameter. 
%However, it is not possible to recover the convergence of the optimal policy for this intensity control formulation when the temperature parameter $\lambda$ tends to zero, even for the simple American Put Option problem.
In comparison to the latter work, we highlight some novelties of our approach. 
Notably, our analysis of the regularized problem is based on regularity analysis that combines PDE arguments and a probabilistic connection with another OS problem. In contrast, the analysis in \cite{dong2024randomized} relies primarily on explicit calculations, which are challenging to extend to higher dimensions. Moreover, within our framework, we can show the convergence of the optimal policies of the regularized problems to the original OS-optimal policy as the temperature parameter $\lambda$ tends to zero.
While our convergence holds regardless of the dimension of the problem, such a convergence is not clear in \cite{dong2024randomized} even for the specific American Put Option problem. {Recently, \cite{dai2024learning} proposed an alternative approach by approximating the optimal stopping problem as a penalized regular control problem with binary decisions, $0$ and $1$. They then applied entropy regularization to the binary decision process, transforming it into a classical entropy-regularized stochastic control problem, to which existing RL algorithms can be applied.}

Another relevant strand of literature leverages non-parametric statistics for stochastic processes to design learning algorithms for control \cite{christensen2023data2,christensen2023nonparametric,christensen2024learning,christensen2023data}. Specifically, \cite{christensen2023nonparametric} and \cite{christensen2024learning} addressed one-dimensional singular and impulse control problems with a non-parametric statistical diffusion structure, where strategies are characterized by one or two values that need to be learned. \cite{christensen2023data} extended this approach to multivariate reflection problems, while \cite{christensen2023data2} emphasized statistical properties such as regret and non-asymptotic PAC bounds.

\subsection{Outline of the paper} 
Section \ref{section 2} examines the properties of the original OS problem, introduces the entropy-regularized version, and explores its relationship with the original problem. The entropy-regularized problem is further explored in Section \ref{section DPP approach}, including a regularity analysis of the value function and the construction of the optimal policy.  
Section \ref{sec:general_algorithm} then shifts focus to the development of learning algorithms, with model-based numerical algorithms studied in Section \ref{sec:model-based} and model-free learning algorithms developed in Section \ref{sec:model-free}.
Finally, Appendix \ref{appendix proof uniqueness stopping time} contains the proof of the uniqueness of the optimal stopping time, while the proofs of some auxiliary technical estimates are collected in Appendix \ref{appendix proof of estimates}.

\subsection{General notation}
For $q\in [1,\infty]$, and a measure space $(E, \mathcal{E}, m)$, we define    
 the set $\mathbb L^q= \mathbb L^q(E)=\mathbb L^q(E,m)$ of measurable functions $f:E\to \mathbb{R}$  s.t. $\int_E |f|^q dm < \infty$, if $q <\infty$, and $\esssup_E |f| < \infty$ for $q=\infty$.
For $d \in \mathbb{N} \setminus \{0\}$, an open set $B \subset \mathbb{R}^d$, $\alpha=(\alpha_1,...,\alpha_{d})\in \mathbb{N}^{d}$, and a function $f:B \to \mathbb{R}$, we denote by $D^\alpha f:= D_1^{\alpha_1} ... D_{d}^{\alpha_{d}} f$ the weak derivative of $f$, where $D_i f:=f_{x_i} := \partial f / \partial x_i$, and we set $|\alpha|:=\alpha_1 + ... +\alpha_d$. 
Sometimes, we will denote by $D_x f$ and $D_x^2f$ the gradient and the Hessian matrix of $f$ (in the weak sense), respectively.
Define the set $C(B)$  of continuous functions $f:B\to \R$,  the set  $C^{\ell }(B)$  of functions $f:B\to \R$ with continuous $\ell$-order derivatives,
and the Sobolev spaces ${W}^{\ell,q}(B)$ of functions $f \in \mathbb L^q(B)$ with $\sum_{|\alpha| \leq \ell} \int_E |D^{\alpha} f|^q dm< \infty$, if $q <\infty$, and $\sum_{|\alpha| \leq \ell} \esssup_E |D^{\alpha} f| < \infty$, for $q=\infty$.
Finally, define the Sobolev spaces ${W}_{loc}^{\ell,q}(B)$ of functions $f$ s.t. $f \in  {W}^{\ell,q}(D)$ for each bounded open set $D\subset B$.

\section{Exploratory formulation and entropy regularization of the OS problem}\label{section 2}
Consider a discount factor $\rho>0$ and, for $n \in \mathbb N \setminus \{0\}$, continuous functions
$ b: \mathbb{R}^{n} \rightarrow \mathbb{R}^{n}, \ \sigma: \mathbb{R}^{n} \rightarrow \mathbb{R}^{n \times m}, \  \pi, G :\mathbb R ^n \to \mathbb R.$
Let  $W=(W_t)_t$ be an $m$-dimensional Brownian motion on a complete probability space $(\Omega, \mathcal F, \mathbb F, \mathbb P )$ and let $\mathbb F^W = (\mathcal F ^W_t)_{t}$ be the right-continuous extension of the filtration generated by $W$.  
Introduce the second-order differential operator
\begin{equation*}\label{eq operator}
    \mathcal L _x \phi (x) := b(x) D_x \phi (x) + \frac12 \tr( \sigma \sigma^* (x) D_x^2 \phi(x) ), \quad \phi \in W_{loc}^{2, 2} \left(\mathbb{R}^{n}\right), 
\end{equation*}
and denote by $\mathcal T$  the set of $\mathbb F^W$-stopping times. 

For an initial condition $x \in \R^n$ and a stopping time $\tau \in \mathcal T$, define the profit functional
$$
\begin{aligned}
J(x;\tau) & := \mathbb{E}\left[\int_{0}^{\tau} e^{- \rho s} \pi (X^x_{s}) \dd s +e ^{- \rho \tau} G(X^x_{\tau})\right], \\
 \text{subject to } \dd X^x_{t} & =b(X^x_{t}) \dd t+\sigma(X^x_{t}) \dd W_{t}, \quad X^x_{0}=x, 
\end{aligned} 
$$
and consider the optimal stopping problem
\begin{equation}\label{eq:value_classic_os}
    V(x) :=  \sup _{\tau \in \mathcal T}  J(x;\tau).
\end{equation}
A stopping time $\tau^*$ is said to be optimal if $J(x;\tau^*) =V(x)$ and the function $V$ is referred to as the value function.

In order to have a well-defined problem, we enforce some requirements on the data of the problem. 
\begin{assumption}
    \label{assumption first}
    The following conditions hold true: 
    \begin{enumerate}
    \item\label{assumption first: condition b sigma} There exists a constant $L>0$ such that, for $\phi = b, \sigma$, we have
    $$ \begin{aligned}
         | \phi (x)| & \leq L (1+|x|),\quad  
        |\phi (\bar x) - \phi (x)|  \leq L |\bar x - x|, \quad \text{for any $x, \bar x \in \mathbb R ^n$}.
    \end{aligned}
    $$ 
    \item\label{assumption first: condition costs} There exist $p\geq 2$ and $K>0$ such that, for $\phi = G, \pi$, we have
    $$
    \begin{aligned}
    |\phi(x)| &\leq K (1+|x|^p), \quad \text{for any $ x \in \mathbb R ^n$}. 
    \end{aligned}
    $$
    \item\label{assumption first: condition rho} $\rho$ is large enough: 
    $
    \rho >  p (\frac32 L + (2p-1) L^2 ). 
    $
    \end{enumerate}
\end{assumption}
The following result is classical.  
\begin{proposition} 
 Under Assumption \ref{assumption first}, there exists an optimal stopping time. 
\end{proposition}
\begin{proof}
We limit ourselves to directing the reader to a reference. More specifically,  Assumption \ref{assumption first} implies (see estimate \eqref{eq estimate SDE growth} below and the related \eqref{eq definition constants}) that, for a suitable constant $C<\infty$, we have
$$
\mathbb E \left[ \int_0^\infty e^{-\rho t} |\pi(X^x_t)| \dd t + \sup_{t\geq 0} e^{-\rho t} |G(X^x_t)| \right] < C(1+|x|^p).
$$ 
Hence, the result follows by Corollary 2.9 at p.\ 46 in \cite{peskir.shiryav.2006optimal} (notice that, differently from \cite{peskir.shiryav.2006optimal}, we do not require stopping times to be finite).
\end{proof}
 
Some of the later results (see in particular Proposition \ref{prop xi does not explore} and Proposition \ref{proposition vanishing entropy Control} below) enjoy a sharper statement under uniqueness of the optimal stopping time, which is typically satisfied in standard Markovian settings.
For the sake of illustration, we discuss the following sufficient conditions in the one-dimensional case.
\begin{lemma}\label{lemma uniqueness stopping time}
Under Assumption \ref{assumption first}, if
\begin{equation} 
    \label{condition for UNIQUENESS}
    \begin{aligned}
        &\text{$n=1$, $G \in C^2(\R)$, $\sigma^2(x) \geq  c>0$ and $D_x G(x) \leq C (1 + |x|)$ for some $c,C\in (0,\infty)$,} \\
        & \text{the function $ \hat \pi := \pi + (\mathcal L_x -\rho) G$ is strictly increasing and there exists $\bar x $ s.t. $\hat \pi (\bar x)=0$,}
    \end{aligned}
\end{equation}
then the optimal stopping time is unique.
\end{lemma}
The proof of this lemma relies on classical arguments in OS and it is provided in Appendix \ref{appendix proof uniqueness stopping time}.

\subsection{Exploratory formulation via singular controls}
Inspired by \cite{touzi.vieille.2002continuous}, we introduce a notion of randomized/exploratory stopping times.
Define the set of \emph{singular controls} as the set of processes
\begin{align*}
    \A (1):= \Big\{ \xi : & \, \Omega \times[0, \infty) \to [0,1], \text{ $\F ^W$-adapted, nondecreasing, càdlàg, with $\xi_{0-}=0$} \Big\}, 
\end{align*} 
and assume the probability space $(\Omega, \cF, \P)$ to be large enough to accommodate a $W$-independent uniformly distributed random variable $U:\Omega \to [0,1]$.

Given any $\xi \in \mathcal A (1)$, we consider the random variable $\tau^\xi$ on $(\Omega, \mathcal F)$ by setting $\tau^\xi:=\inf \left\{t \geq 0 \, | \, \xi_{t}>U\right\}$, with the convention $\inf \emptyset = +\infty$, and we will refer to it as \emph{randomized/exploratory stopping time}.
Notice that $\tau^\xi$ is not necessarily an $\mathbb F ^W$-stopping time.
However, if $\tau \in \mathcal T$, the process $\xi^\tau \in \mathcal A (1)$ defined by $\xi^\tau := ( \mathds 1 _{\{ t \geq \tau \} } )_t$ is such that $\tau = \tau^{\xi^\tau}$, which gives a natural inclusion of $\mathcal T$ into $\mathcal A (1)$.
Moreover, we have
$$
\mathbb{P} \big( \tau^\xi \leq t \mid \cF _{t}^{W}\big)
=\mathbb{P}\left( U \leq {\xi}_{t} \mid \cF _{t}^{W}\right)
=\int_{0}^{\xi_{t}} \dd u=\xi_{t}, 
$$
so that $\xi_t$ can be interpreted as the probability of stopping before time $t$.
Notice also that $\xi_\infty := \lim_{t \to \infty} \xi_t$ is not necessarily equal to 1: indeed, the related randomized stopping time $\tau^\xi$ is not necessarily finite.

For $\xi \in \mathcal A (1)$, by evaluating the performance of the randomized stopping time $\tau^\xi$, we obtain
\begin{align*}
J(x;\tau^\xi)  =& 
\mathbb{E}\bigg[ \int_{0}^{\tau^\xi} e^{- \rho t} \pi (X^x_{t}) \dd t +e ^{- \rho \tau^\xi} G\left(X^x_{\tau^\xi}\right)\bigg] \\
 =& \mathbb{E}\left[\int_{0}^{\infty} e^{- \rho t} \pi (X^x_{t}) \mathds{1}_{\{t \leq \tau^\xi \}} \dd t 
 +\int_{0}^{\infty} e ^{- \rho t} G(X^x_t ) \mathbb{P} \big( \tau^\xi \in \dd t \mid \cF^{W}\big) \right] \\
 =& \mathbb{E}\left[\int_{0}^{\infty} e^{-\rho t} \pi (X^x_{t})  \mathds{1}_{\left\{\xi_{t}< U\right\}} \dd t +  \int_{0}^{\infty} e ^{- \rho t} G(X^x_t ) \dd \xi_t \right] \\
 =& \mathbb{E}\left[\int_{0}^{\infty} e^{-\rho t} \pi (X^x_{t}) \left(1-\xi_{t}\right) \dd t + \int_{0}^{\infty} e ^{- \rho t} G(X^x_t ) \dd \xi_t  \right].
\end{align*}
{Here and in the sequel, with slight abuse of notation we set $\int_0^\infty \phi_t \dd \xi_t : = \int_{0-}^\infty \phi_t \dd \xi_t$ for any continuous function $\phi$ for which the integral is well defined.}
The previous expression suggests to define a profit functional in terms of the singular controls $\xi \in \mathcal A (1)$ by
\begin{equation*}
    J^0 (x;\xi)  := \mathbb{E}\left[\int_{0}^{\infty} e^{-\rho t} \pi (X^x_{t}) \left(1-\xi_{t}\right) \dd t + \int_{0}^{\infty} e ^{- \rho t} G(X^x_t ) \dd \xi_t  \right], \quad \xi \in \mathcal A (1),
\end{equation*}
as a natural extension of $J$ to randomized stopping times.

However, allowing for randomized stopping times does not change the optimal value and (crucial for our story) it  does not imply that optimal actions are necessarily randomized, as discussed in the next proposition and remark.

\begin{proposition}\label{prop xi does not explore}
Under Assumption \ref{assumption first}, for any $x \in \R^n$ we have 
$$
V(x) = \sup _{\xi \in \mathcal A (1) } J^0(x;\xi) = J^0 (x;\xi^*), 
$$
with $\xi^* := ( \mathds 1 _{ \{ t \geq \tau^* \} } )_t$ and $\tau^*$ optimal for $J(x;\cdot)$.    
Moreover, if the  optimal stopping time $\tau^*$ is unique (e.g., if \eqref{condition for UNIQUENESS} holds), then $\xi^* := (\mathds 1 _{ \{ t \geq \tau^* \} })_t$ is the unique optimal control for $J^0(x;\cdot).$
\end{proposition}
\begin{proof}
For generic $\xi \in \mathcal A (1)$ and $z \in [0,1]$, we can define the stopping time $\tau ^\xi (z)$ as
$$
\tau^\xi (z) := \inf \{ t \geq 0 \, | \, \xi_t \geq z \},
$$
with the convention $\inf \emptyset := + \infty$.
Indeed, notice that $\tau^\xi (z) = + \infty$ if $z > \xi_\infty$.
By using Fubini's theorem and then the change of variable formula (see Chapter 0 in \cite{revuz.yor.2013continuous}), we have
$$
\begin{aligned}
    J^0(x&;\xi) \\
    =& \mathbb{E}\left[\int_{0}^{\infty} e^{-\rho t} \pi (X^x_{t}) \left(1 -\xi_\infty + \int_t^\infty  \dd \xi_s  \right) \dd t + \int_{0}^{\infty} e ^{- \rho t} G(X^x_t )  \dd \xi_t  \right] \\
     =& \mathbb{E}\left[ (1 -\xi_\infty) \int_{0}^{\infty} e^{-\rho t} \pi (X^x_{t})   \dd t + \int_{0}^{\infty}  \left( \int_0^t e^{-\rho s} \pi (X^x_{s})   \dd s +  e ^{- \rho t} G(X^x_t ) \right)  \dd \xi_t  \right] \\
     =& \mathbb{E}\left[ (1 -\xi_\infty) \int_{0}^{\infty} e^{-\rho t} \pi (X^x_{t})   \dd t 
     +\int_{0}^{\xi_\infty} \bigg( \int_0^{\tau^\xi (z)} e^{-\rho s} \pi (X^x_{s})   \dd s +  e ^{- \rho {\tau^\xi (z)}} G(X^x_{\tau^\xi (z)} ) \bigg)   \dd z  \right] \\
     =& \int_{0}^1 \mathbb{E}\bigg[   \int_0^{\tau^\xi (z)} e^{-\rho s} \pi (X^x_{s})   \dd s +  e ^{- \rho {\tau^\xi (z)}} G(X^x_{\tau^\xi (z)} )   \bigg]   \dd z 
     =\int_{0}^1 J(x;\tau^\xi (z) )  \dd z.
\end{aligned}
$$
Thus, since  $\tau^\xi(z)$ is an $\mathbb F ^W$-stopping time, by the optimality of $\tau^*$ we have
$$
J^0(x; \xi) \leq \int_{0}^1 \sup_{\tau \in \mathcal T } J(x;\tau )   \dd z = J(x;\tau^*) = V(x) = J^0(x;\xi^*), 
$$
where the last equality follows from the fact that  $\xi^*$ is such that $\tau^{\xi^*}(z) = \tau^*$ for any $z \in (0,1]$. This proves the first part of the proposition.

Assume now that the optimal stopping time $\tau^*$ is unique.
By repeating the previous argument with $\xi$ being optimal, we obtain that $J(x;\tau^\xi (z)) = J(x;\tau^*)$ $ \dd z$-a.e.\ in $(0,1)$. 
By uniqueness of the optimal stopping time, we deduce that $\tau^\xi (z) = \tau^*$ $ \dd z$-a.e.\ in $(0,1)$, which in turn implies $\xi = \xi^*$.
\end{proof}

\begin{remark}[Non-exploratory behavior of the optimal controls]\label{remark xi does not explore}
We elaborate on Proposition \ref{prop xi does not explore} from the RL and exploratory points of view. 
Under the uniqueness of the optimal stopping time $\tau^*$, the optimal singular control $\xi^*$ in fact corresponds to the strict stopping time $\tau^*$. 
Thus, no randomization is needed for the optimal strategy. 
In particular, once an action becomes necessary, the agent stops the process, which prevents the gradual collection of information on the performance of other actions.
From the RL point of view, this non-exploratory behavior of the optimal control $\xi^*$ suggests that optimizing and collecting information do not naturally come together, and a modification is necessary in order to overcome this phenomenon (see the related Remark \ref{remark the optimal control randomizes} below for a discussion on the consequences and benefits of our entropy regularization). 
\end{remark}

\subsection{Entropy regularization}
In light of Proposition \ref{prop xi does not explore} and of the related Remark \ref{remark xi does not explore}, we introduce a regularization term to incentivize exploration/randomization.  This is achieved by incorporating an entropy term to regularize the problem, drawing motivation from the RL literature.

Since the sample paths of the processes $\xi \in \mathcal A (1)$ are not necessarily absolutely continuous with respect to the Lebesgue measure, we choose (a discounted version of) the \emph{cumulative residual entropy} (see \cite{rao.chen.vemuri.wang.2004cumulative}) weighted by a parameter $\lambda >0$; 
namely, we consider, for $\xi \in \A (1)$ the entropy
\begin{equation}\label{eq:CRE}
\begin{aligned}
\Lambda^\lambda (\tau^\xi) :=&-\lambda \int_{0}^{\infty} e^{- \rho t} \mathbb{P} \left( \tau^\xi \geq t \mid \cF _{t}^{W} \right) \log \left(\mathbb{P}\left(\tau^\xi \geq t \mid \mathcal{F}_{t}^{W}\right)\right)  \dd t \\
 =& -\lambda \int_{0}^{\infty} e^{- \rho t}\left(1-\xi_{t}\right) \log \left(1-\xi_{t}\right)  \dd t =: \Lambda^\lambda (\xi).
\end{aligned}
\end{equation}
The entropy $\Lambda^\lambda$ is nonnegative and achieves its highest values when the probability $\xi_t$ is near the level $e^{-1}$, thus incentivising the use of randomized stopping times.

Building on this intuition,  we define an entropy regularized OS problem 
\begin{equation}\label{eq value singular control problem}
    V^\lambda (x) := \sup _{\xi \in \mathcal A (1)} J^\lambda ( x; \xi ), 
\end{equation}
where the exploration--exploitation trade off is captured by the profit functional 
\begin{equation*}
    J^\lambda ( x; \xi ):= \mathbb E \bigg[ \int _0^\infty e ^{-\rho t} \underbrace{\big( \pi (X^x_t) (1-\xi_t) \dd t  + G(X^x_t) \dd\xi_t \big)}_{\text{exploitation}} -\lambda \int_{0}^{\infty} e^{- \rho t} \underbrace {\left(1-\xi_{t}\right) \log \left( 1-\xi_{t} \right)}_{\text{exploration}} \dd t \bigg].
\end{equation*}
A control $\xi^\lambda \in \mathcal A (1)$ is said to be optimal if $J^\lambda (x;\xi^\lambda ) = V^\lambda (x)$.

Before studying the entropy regularized problem and its connections with the original stopping problem, we list here few remarks. 
\begin{remark}\label{remark existence of optimal control} 
For any $x \in \R^n$ and $\lambda >0$, the existence of a unique optimal control $\xi^\lambda$ can be established by exploiting the strict concavity of the functional $J^\lambda(x; \cdot)$ {in the control variable $\xi$} (see in particular Theorem 8 in \cite{Menaldi.Taksar89}).  
Moreover, Theorem \ref{theorem optimal control} below will characterize such an optimal control and illustrate that the randomized strategy maximizing $J^\lambda$ does not correspond to a (classical) stopping time (see also Remark \ref{remark the optimal control randomizes}). 
\end{remark}

\begin{remark}[Connection between optimal stopping and singular control] 
It is well-known in the stochastic control literature that optimal stopping problems and singular control problems enjoy intimate connections.
On the one hand,  the derivative (in the controlled state variable) of the value function of a singular control problem can be written as the value function of an optimal stopping problem \cite{bather1966continuous,boetius1998connections,de2018stochastic,guo2009class,Karatzas&Shreve84}.
On the other hand, randomized stopping strategies can be naturally seen as finite-fuel singular controls (see \cite{touzi.vieille.2002continuous} or the discussion above), which has the advantage of compactifying the space of stopping strategies (thus explaining its heavy use in the Dynkin game literature).

In this framework, our entropy regularization  adds two crucial properties on the latter connection. 
First, while solely considering randomized strategies still leads to a degenerate singular control problem (where the optimal actions are pure jumps, as in Proposition \ref{prop xi does not explore}), our cumulative residual entropy regularizes the problem by making it concave and allows for a more explicit characterization of the optimal control in terms of a reflecting strategy (see Theorem \ref{theorem optimal control} below). Secondly, this entropy regularization very well approximates the original OS problem, both in terms of value and optimal strategy (see Propositions \ref{proposition vanishing entropy Value} and \ref{proposition vanishing entropy Control} in the next subsection). For these reasons, our entropy regularization seems to be a useful tool in order to approximate/characterize the equilibria also in the context of Dynkin games, which we leave for future research.
\end{remark}

\begin{remark}
Our choice to work with a Markovian formulation is motivated by RL. 
Since the typical optimization problem that RL aims to solve (in the discrete time case) is a Markov decision problem, starting our study with a Markovian optimal stopping problem seems to be a natural choice. 
However, it should be pointed out that our entropy regularization formulation (and its subsequent analysis) also allows us to treat optimal stopping problems in a non-Markovian framework. 
In this regard, we limit ourselves to mentioning that (while we will solve the Markovian entropy-regularized problem using the DDP approach in Section \ref{section DPP approach} below), in order to tackle the non-Markovian entropy-regularized OS problem, one can rely on representation theorems via backward stochastic differential equations (see \cite{bank&elkaroui2004, Bank&Riedel01} among others).
\end{remark}

\subsection{Vanishing entropy limit}
With elementary arguments, one can show that the entropy regularized problem approximates the original OS problem as the temperature parameter $\lambda$ tends to zero.
\begin{proposition}\label{proposition vanishing entropy Value}
Under Assumption \ref{assumption first}, for any $\lambda , \bar \lambda \in [0,1]$ we have 
$$
\sup _{ x \in \R ^n } | V^\lambda ( x) - V^{\bar \lambda } (x)| \leq   
 | \lambda - \bar \lambda| (\rho e)^{-1}.
 $$
In particular, $V^\lambda  \to V $ uniformly on $ \mathbb R ^n$,  as $\lambda \to 0$.    
\end{proposition}

\begin{proof}
Take $x \in \mathbb R$ and  $\lambda , \bar \lambda \in [0,1]$.
If $\lambda \geq \bar \lambda$, for every $\xi \in \mathcal A (1)$ we have  $J ^\lambda (x;\xi) \geq J ^{\bar \lambda } (x;\xi)$, so that
$$
0 \leq V^\lambda ( x) - V^{\bar \lambda } (x).
$$
Moreover, upper bounding the right-hand side above, we find
$$
\begin{aligned}
0 \leq V^\lambda ( x) - V^{\bar \lambda } (x)  
& = \sup _{\xi \in \mathcal A (1)} J ^\lambda (x;\xi) - \sup _{\bar \xi \in \mathcal A (1)} J ^{\bar \lambda } (x; \bar \xi) \\
& \leq \sup _{\xi \in \mathcal A (1)} \Big( J ^\lambda (x;\xi) -  J ^{\bar \lambda } (x;  \xi) \Big) \\
 & = (  \lambda - \bar \lambda)   \sup _{\xi \in \mathcal A (1)} \mathbb E  \bigg[\int_0^\infty e ^{-\rho t}     \big( - (1-\xi_t) \log (1-\xi_t) \big)  \dd t \bigg] \\
 & = ( \lambda - \bar \lambda) \frac1{\rho} \sup_{z \in [0,1]} (-z \log z) =  ( \lambda - \bar \lambda) (\rho e)^{-1},
\end{aligned}
$$
which is the claimed bound. 
We conclude the proof by choosing $\bar \lambda  = 0$ and taking limits as $\lambda \to 0$.
\end{proof}

\begin{proposition}\label{proposition vanishing entropy Control}
For any $x \in \R^n$ and $\lambda >0$, let  $\xi^\lambda$ be the unique optimal control as in Remark \ref{remark existence of optimal control}.
Under Assumption \ref{assumption first}, the following statements hold true:
\begin{enumerate}
    \item There exist a subsequence $(\lambda_k)_k$ with $\lambda_k \to0$ and an optimal control $\xi^*$ for $J^0 (x;\cdot)$ such that
$\xi^{\lambda_k} \to \xi^*$ weakly in $\mathbb L ^2 (\Omega \times [0,\infty) )$ as $k \to \infty$.
    \item If the optimal stopping time $\tau^*$ is unique (e.g., if \eqref{condition for UNIQUENESS} holds), then $ \xi ^\lambda  \to   (\mathds 1 _{ \{ t \geq \tau^* \} })_t $ weakly in $\mathbb L ^2 (\Omega \times [0,\infty) )$ as $\lambda \to 0$.     
\end{enumerate}
\end{proposition}

\begin{proof} 
Since the family $(\xi^\lambda)_\lambda$ is bounded,  we can find a subsequence $(\lambda_k)_k$  with $\lambda_k \to 0$   and a limit point $\xi^*$ such that
$\xi^{\lambda_k} \to \xi^*$ weakly in $\mathbb L ^2 (\Omega \times [0,\infty) )$ as $k \to \infty$.
We need to show that $\xi^*$ is optimal for $J^0(x;\cdot)$. 

First of all, using Proposition \ref{proposition vanishing entropy Value}, we find
\begin{equation}\label{eq J to V convergence 1}
\begin{aligned}
 |V(x) - J^0(x; \xi^{\lambda_k} ) | 
 & \leq |V(x)- J^{\lambda_k} (x; \xi^{\lambda_k} ) | + | J^{\lambda_k} (x; \xi^{\lambda_k} ) - J^0(x; \xi^{\lambda_k} )| \\
 & \leq |V(x)- V^{\lambda_k} (x ) | + \E [ \Lambda ^{\lambda_k}  (\xi^{\lambda_k}) ] \\
 & \leq C {\lambda_k} \to 0, \quad  \text{as $k\to \infty$.}
\end{aligned}
\end{equation}
Moreover, since $\sup_{t \geq 0} \xi^{\lambda_k}_t \leq 1$, we can employ (after some minimal adjustment to take care of our infinite time horizon $[0,\infty)$) Lemma 3.5 in \cite{K} in order to find a subsequence (not relabelled)  $(\lambda_k)_k$ with $\lambda_k \to 0$ and a limit point $\xi \in \mathcal A (1)$ such that, setting
$
\zeta^m_t := \frac1m \sum_{k=1}^m \xi^{\lambda_k}_t,
$
we have
$$
\int_0^\infty e^{-\rho t} f_t \dd\zeta^m_t \to \int_0^\infty e^{-\rho t} f_t \dd\xi_t,  \ \mathbb P \text{-a.s., as $m \to \infty$, for any $f \in C_b([0,\infty))$.}
$$
It is easy to show that $\xi = \xi^*$. 
Furthermore, since $\zeta^m \in \mathcal A (1)$, by Proposition  \ref{prop xi does not explore} and the concavity of $J^0 (x;\cdot)$ we find
$$
\begin{aligned}
    V(x) \geq J^0 (x; \zeta^m) &\geq \frac{1}{m} \sum_{k=1}^m J^0(x;\xi^{\lambda_k}) \\
    & = \frac{1}{m} \sum_{k=1}^m J^{\lambda _k} (x;\xi^{\lambda_k}) +  \frac{1}{m} \sum_{k=1}^m \big(J^{0} (x;\xi^{\lambda_k}) - J^{\lambda _k} (x;\xi^{\lambda_k}) \big) \\
    & = \frac{1}{m} \sum_{k=1}^m J^{\lambda _k} (x;\xi^{\lambda_k}) -  \frac{1}{m} \sum_{k=1}^m \E [\Lambda^{\lambda_k} (\xi^{\lambda_k})] \\
     & \geq \frac{1}{m} \sum_{k=1}^m J^{\lambda _k} (x;\xi^{\lambda_k}) -  (e \rho)^{-1} \frac{1}{m} \sum_{k=1}^m {\lambda_k} \to V(x), \quad \text{as $m\to \infty$,}
\end{aligned}
$$
where the limits in the right-hand side follows from \eqref{eq J to V convergence 1}.
Thus, we obtain
$
\lim_{m} J^0(x;\zeta^m) = V(x). 
$  
By using the convergence of $\zeta^m $ to $ \xi^*$ and the dominated convergence theorem, the latter limits implies that
$$
J^0(x;\xi^*) = \lim_{m} J^0(x;\zeta^m) = V(x),
$$
proving the optimality of $\xi^*$ ({\it cf.} \ Proposition \ref{prop xi does not explore}).
 
If the optimal stopping time $\tau^*$ is unique, then, by Proposition \ref{prop xi does not explore}, the optimal randomized stopping time $\xi^*$ is unique as well.
Therefore, $\xi^*$ is the unique limit point of any subsequence of $(\xi^\lambda)_\lambda$, so that the whole sequence converges to $\xi^*$.
\end{proof}

\begin{remark}
From the RL algorithmic design point of view, it is important to notice that Proposition \ref{proposition vanishing entropy Value} quantifies the error which is made when replacing the original OS problem \eqref{eq:value_classic_os} with the entropy regularized OS problem \eqref{eq value singular control problem}.  
\end{remark}
\begin{remark}\label{remark vanishing entropy limits boundary}
Regarding Propositions \ref{proposition vanishing entropy Value} and \ref{proposition vanishing entropy Control}, one could try to investigate the convergence of the optimal strategies in terms of their characterization.
More precisely, we will prove in Theorem \ref{theorem optimal control} below that the optimal singular control of the entropy-regularized OS problem is of reflection type, 
while an optimal stopping time for the OS problem is (under suitable assumptions) an hitting time $\tau^* = \inf\{ t \geq 0 \, |\, V(X^x_t) = G(X^x_t) \}$.
Thus, it is natural to question whether, and in what sense, the reflecting boundary of the entropy-regularized OS problem \eqref{eq value singular control problem} converges to the optimal stopping boundary of the original OS problem  \eqref{eq:value_classic_os}.
Although establishing this result at a general level may be challenging, a positive answer is provided in our companion paper \cite{dianetti2026reinforcement}, in which we study a classical example of an OS problem arising from real option evaluation.
Moreover,  this question is further addressed in Section \ref{subsection More on the vanishing entropy limit}.
\end{remark}

\section{Analysis of the entropy regularized OS problem }\label{section DPP approach}
{In this section we study the entropy regularized OS problem, obtaining a characterization of the value function in terms of a related HJB equation (see Theorem \ref{theorem value function}) and a characterization of the optimal control as a reflecting strategy (see Theorem \ref{theorem optimal control}). 
We then show further relations with the original OS problem, deriving a formula for the minimal optimal stopping time in terms of the reflecting strategy (see Theorem \ref{theorem optimal stopping boundary representation}) and a further vanishing entropy limit (see Theorem \ref{theorem convergence of modified strategies}).
}

\subsection{Solving the entropy regularized OS problem via dynamic programming}\label{subsection DPP approach}
To approach  problem \eqref{eq value singular control problem} via DPP we introduce, for $y \in[0,1]$, the additional controlled state process
$ 
Y_{t}^{y,\xi} = y-\xi_{t}
$
and the set of admissible controls 
\begin{eqnarray}\label{eq:admissible_set_y}
    \A(y) : = \{ \xi \in \A(1) \text{ with } \xi_{t} \in[0, y] \text{ a.s.\ for any } t \geq 0\}.
\end{eqnarray}
Thus, for $(x,y) \in \mathbb R ^n \times [0,1]$ and $\xi \in \A (y)$, define the profit functional
$$
\begin{aligned}
J^\lambda (x, y;\xi) & :=  \mathbb{E} \left[ \int_0^{\infty}  e^{- \rho t} \big( ( \pi (X^x_{t} )Y_{t}^{y,\xi}-\lambda Y_{t}^{y,\xi} \log  (Y_{t}^{y,\xi} ) \big) \dd t 
+  \int_0^{\infty}  e^{- \rho t} G(X^x_t) \dd\xi_t \right],
\end{aligned} 
$$
subject to 
\begin{eqnarray}
    \dd X^x_{t} &=&b(X^x_{t}) \dd t+\sigma(X^x_{t}) \dd W_{t}, \quad X^x_{0}=x, \label{eq:x-process}\\
 \dd Y_t^{y,\xi} &=& -\dd\xi_t, \quad Y^{y,\xi}_{0-} =y.\label{eq:y-process}
\end{eqnarray}

We then consider the optimization problem
\begin{equation}\label{eq:singular_value}
    V^\lambda (x, y):=\sup _{\xi \in \A(y)}  J^\lambda (x, y;\xi),
\end{equation}
and we notice that $ V^\lambda (x) = V^\lambda(x, 1)$.

In order to study the entropy regularized problem, we introduce the following requirements, which imply Assumption \ref{assumption first}. 
\begin{assumption}
    \label{assumption general}
    The following conditions hold true: 
    \begin{enumerate}
    \item\label{condition b sigma} There exists a constant $L>0$ such that, for $\phi = b, \sigma$, we have
    $$ \begin{aligned}
         | \phi (x)| & \leq L (1+|x|),\\ 
        |\phi (\bar x) - \phi (x)|  & \leq L |\bar x - x|, \\
        | \phi ( \delta \bar x + (1-\delta)x) - \delta \phi(\bar x) - (1-\delta) \phi ( x) |  &\leq  L \delta (1-\delta)  |\bar{x}-x|^2,
    \end{aligned}
    $$
    for any $x, \bar x \in \mathbb R ^n$, $\delta \in [0,1]$. 
    \item\label{condition costs} There exist $p\geq 2$ and $K>0$ such that, for $\phi = G, \pi$, we have
    $$
    \begin{aligned}
    |\phi(x)| &\leq K (1+|x|^p), \\
    |\phi (\bar x) - \phi (x) | &\leq K (1+|\bar x|^{p-1} + |x|^{p-1})|\bar x -x |, \\
    |\phi ( \delta \bar x + (1-\delta)x) - \delta \phi(\bar x) - (1-\delta) \phi ( x) | &\leq  K \delta (1-\delta) (1+|\bar x|^{p-2} + |x|^{p-2}) |\bar{x}-x|^2, 
    \end{aligned}
    $$
    for any $x, \bar x \in \mathbb R ^n$, $\delta \in [0,1]$. 
    \item\label{codition ellipticity} The matrix $a(x):=\sigma^* \sigma (x) $ is uniformly elliptic; that is, there exists a constant $\kappa_\sigma >0$ such that 
    $$
    \sum_{i,j=1}^n a_{i,j} (x) z_i z_j \geq \kappa_\sigma |z|^2, \quad \text{for any $z \in \mathbb R ^n$.}
    $$
    \item\label{condition rho} $\rho$ is large enough: namely,
    %$\rho > \frac{\hat c _0 ( 4(p-1))}{2}, \hat c _1(2), {\hat c _1(4)} ,     \frac{\hat c _0 ( 2(p-1))+ \hat c _1(2)}{2},     \frac{\hat c _0(2p)}{2},    \hat c _0 (p),    \frac{\hat c _0 ( 2(p-1))+ \hat c _2(2)}{2},    \hat c _2 (2),\\    \frac{\hat c _2 (4)}2,     \frac{\hat c _1 (8)}{2},     \frac{\hat c _0 ( 4(p-2))}{2}    $
     $$
     \begin{aligned}
    \rho > \max \Big\{ & 
    \frac{\hat c _0 ( 4(p-1))}{2}, 
    \hat c _1(2), 
    {\hat c _1(4)} , 
    \frac{\hat c _0 ( 2(p-1))+ \hat c _1(2)}{2}, 
    \frac{\hat c _0(2p)}{2},
     \hat c _0 (p), \\
    & \qquad  \frac{\hat c _0 ( 2(p-1))+ \hat c _2(2)}{2}, 
    \hat c _2 (2),
    \frac{\hat c _2 (4)}2, 
    \frac{\hat c _1 (8)}{2}, 
    \frac{\hat c _0 ( 4(p-2))}{2} \Big\},
    \end{aligned}
    $$
    where, for generic $q \geq 0$ we define the constants
    \begin{equation}\label{eq definition constants}
    \begin{aligned}
        \hat c _0(q) &:= q \Big(\frac32 L + (q-1) L^2 \Big), \\
        \hat c _1(q) &:= q \Big(L+\frac{q-1}{2}L^2 \Big), \\
        \hat c _2 (q)& := {2 q\Big(L+\frac{2 q-1}{2} L^{2}\Big) }+ {(2 q-1) L+2(q-1)^{2} L^{2}}, 
    \end{aligned}
    \end{equation}
    which are related to the growth behavior of the underlying diffusion (see the estimates \eqref{eq estimate SDE growth}-\eqref{eq estimate SDE semiconv SUP} below).
    \end{enumerate}
\end{assumption}

The following theorem characterizes the value function of the entropy regularized OS problem.
The proof of the existence of a solution to the HJB equation is given in Subsection \ref{subsection proof V solves HJB}, while the proof of its uniqueness is given in Subsection \ref{subsection uniqueness HJB}.
\begin{theorem}\label{theorem value function}
The value function $V^\lambda$ is $C (\R^n \times [0,1]) \cap W_{loc}^{2, 2} \left(\mathbb{R}^{n} \times (0,1) \right)$, is concave in $y$, is such that $V^\lambda(x,y) \leq C (1+|x|^p)$ and it solves in the a.e.\ sense the HJB equation
$$
\max \big\{ \left(\mathcal{L}_{x} -\rho \right) V^\lambda(x, y) + \pi (x) y - \lambda y \log y, - V^\lambda_{y}(x, y) + G(x) \big\}=0, 
$$
with boundary condition $V^\lambda(x, 0)=0$.

{Moreover, the HJB equation admits a unique solution in the class of $ C (\R^n \times [0,1]) \cap W_{loc}^{2, 2} \left(\mathbb{R}^{n} \times (0,1) \right)$ functions which are concave in $y$ and such that $V(x,y) \leq C (1+|x|^p)$. 
}
\end{theorem}

We next move to the characterization of the optimal control.
Define the free boundary function $g_\lambda : \R^n \to [0,1]$ by
\begin{equation}\label{eq def free boundary}
    g_\lambda (x) := \sup \big\{ y \in[0,1]  \, | \, - V^\lambda_y (x,y) +  G(x) < 0 \big\},
\end{equation}
where we set $g_\lambda (x):=1$ if $\{ y \in[0,1] \, | \, - V^\lambda_y (x,y) +  G(x) < 0 \} = \emptyset$.
Notice that the function $g_\lambda$ is well defined by the concavity of $V^\lambda$, and that it is upper semi-continuous because of the continuity of $-V_y^\lambda + G$. 
{Intuitively, the function $g_\lambda$ separates the ``exploration region'' 
\begin{equation}
    \label{eq exploration region}
    \mathcal E_\lambda := \{ (x,y) \, | \, - V^\lambda_y (x,y) +  G(x) < 0 \},
\end{equation}
 in which no action is  required, from the ``stopping region'' 
 \begin{equation}
    \label{eq stopping region}
\mathcal S _\lambda := \{ (x,y) \, | \, - V^\lambda_y (x,y) +  G(x) = 0 \},
\end{equation}
in which the agent has to act.}

\begin{remark}
    We underline that, while the free boundary of the original stopping problem can be (only) locally expressed as a graph of function, the regularized free boundary becomes the graph of the  global function $g_\lambda$ mapping the state variable $x \in \R^n$ of the original stopping  problem into the probability of stopping $y \in [0,1]$.
\end{remark}

The following theorem characterizes the optimal control in terms of the function $g_\lambda$.
Its proof can be essentially found in the Step 2 in Subsection \ref{subsection uniqueness HJB} and it  is briefly summarized in Subsection \ref{subsection proof optimal control}.

\begin{theorem}\label{theorem optimal control}
For any $(x,y) \in \R ^n \times [0,1]$, there exists a unique optimal control $\xi^\lambda \in \mathcal A (y)$, which is given by the reflection policy at the boundary $g_\lambda$; that is,
$$
\xi^\lambda_t := \sup_{s\leq t } \big( y- g_\lambda (X^x_s) \big)^+, \quad t\geq 0, \quad \xi^\lambda_{0-} = 0.
$$
%For any $(x,y)$, there exists a unique optimal control $\xi^{*} \in \mathcal A (y)$, and $\xi^*$ is such that $\mathbb{P} \otimes d t$-a.e.,that
%$$(\mathcal L _{x}-\rho) V^\lambda (X_{t}, Y_{t}^{\xi^{*}})+\pi (X_{t}) Y_{t}^{\xi^{*}}-\lambda Y_{t}^{\xi^{*}} \log (Y_{t}^{\xi*})=0,$$
%and that $\mathbb{P}$-a.s., $\forall t \geqslant 0 t$, 
%$$ \xi_{t}^{*}=\int_{0}^{t} \mathds{1}_{ \{ V^\lambda_{y} (X_{s}, Y_{s}^{\xi^*} ) =G (X_{s}) \}} d \xi_{s}^{*}. $$
\end{theorem}

A few comments are in order.
\begin{remark}[Optimal control and the related Skorokhod problem]
Notice that the optimal control $\xi^{\lambda}$ is the control with minimal total variation such that
\begin{equation*}
( X^x_{t}, Y_{t}^{y, \xi^{\lambda}}) \in \overline{\mathcal E_\lambda}, \quad \text{for any $t \geq 0, \ \P$-a.s.,}    
\end{equation*}
where $ \overline{\mathcal E_\lambda}$ is the closure of ${\mathcal  E_\lambda}$.
Such a process is also known as the solution to the Skorokhod reflection problem for the underlying state process in the waiting region ${\mathcal  E_\lambda}$. 
Generally, characterizing the optimal singular control in terms of the related Skorokhod problem represents a challenging open problem in singular control theory (we refer to the discussion in \cite{dianetti2021multidimensional} for further details).
In our setting, the explicit characterization of $\xi^\lambda$ (as in Theorem \ref{theorem optimal control}) is possible due to the fact that $X^x$ is uncontrolled and the controls are one-dimensional.
\end{remark}

\begin{remark}[Exploratory behavior of optimal controls]\label{remark the optimal control randomizes}
Unlike the OS problem with exploratory strategies but without modification on the objective function (see Proposition \ref{prop xi does not explore} and the related Remark \ref{remark xi does not explore}), Theorem \ref{theorem optimal control} clarifies that the optimal control of the entropy regularized OS problem is no longer a strict stopping time. 
In other words, introducing entropy regularization leads to optimal strategies of an exploratory nature—more precisely, strategies that are exclusively exploratory.

From an RL perspective, the exploratory behavior of the optimal control is crucial, as it allows for simultaneous optimization and information gathering from the environment. Specifically, once the state process reaches the boundary of the exploration region, not only is action required, but, given the regularity of the free boundary $g_\lambda$, an infinite number of infinitesimal actions are necessary to prevent the state process from entering the stopping region. This enables the agent to continuously gather information on the optimality of selecting a particular reflecting boundary in a non-episodic manner.
\end{remark}

\subsection{More on the vanishing entropy limit}
\label{subsection More on the vanishing entropy limit}
In this subsection we discuss some further results on the connection between our original OS problem and its entropy regularized version. 

We begin by recalling from \eqref{eq exploration region} and \eqref{eq stopping region} the definitions of exploration region $\mathcal E_\lambda := \{ (x,y) \in \R^n \times [0,1] \, | \, - V^\lambda_y (x,y) +  G(x) < 0 \}$  and stopping region $\mathcal S _\lambda = \{ (x,y) \in \R^n \times [0,1]\, | \, - V^\lambda_y (x,y) +  G(x) = 0 \}$.
For any $y \in [0,1]$, define the sets 
$$
\begin{aligned}    
\mathcal E_\lambda (y) &: = \{ x \in \R^n \, | \, - V^\lambda_y (x,y) +  G(x) < 0 \}, \\
\mathcal S _\lambda (y)& := \{ x \in \R^n \, | \, - V^\lambda_y (x,y) +  G(x) = 0 \},
\end{aligned}
$$
and notice that, since $-V_y^\lambda + G \leq 0$ (cf.\  Theorem \ref{theorem value function}), one has $\mathcal S _\lambda (y) = \R^n \setminus \mathcal E_\lambda (y)$. 
Following the arguments in \cite{Karatzas&Shreve84} (see in particular Theorem 3.4) we can show that, for any $(x, y) \in \mathbb{R}^{n} \times(0,1]$, one has the representation
\begin{equation}\label{eq conncection OS}
V^\lambda_y (x, y) =\sup _{\tau \in \mathcal T} \mathbb{E}\left[\int_{0}^{\tau} e ^{-\rho t} \left( \pi\left(X^x_{t}\right) - \lambda (1+\log y) \right) \dd t+e^{-\rho {\tau}} G\left(X^x_{\tau}\right)\right].
\end{equation}
From this representation, we derive the following monotonicity property of the exploration region. 

\begin{lemma}
    \label{lemma monotonicity boundary} 
    Under Assumption \ref{assumption general}, for $0\leq \lambda \leq \bar \lambda$, we have that:
    \begin{enumerate}
        \item If $y > e^{-1}$, then  $\mathcal E_{\bar \lambda} (y) \subseteq   \mathcal E_{ \lambda} (y)$;
        \item If $y=e^{-1}$, then $\mathcal E_{\bar \lambda}(y) = \mathcal E_{ \lambda} (y) $;
        \item If $ y<e^{-1}$, then $ \mathcal E_{ \lambda} (y) \subseteq   \mathcal E_{\bar \lambda} (y) $.
    \end{enumerate}
\end{lemma}
\begin{proof}
We begin by proving the first claim.
When $y > e^{-1}$, we have $1+\log y >0$, hence $ - \bar \lambda (1+\log y) <  - \lambda (1+\log y)$. 
Thanks to the representation \eqref{eq conncection OS}, we obtain that 
$$
V_y^{\bar \lambda} (x,y) \leq V_y^{ \lambda} (x,y), \quad \text{for all $x \in \R ^n$ and $y > e^{-1}$.}
$$
Hence, from the definition of $\mathcal E _\lambda(y)$ we find  $\mathcal E_{\bar \lambda} (y) \subseteq   \mathcal E_{ \lambda} (y)$. 

For $y = e^{-1}$, we have $1+\log y = 0$, hence $\mathcal E _\lambda(y)$ does not depends on $\lambda$, proving the second claim. 
The third claim can be proved analogously. 
\end{proof}

For $\lambda \geq 0$, denote by $\tau^\lambda (x,y)$ the minimal optimal stopping time of the problem \eqref{eq conncection OS}.
Under Assumption \ref{assumption general},  Theorem 2.3.5 in \cite{lamberton.2009optimal} characterizes  the minimal stopping time as the first hitting time to the region  ${\mathcal S_\lambda}$; i.e., $\tau^\lambda (x,y) = \inf \{ t \, | \, X_t^x \in \mathcal S _\lambda (y) \} $. 
Since $\tau^0 (x,y)$ does not depend on $y$, it will be simply denoted by $\tau^0 (x)$.

\begin{lemma}
    \label{lemma limits stopping times}
     Under Assumption \ref{assumption general}, for any $x \in \R ^n$, we have $\tau^\lambda (x,1) \to \tau ^0(x)$, $\P$-a.s., as $\lambda \to 0$.
\end{lemma}
\begin{proof}
    Thanks to Lemma \ref{lemma monotonicity boundary}, 
    for $0\leq \lambda \leq \bar \lambda$, we have that
     $\mathcal E_{\bar \lambda} (1) \subseteq   \mathcal E_{ \lambda} (1)$. 
     This implies that $\mathcal S_{ \lambda} (1) \subseteq   \mathcal S_{\bar \lambda} (1)$, from which we deduce that
     $$
     \tau^{\lambda} (x,1) = \inf \{ t \, | \, X_t^x \in \mathcal S _\lambda (1) \} \geq \inf \{ t \, | \, X_t^x \in \mathcal S _{\bar \lambda} (1) \} = \tau^{\bar \lambda} (x,1).
     $$
 Therefore, the family $\tau^\lambda (x,1)$ is nonincreasing in $\lambda$.
 Using this monotonicity, we can define a limiting stopping time 
 $$
 \tau ^*(x):= \sup_{\lambda >0} \tau^\lambda (x,1) = \lim _{\lambda \to 0} \tau^\lambda (x,1),
 $$
 and we notice that $\tau^0(x) \geq \tau^* (x)$.

Next, by optimality of $\tau^\lambda (x,1)$, we have
$$
\begin{aligned}
    &\mathbb{E}\left[\int_{0}^{\tau^\lambda (x,1)} e ^{-\rho t} \left( \pi\left(X^x_{t}\right) - \lambda (1+\log y) \right) \dd t+e^{-\rho {\tau^\lambda (x,1)}} G\left(X^x_{\tau^\lambda (x,1)}\right)\right] \\
    & \quad \leq  \mathbb{E}\left[\int_{0}^{\tau} e ^{-\rho t} \left( \pi\left(X^x_{t}\right) - \lambda (1+\log y) \right) \dd t+e^{-\rho {\tau}} G\left(X^x_{\tau}\right)\right], \quad \text{for any $\tau \in \mathcal T$.}
\end{aligned}
$$
Taking limits as $\lambda \to 0$ in the latter inequality (using the growth conditions in Assumption \ref{assumption general} and the dominated convergence theorem), we deduce that
$$
\begin{aligned}
    &\mathbb{E}\left[\int_{0}^{\tau^* (x)} e ^{-\rho t}  \pi\left(X^x_{t}\right)  \dd t+e^{-\rho {\tau^* (x)}} G\left(X^x_{\tau^* (x)}\right)\right] \\
    & \quad \leq  \mathbb{E}\left[\int_{0}^{\tau} e ^{-\rho t}   \pi\left(X^x_{t}\right) \dd t+e^{-\rho {\tau}} G\left(X^x_{\tau}\right)\right], \quad \text{for any $\tau \in \mathcal T$.}
\end{aligned}
$$
Hence, $\tau^*(x)$ is also optimal. Thus, since   $\tau^0(x)$ is the minimal optimal stopping time and $\tau^0(x) \geq \tau^* (x)$, we conclude that $\tau^*(x) = \tau^0(x)$. 
\end{proof} 

For the rest of this subsection, we fix the initial condition of the problem to $(x,1)$, for some $x \in \R^d$.
Recall  the definition of optimally reflecting strategy  $\xi^\lambda$ from Theorem \ref{theorem optimal control}, and define an associated stopping time
$$
\hat \tau ^\lambda (x,1) := \inf \{ t \, |\, \xi^\lambda_t \geq  1- e^{-1} \}.
$$
The results in the previous subsection now allows to obtain the following characterization, revealing that $\hat \tau ^\lambda (x,1)$ is actually independent of $\lambda$.
\begin{theorem}
    \label{theorem optimal stopping boundary representation}
   Under Assumption \ref{assumption general}, for any $\lambda \geq 0$, we have $
\hat \tau ^\lambda (x,1) = \tau ^0(x).
$
\end{theorem}
\begin{proof}
Given the characterization of $\xi^\lambda$ of Theorem \ref{theorem optimal control} and the definition of $\hat \tau ^\lambda (x,1)$, we have
$$
\hat \tau ^\lambda (x,1) =  \inf \{ t \, |\, X^x_t \in \mathcal S _\lambda (e^{-1}) \}.
$$
By Lemma \ref{lemma monotonicity boundary}, we have $\mathcal S _\lambda (e^{-1}) = \mathcal S _0 (e^{-1})$, so that
$$
\hat \tau ^\lambda (x,1) 
=  \inf \{ t \, |\, X^x_t \in \mathcal S _\lambda (e^{-1}) \}
= \inf \{ t \, |\, X^x_t \in \mathcal S _0 (e^{-1}) \}
= \tau^0(x),
$$
which completes the proof.
\end{proof}
\begin{remark}
From an algorithmic point of view, such a representation has a clear advantage: 
If we were able to learn $\xi^\lambda$ correctly, we would be able to find the original optimal stopping time $\tau^0(x)$ simply by the expression $\tau^0(x)= \inf \{ t \, |\, \xi^\lambda_t \geq 1-e^{-1} \}$.
We also notice that, in such an expression, the dependence on the threshold $e^{-1}$  is a direct consequence of our choice of the entropy function $y \mapsto -y \log y$.
\end{remark}

We next introduce the following family of reflecting strategies
\begin{equation*}
    \hat \xi ^\lambda_t := \xi^\lambda_t \land (1-e^{-1}) + (1-e^{-1}) \mathds 1 _{[\hat \tau ^0(x), \infty)} (t).
\end{equation*} 
In light of the Theorem \ref{theorem optimal stopping boundary representation}, we have
$\hat \xi ^\lambda_t = \xi^\lambda_t \land (1-e^{-1}) + (1-e^{-1}) \mathds 1 _{[\hat \tau ^\lambda(x,1), \infty)} (t)$. This expression characterizes $\hat \xi^\lambda$ as a reflecting process at the boundary $g_\lambda$ until the threshold $e^{-1}$, with a jump to $0$ at the first hitting time of the threshold. 

For the family $\hat \xi ^\lambda$, we can prove a general convergence result, which improves our previous Proposition \ref{proposition vanishing entropy Control}.
\begin{theorem}
    \label{theorem convergence of modified strategies}
    Under Assumption \ref{assumption general}, for any $x \in \R ^n$, we have $\hat \xi^\lambda_t  \to \xi ^0_t := \mathds 1 _{[\tau ^0(x), \infty)} (t) $, $\P$-a.s., as $\lambda \to 0$.
\end{theorem}
\begin{proof}
    By the reflecting property of the strategy $\xi^\lambda$, the first time at which $\hat \xi ^\lambda$ increases is $\tau ^\lambda (x,1)$, so that we have
    $$
   \xi ^0_t := \mathds 1 _{[\tau ^0(x), \infty)} (t)  \leq \hat \xi^\lambda_t \leq \ \mathds 1 _{[\tau ^\lambda(x,1), \infty)} (t). 
    $$
    Therefore, we conclude the proof by using Lemma \ref{lemma limits stopping times} in order to take limits as $\lambda \to 0$ in the latter inequalities.
\end{proof}

%\subsection{Boundary and strategy convergence}
\begin{comment}
\begin{theorem}
    \label{theorem boundary convergence}
    When $G \in C^3$,  $\hat{ \pi} :=\pi +(\mathcal L _x - \rho)G$ is $C^1$ and $\hat {\pi} _{x_n} >0$, then the boundary converges as $\lambda \to 0$. 
\end{theorem}

\begin{theorem}
\label{theorem strategy convergence improved}
    In the assumptions of Theorem \ref{theorem boundary convergence}, for any $x \in \R ^n$, we have $ \xi^\lambda_t  \to \xi ^0_t := \mathds 1 _{[\tau ^0(x), \infty)} (t) $, $\P$-a.s., as $\lambda \to 0$.
\end{theorem}
\begin{proof}
    \jo{Need for Lipschitianity of the boundary of the OS problem \eqref{eq conncection OS}}
\end{proof}
\end{comment}

\subsection{Preliminary estimates}
Before diving into the proof of Theorem \ref{theorem value function}, we discuss some estimates that will be used several times in the sequel.

First, for a generic $q \geq 1$ and $ \hat c _0(q),  \hat c _1(q),  \hat c _2(q)$ as in \eqref{eq definition constants}, we have the estimates
\begin{align}
\label{eq estimate SDE growth} 
\mathbb{E}[ |  X ^x_t |^{q} ]  & \leq C (1+|x|^q) e^{  \hat c _0(q) t}, \quad q\geq 1,   \\ 
\label{eq estimate SDE Lip} \mathbb{E}[ |  X^{\bar x}_t - X ^x_t |^{q} ] & \leq |\bar{x} - x|^q e^{  \hat c _1(q) t} , \quad q\geq 2, \\ 
 \label{eq estimate SDE semiconv} 
 \mathbb{E} \big[  \big|  \delta X^{\bar x}_t + (1-\delta) X^x_t - X ^{\delta \bar x + (1-\delta) x}_t \big|^q \big] & \leq C \delta (1-\delta) |\bar{x} - x|^{2q} e^{  \hat c _2(q) t}, \quad q \geq 2,
\end{align}
for any $x,\bar x \in \mathbb R ^n, \ \delta \in [0,1]$.
Even if these estimates are obtained through standard stochastic calculus techniques, it was difficult for us to find a reference in which the exact orders of growth $\hat c_0, \hat c_1, \hat c_2$ appear (importantly, these constants are those appearing in our Condition \ref{condition rho} in Assumption \ref{assumption general}). 
For this reason, we provide a proof in Appendix \ref{appendix proof of estimates}.

When $\rho$ is large enough (as in Condition \ref{condition rho} of Assumption \ref{assumption general}), 
the estimates above together with the Burkholder-Davis-Gundy inequality imply that
\begin{align}
\label{eq estimate SDE growth SUP} 
\mathbb{E}\bigg[ \sup_{t \leq \tau}  e^{-\rho t} |  X ^x_t |^{q} \bigg]  & \leq C (1+|x|^q), \quad q = 2(p-1),p,2(p-2),  \\ 
\label{eq estimate SDE Lip SUP} 
\mathbb{E} \bigg[ \sup_{t \leq \tau}  e^{-\rho t} |  X^{\bar x}_t - X ^x_t |^{q} \bigg] & \leq C |\bar{x} - x|^q , \quad q= 2, 4  \\ 
 \label{eq estimate SDE semiconv SUP} 
 \mathbb{E} \bigg[ \sup_{t \leq \tau}  e^{-\rho t} \big|  \delta X^{\bar x}_t + (1-\delta) X^x_t - X ^{\delta \bar x + (1-\delta) x}_t \big|^2 \bigg] & \leq C \delta (1-\delta) |\bar{x} - x|^4,
\end{align}
for any $x,\bar x \in \mathbb R ^n, \ \delta \in [0,1]$ and $\tau \in \mathcal T$. 
Again, we provide a proof in Appendix \ref{appendix proof of estimates}.

Finally, using the growth  of $\pi$ and $G$ in Condition \ref{condition costs} of Assumption \ref{assumption general}, as well as the estimates \eqref{eq estimate SDE growth} and \eqref{eq estimate SDE growth SUP}, we find
\begin{equation}\label{eq estimate V growth}
\begin{aligned}
V^\lambda (x,y) & \leq
\sup_{\xi \in \mathcal A (y)} 
 \mathbb E \bigg[ \int_0^{\infty} e^{-\rho t} |\Pi(X^x_t,Y^{y,\xi}_{t-}) |   \dd t      + \int_0^{\infty} e^{-\rho t} |G (X^x_t)| \dd \xi_t \bigg] \\    
& \leq \sup_{\xi \in \mathcal A (y)} 
 \mathbb E \bigg[  \int_0^{\infty} e^{-\rho t} (1+|X^x_t|^p)   \dd t     + \lim_{T\to \infty} \sup_{t \leq T } e^{-\rho t} (1+|X^x_t|^p) \xi_{ T}  \bigg] \\
 &\leq C (1+|x|^p),
\end{aligned}
\end{equation}
where we have used the fact that $\xi_t \leq y \leq 1$ for any $\xi \in \mathcal A (y)$.

\subsection{Proof of Theorem \ref{theorem value function}: solution to the HJB}\label{subsection proof V solves HJB}
The concavity of $V^\lambda$ follows by straightforward arguments (see Remark \ref{remark existence of optimal control} for a reference), while its growth rate was shown in \eqref{eq estimate V growth}.
Thus, we will only show in the sequel that $V^\lambda$ has the claimed regularity and it solves the HJB equation
in the almost everywhere sense.
The argument is based on a penalization scheme (see e.g.\ \cite{williams.chow.menaldi.1994regularity} among others) and on a connection with optimal stopping (see e.g.\ \cite{Karatzas&Shreve84}).

The proof is divided into five steps, where we will often use the notation $\Pi(x,y) := y \pi (x) -\lambda y \log y$.
\smallbreak\noindent
\emph{Step 1 (Penalization).} 
Let $\beta : \mathbb{R} \rightarrow[0, \infty)$ be a $C^{\infty}$ nondecreasing convex function such that
$$
\beta(z)= \begin{cases}0, & \text { if } z \leq 0, \\ {2z-1}, & \text { if } z \geqslant 1. \end{cases}
$$
For any $\varepsilon>0$, consider the control set $U _{\varepsilon}:= \big\{ (a^1, a^2) \in \mathbb R ^2  \, | \, sa^1 - a^2 \leq \frac{\beta(s)}{\varepsilon} \  \forall s \in \mathbb R $ and $ 0 \leq a^2 \leq \frac{1}{\varepsilon} \big\}$
and the functional
$$
\begin{aligned}
J^{\lambda, \varepsilon} (x,y;(u,\eta)) &:= \mathbb{E} \left[ \int_{0}^{\infty} e^{- \rho t}\left( \Pi\left( X^x_{t}, Y^{y,u}_{t}\right)-G(X^x_{t}) u_{t} - \eta_t \right) \dd t\right], \quad (u,\eta) \in \mathcal U _\varepsilon (y) \\ 
 \text{subject to}\quad \dd X^x_{t}&=b(X^x_{t}) \dd t+\sigma(X^x_{t}) \dd W_{t}, \quad X^x_0=x,  \\ 
\dd Y^{y,u}_t &= -u_t \dd t, \quad Y^{y,u}_0 =y, 
\end{aligned}
$$
where 
$
\text{$\mathcal U _{\varepsilon} (y):= \{ (u,\eta): \Omega \times[0, \infty) \rightarrow U _{\varepsilon}  \, | \, Y^{y,u}_{t} \geqslant 0 \ \forall t \geqslant 0, \mathbb P $-a.s. $\}$}.
$
For any $\varepsilon>0$, define the optimization problem
\begin{equation}\label{eq control problem penalized}
V^{\lambda, \varepsilon}(x, y) : = \sup _{(u,\eta) \in \mathcal U _{\varepsilon}(y)} J^{\lambda, \varepsilon}(x,y;(u,\eta)).
\end{equation}
Similarly to \eqref{eq estimate V growth}, we can find $C$ such that, for $y$ small enough, one has 
\begin{equation}
    \label{eq continuity of V lambda epsilon in 0}
    \begin{aligned}
    |V&^{\lambda, \varepsilon}(x,y)| \\
    &\leq \sup _{(u,\eta) \in \mathcal U _{\varepsilon}(y)}  \mathbb{E} \left[ \int_{0}^{\infty} e^{- \rho t}\left| \Pi\left( X^x_{t}, Y^{y,u}_{t}\right)-G(X^x_{t}) u_{t} - \eta_t \right| \dd t\right] \\
    &\leq  \sup _{(u,\eta) \in \mathcal U _{\varepsilon}(y)} \mathbb E \bigg[ \int_0^{\infty} e^{-\rho t} y(|\pi(X^x_t) |-\lambda  \log y) \dd t +  \lim_{T \to \infty}  \sup_{t \leq T} \big(e^{-\rho t}  |G(X_t^x)|\big) \int_0^T u _t \dd t  \bigg] \\
    & \leq C y \bigg( 
 \mathbb E \bigg[  \int_0^{\infty} e^{-\rho t} (1+|X^x_t|^p - \lambda \log y)   \dd t     + \lim_{T\to \infty} \sup_{t \leq T } e^{-\rho t} (1+|X^x_t|^p)  \bigg] \bigg)  \\
 & \leq C y (1 + |x|^p -\lambda \log y) \to 0, \quad\text{as $y\to 0$},
    \end{aligned}
\end{equation}
which implies the continuity of $V^{\lambda, \varepsilon}$ for $y=0$.

Standard arguments (see e.g. Chapter IV.10 in \cite{fleming.soner2006}) give that $V^{\lambda, \varepsilon}$ is a  solution  to the HJB equation in the a.e.\ sense; that is, $V^{\lambda, \varepsilon} \in C (\R^n \times [0,1]) \cap W_{loc}^{2, 2} \left(\mathbb{R}^{n} \times (0,1) \right)$ and
\begin{equation}\label{eq HJB penalized}
\rho V^{\lambda, \varepsilon}(x, y)-\mathcal{L}_{x} V^{\lambda, \varepsilon}(x, y)-\Pi(x, y)-\frac 1 {\varepsilon} \beta\left(G(x)-V_{y}^{\lambda, \varepsilon}(x, y)\right)=0,  \quad \text{a.e.\ in $\mathbb R ^n \times (0,1),$}
\end{equation}
with boundary condition $V^{\lambda, \varepsilon}(x,0)=0$.

Moreover, following the rationale of  the proof of Theorem 3.2 in \cite{williams.chow.menaldi.1994regularity}, one can show that
$V^\lambda (x, y)=\lim _{\varepsilon \rightarrow 0} V^{\lambda, \varepsilon}(x,y)$. 
%For $0<\bar{\varepsilon} \leqslant \varepsilon$, we have $\mathcal U_{\varepsilon} \subset \mathcal U_{\bar{\varepsilon}}$ and by \eqref{eq beta monotone} we find $B^{\bar{\varepsilon}}(x, u) \leq B^{\lambda, \varepsilon}(x, u)$ for any $(x, u) \in \mathbb{R} \times\left[0, \frac{1}{\varepsilon}\right]$, which in turn implies
%$$\begin{aligned}V^{\bar{\varepsilon}}(x, y)&= \sup _{(u,\eta) \in \mathcal U _{\bar{\varepsilon}} } J^{\bar{\varepsilon}} (x,y;(u,\eta))& \geqslant \sup _{(u,\eta) \in \mathcal U _{\varepsilon}} J^{\lambda, \varepsilon}(x,y;(u,\eta))=V^{\lambda, \varepsilon}(x, y).\end{aligned}$$

\smallbreak\noindent
\emph{Step 2 (Uniform Lipschitz estimates for $V^{\lambda, \varepsilon}$ in $x$).}
Take $x, \bar x \in \mathbb R^n$, $y \in (0,1]$, $\zeta >0$ and   $(u,\eta) \in \mathcal U _\varepsilon (y)$  such that $V^{\lambda, \varepsilon}(\bar x, y)  - \zeta 
\leq J^{\lambda, \varepsilon}(\bar x, y;u,\eta )$. 
Since $(u,\eta)$ is suboptimal for $(x,y)$, we have
\begin{equation}\label{eq Lipschit V x first}
\begin{aligned}
V^{\lambda, \varepsilon}(\bar x, y) &- V^{\lambda, \varepsilon}(x,y) - \zeta 
\leq J^{\lambda, \varepsilon}(\bar x, y;u,\eta ) - J^{\lambda, \varepsilon}(x,y;u,\eta) \\
 & = \mathbb E  \bigg[\int_0^\infty e ^{-\rho t} \big( Y^{y,u}_t( \pi (X_t^{\bar x}) - \pi (X_t^x)) + u_t ( G(X_t^{\bar x}) - G(X_t^x)) \big) \dd t \bigg].
\end{aligned}
\end{equation}
In order to continue this estimate, we notice that by monotone convergence theorem, we have
\begin{equation}\label{eq monotone for proof}
 \mathbb E  \bigg[ \int_0^\infty e ^{-\rho t}  u_t ( G(X_t^{\bar x}) - G(X_t^x)) \dd t \bigg] \leq \lim_{T \to \infty}
 \mathbb E  \bigg[ \int_0^T e ^{-\rho t}  u_t | G(X_t^{\bar x}) - G(X_t^x)| \dd t \bigg]. 
\end{equation}
Furthermore,  since $(u,\eta) \in \mathcal U _\varepsilon (y)$, we have $\int_0^t u_s \dd s \leq y \leq 1$. 
Thus, using Assumption \ref{assumption general} and Hölder inequality,  we can estimate the right-hand side of \eqref{eq monotone for proof} to obtain
$$
\begin{aligned}
\mathbb E  \bigg[ \int_0^T & e ^{-\rho t}  u_t | G(X_t^{\bar x}) - G(X_t^x)| \dd t \bigg] \\
\leq & \mathbb E  \bigg[ \sup_{t \leq T} \Big( e ^{-\rho t}  | G(X_t^{\bar x}) - G(X_t^x)| \Big)  \int_0^T u_t  \dd t \bigg] \\
 \leq &  C \mathbb{E} \bigg[ \sup_{t \leq T} \Big( e^{-\rho t} 
\big(1+ |  X ^{\bar x}_t |^{p-1} + | X_t^{x}  |^{p-1} \big) \big|  X ^{\bar x}_t- X^{x }_t \big| \Big) \bigg] \\
 \leq &  C   \Big( \mathbb E \Big[ \sup_{t \leq T} \Big( e ^{-\rho t} \big( 1 +   | X ^{\bar x }_t|^{2(p-1)}  +  | X ^{x }_t|^{2(p-1)} \big) \Big) \Big] \Big)^{\frac12} \Big( \mathbb E \Big[ \sup_{t \leq T} \big( e ^{-\rho t} |  X ^{\bar x}_t -  X ^{x}_t |^2 \big) \Big]  \Big)^{\frac12}.
\end{aligned}
$$
Finally, using the estimates  \eqref{eq estimate SDE growth SUP} and \eqref{eq estimate SDE Lip SUP}, we get
$$
\begin{aligned}
\mathbb E  \bigg[ \int_0^T & e ^{-\rho t}  u_t | G(X_t^{\bar x}) - G(X_t^x)| \dd t \bigg] \leq C (1 + |x|^{p-1} + |\bar x|^{p-1}) |\bar x - x|,
\end{aligned}
$$
which, thanks to \eqref{eq monotone for proof}, allows to  conclude that
\begin{equation}\label{eq estimate u G lip in x}
    \mathbb E  \bigg[ \int_0^\infty  e ^{-\rho t}  u_t ( G(X_t^{\bar x}) - G(X_t^x)) \dd t \bigg] \leq C  (1 + |x|^{p-1} + |\bar x|^{p-1})  |\bar x - x|.
\end{equation}

We next estimate the term in \eqref{eq Lipschit V x first}  which involves $\pi$. 
Since $Y^{y,u}_t \leq y \leq 1$,  using Assumption \ref{assumption general} and Hölder inequality  we find
$$
\begin{aligned}
 \mathbb E  \bigg[\int_0^\infty & e ^{-\rho t}  Y^{y,u}_t( \pi (X_t^{\bar x}) - \pi (X_t^x))  \dd t \bigg] \\
 \leq &  C \mathbb{E} \bigg[ \int_0^T e^{-\rho t} 
\big(1+ |  X ^{\bar x}_t |^{p-1} + | X_t^{x}  |^{p-1} \big) \big|  X ^{\bar x}_t- X^{x }_t \big| \dd t \bigg] \\
 \leq &    C   \int_0^T e ^{-\rho t} \Big( 1 + \big( \mathbb E \big[ | X ^{\bar x }_t|^{2(p-1)} \big] \big)^{\frac12} + \big( \mathbb E \big[ | X ^{x }_t|^{2(p-1)} \big] \big)^{\frac12}  \Big) \Big( \mathbb E  \big[ |  X ^{\bar x}_t -  X ^{x}_t |^2 \big]  \Big)^{\frac12}  \dd t.
\end{aligned}
$$
Finally, using the estimates in \eqref{eq estimate SDE growth} and \eqref{eq estimate SDE Lip}, we get
$$
\begin{aligned}
\mathbb E  \bigg[\int_0^\infty & e ^{-\rho t}  Y^{y,u}_t( \pi (X_t^{\bar x}) - \pi (X_t^x))  \dd t \bigg] \\
 \leq &   C (1 + |x|^{p-1} + |\bar x|^{p-1}) |\bar x - x| 
 \int_0^T e ^{ \big(-\rho  + \frac{\hat c _0 ( 2(p-1))}{2} + \frac{\hat c _1 (2)}{2} \big) t } \dd t.
\end{aligned}
$$
Hence, thanks to Condition \ref{condition rho} in Assumption \ref{assumption general}, we conclude that
\begin{equation}
    \label{eq estimate pi Lip in x}
    \mathbb E  \bigg[\int_0^\infty e ^{-\rho t}  Y^{y,u}_t( \pi (X_t^{\bar x}) - \pi (X_t^x))  \dd t \bigg] \leq C  (1 + |x|^{p-1} + |\bar x|^{p-1})  |\bar x - x|. 
\end{equation}

Plugging \eqref{eq estimate u G lip in x} and \eqref{eq estimate pi Lip in x} into \eqref{eq Lipschit V x first}, we obtain
$$
\begin{aligned}
V^{\lambda, \varepsilon}(\bar x, y) - V^{\lambda, \varepsilon}(x,y) - \zeta  
\leq  C  (1 + |x|^{p-1} + |\bar x|^{p-1})  |\bar x - x|,
\end{aligned}
$$
which, by the arbitrariness of $\zeta$,  gives the local Lipschitz property of $V^\varepsilon$ in $x$.

\smallbreak\noindent
\emph{Step 3 (Uniform Lipschitz estimates for $V^{\lambda, \varepsilon}$ in $y$).}
%We next show the Lipschitzianity in $y$.
Take $x \in \mathbb R^n$, $y, \bar y \in (0,1]$, $\zeta >0$ and  $(\bar u,\bar \eta) \in \mathcal U _\varepsilon (\bar y)$ such that $V^{\lambda, \varepsilon}( x, \bar y)  - \zeta 
\leq J^{\lambda, \varepsilon}( x, \bar y;\bar u, \bar \eta )$.
We will distinguish the two cases $\bar y \leq y$ and $\bar y >y$.

If $\bar y \leq y$, then we have $Y^{y,\bar u}_t \geq0$, so that $(\bar u, \bar \eta) \in \mathcal U _\varepsilon (y)$.
By suboptimality of $(\bar u, \bar \eta)$  for the initial condition $(x,y)$, we have
\begin{equation}\label{eq Lipschitz V y first}
\begin{aligned}
V^{\lambda, \varepsilon}( x, \bar y) &- V^{\lambda, \varepsilon}(x,y) - \zeta 
\leq J^{\lambda, \varepsilon}( x, \bar y;\bar u, \bar \eta ) - J^{\lambda, \varepsilon}(x,y; \bar u ,\bar \eta) \\
 & = \mathbb E  \bigg[\int_0^\infty e ^{-\rho t} \big( \pi (X_t^{ x}) (\bar y - y) - \lambda (Y^{\bar y,\bar u}_t \log Y^{\bar y,\bar u}_t - Y^{y, \bar u}_t \log Y^{ y, \bar u}_t) \big) \dd t \bigg] \\
 & = C (1 + |x|^p) |\bar y - y| + \mathbb E  \bigg[\int_0^\infty e ^{-\rho t} \big( - \lambda (Y^{\bar y,\bar u}_t \log Y^{\bar y,\bar u}_t - Y^{y,\bar u}_t \log Y^{ y,\bar u}_t) \big) \dd t \bigg],
\end{aligned}
\end{equation}
where we have used Condition \ref{condition rho} in Assumption \ref{assumption general} together with \eqref{eq estimate SDE growth}.
In order to continue the latter estimate, let $\hat z$ be the maximum point of the function $z \mapsto - \lambda z \log z$ (i.e., $\hat z = e^{-1}$), and notice that $z \mapsto - \lambda z \log z$ is increasing in $[0,\hat z]$ and Lipschitz in $[\frac{\hat z}2, 1]$, with Lipschitz constant smaller than  $C_\lambda := 1 + |\log (\hat z /2)|$. 
Suppose now that $y - \bar y \leq \frac{\hat z}2$. 
If $Y^{\bar y,\bar u}_t \leq \frac{\hat z}2 $, then $ Y^{y,\bar u}_t \leq \hat z$, and since $Y^{\bar y,\bar u}_t \leq Y^{y,\bar u}_t$, we obtain $ - \lambda (Y^{\bar y,\bar u}_t \log Y^{\bar y,\bar u}_t - Y^{y,\bar u}_t \log Y^{ y,\bar u}_t) \leq 0$ by monotonicity.
On the other hand, if  $Y^{\bar y,\bar u}_t \geq \frac{\hat z}2 $, then $ Y^{y,\bar u}_t \geq \frac{\hat z}2$, and by Lipschitzianity we find
$ - \lambda (Y^{\bar y,\bar u}_t \log Y^{\bar y,\bar u}_t - Y^{y,\bar u}_t \log Y^{ y,\bar u}_t) \leq  C_\lambda (Y^{ y,\bar u}_t - Y^{\bar y,\bar u}_t) = C_\lambda (y-\bar y)$. 
Going back to \eqref{eq Lipschitz V y first}, this argument implies that
$$
\begin{aligned}
 0 \leq y - \bar y \leq \frac{\hat z}2 \implies V^{\lambda, \varepsilon}( x, \bar y) &- V^{\lambda, \varepsilon}(x,y) - \zeta 
\leq C (1 + |x|^p + |\log(\hat z /2)|) |\bar y - y|, 
\end{aligned}
$$
which give the local Lipschitz property of $V^\varepsilon$ in $y$ in the case $\bar y \leq y$.

On the other hand, if $\bar y \geq y$, then the control $( \bar u, \bar \eta)$ is not necessarily admissible for $y$, as $Y_t^{y,\bar u}$ could become smaller than 0. 
Define then $u_t: = \bar u_t \frac{y}{\bar y}$ and $\eta_t: = \bar \eta_t \frac{y}{\bar y}$. 
Notice that, if $s\bar u_t - \bar \eta _t \leq 0$, then $s u_t - \eta_t = \frac{y}{\bar y} ( s\bar u_t - \bar \eta_t) \leq 0 \leq \frac 1 \varepsilon \beta(s)$, while  $s u_t - \eta_t = \frac{y}{\bar y} ( s\bar u_t - \bar \eta_t) \leq s\bar u_t - \bar \eta_t \leq \frac 1 \varepsilon \beta(s)$ if $s\bar u_t - \bar \eta_t > 0$.
Thus, the process $(u,\eta) $ takes values in the set $ U _\varepsilon$. 
Moreover, since 
\begin{equation}\label{eq y bar y}
    Y^{y,u}_t = y- \int_0^t u_s \dd s = \frac{y}{\bar y} \Big( \bar y- \int_0^t \bar u_s \dd s \Big) = \frac{y}{\bar y}Y^{\bar y, \bar u} \geq 0,
\end{equation}
it follows that $(u,\eta) \in \mathcal U _\varepsilon (y)$.
Hence, by the suboptimality of $(u,\eta)$ for the initial condition $(x,y)$, using \eqref{eq y bar y}, we find
\begin{equation}\label{eq proof Lip y GIorgio}
\begin{aligned}
V^{\lambda, \varepsilon}&( x, \bar y) - V^{\lambda, \varepsilon}(x,y) - \zeta \\
& \leq J^{\lambda, \varepsilon}( x, \bar y;\bar u, \bar \eta ) - J^{\lambda, \varepsilon}(x,y;  u , \eta) \\
&  = \mathbb E  \bigg[\int_0^\infty e ^{-\rho t} \Big( - \lambda \int_0^1 \Big(1+ \log \big(  r  Y^{\bar y,\bar u}_t   + (1-r)Y^{y,  u}_t \big) \Big)  ( Y^{\bar y,\bar u}_t  - Y^{y, \bar u}_t ) \dd r \\ 
& \quad \quad \quad \quad \quad \quad \quad \quad  +\pi (X_t^{ x}) (\bar y - y) \frac{Y_t^{\bar y, \bar u}}{\bar y} + G (X_t^x) \frac{\bar y - y}{\bar y} \bar u _t - \frac{\bar y - y}{\bar y} \bar \eta_t \Big) \dd t \bigg]\\
 & \leq \mathbb E  \bigg[\int_0^\infty e ^{-\rho t} \Big(
 - \lambda \int_0^1 \Big(1+ \log ( r \bar y + (1-r) y ) + \log \big( \frac{Y^{y, \bar u}_t}{\bar y} \big) \Big) \frac{Y^{y, \bar u}_t}{\bar y} {(\bar y -y) }  \Big)  \dd r \\ 
& \quad \quad \quad \quad \quad \quad \quad \quad  + |\pi (X_t^{ x})| (\bar y - y)  +  G (X_t^x) \frac{\bar y - y}{\bar y} \bar u _t  \Big) \dd t \bigg]\\ 
& \leq (\bar y -y) \mathbb E  \bigg[\int_0^\infty e ^{-\rho t} \Big(
 \lambda  \Big( 1 + |\log y|  - \log \big( \frac{Y^{y, \bar u}_t}{\bar y} \big)  \frac{Y^{y, \bar u}_t}{\bar y} \Big) + |\pi (X_t^{ x})| + \frac{G (X_t^x)}{\bar y} \bar u _t  \Big) \dd t \bigg] \\
 & \leq C (\bar y -y) \mathbb E  \bigg[\int_0^\infty e ^{-\rho t} \Big(
 \lambda  (  1 + |\log y|  ) + |\pi (X_t^{ x})| + \frac{G (X_t^x)}{\bar y} \bar u _t  \Big) \dd t \bigg]. 
\end{aligned}
\end{equation}
Similarly to the arguments that lead to \eqref{eq estimate u G lip in x}, we have
$$
\begin{aligned}
 \mathbb E  \bigg[ \int_0^\infty e ^{-\rho t}  \bar u_t  G(X_t^x) \dd t \bigg] \leq& \lim_{T \to \infty}
 \mathbb E  \bigg[ \int_0^T e ^{-\rho t}  \bar u _t  | G(X_t^x) | \dd t \bigg]\\ 
 \leq &\lim_{T \to \infty}   \mathbb E  \bigg[ \sup_{t \leq T} \big(e^{-\rho t}  |G(X_t^x)|\big) \int_0^T \bar u _t \dd t  \bigg]\\
 \leq &\lim_{T \to \infty}   C \mathbb E  \bigg[\sup_{t \leq T} \big(e^{-\rho t} ( 1 +  |X_t^{ x}|^{p} )\big)   \bigg]  \\
 \leq & C  (1 + |x|^{p}), 
\end{aligned}
$$
where we have used the estimate in \eqref{eq estimate SDE growth SUP}. 
Plugging the last inequality into \eqref{eq proof Lip y GIorgio}, we obtain
$$
\begin{aligned}
V^{\lambda, \varepsilon}( x, \bar y) - V^{\lambda, \varepsilon}(x,y) - \zeta 
  &\leq C |\bar y -y| \bigg( \frac{1+|x|^p}{\bar y}  + |\log y| + \mathbb E  \bigg[\int_0^\infty e ^{-\rho t} \big( 1   +  |X_t^{ x}|^{p} \big) \dd t   \bigg] \bigg) \\
   &\leq C |\bar y -y| \Big(  \frac{1+|x|^p}{\bar y}  + |\log y| \Big) \int_0^\infty e ^{ \big(-\rho +\hat c _0 (p) \big) t }   \dd t       \\
 & \leq C |\bar y -y| \Big( \frac{1+|x|^p}{\bar y}  + |\log y| \Big),
\end{aligned}
$$
where we have used the condition $\rho > \hat c _0(p)$ given in Assumption \ref{assumption general}.
This shows the local Lipschitz property of $V^\varepsilon$ in $y$ for $\bar y \geq y$. 

\smallbreak\noindent
\emph{Step 4 (Uniform  estimates for $ D^2_x V^{\lambda, \varepsilon}$).}
We first show the following semiconvexity estimate: there exists a constant $C$ such that
\begin{equation}
\label{eq semiconvexity}
V^{\lambda, \varepsilon}( x^\delta, y) - \delta V^{\lambda, \varepsilon}(\bar x, y) - (1-\delta)V^{\lambda, \varepsilon}( x,y) 
\leq  C \delta (1-\delta) ( 1 +|{x}|^{p-1}+ |\bar{x}|^{p-1} )  |\bar{x} - x|^2  , 
\end{equation} 
for each $\delta \in [0,1]$,  $\bar{x}, x \in \mathbb R^n$, $ y \in [0,1]$ and $\varepsilon >0$, and where
$x^\delta:=\delta \bar{x} + (1-\delta)x$.

Take an arbitrary $\zeta>0$ and $(u, \eta) \in \mathcal{U}_\varepsilon(y)$  such that $V^{\lambda, \varepsilon} ( x^\delta, y) -\zeta
 \leq J^{\lambda, \varepsilon}( x^\delta, y; u, \eta)$. 
Notice that such a control is admissible also for the initial conditions $(\bar x, y)$ and $(x,y)$.
Since $(u, \eta)$ is not necessarily optimal for $(x,y)$ or $(\bar{x}, \bar y)$, we have
\begin{equation}\label{eq proof GIorgio OX}
\begin{aligned}
V^{\lambda, \varepsilon}& ( x^\delta, y) - \delta V^{\lambda, \varepsilon}(\bar x,  y) - (1-\delta)V^{\lambda, \varepsilon}( x,y) 
-\zeta \\
& \leq J^{\lambda, \varepsilon}( x^\delta, y; u, \eta) - \delta J^{\lambda, \varepsilon}(\bar x, y;u, \eta) - (1-\delta)J^{\lambda, \varepsilon}( x,y;u, \eta) \\
& \leq \mathbb{E} \bigg[ \int_0^\infty e^{-\rho t} \Big( Y_t^{y,u} \big( \pi (X_t^\delta)  - \delta \pi (X_t^{\bar x}) - (1-\delta) \pi (X_t^{x})  \big) \\
& \quad \quad \quad \quad \quad \quad \quad  \quad \quad 
+ u_t \big( G (X_t^\delta) - \delta G (X_t^{\bar x}) - (1-\delta) G (X_t^{x})  \Big)  \dd t \bigg].
\end{aligned}
\end{equation}
Next, set $\hat X_t := \delta  X^{\bar x}_t  + (1-\delta) X^x_t$, $X^\delta_t:= X^{x^\delta}_t$ and define, for a generic function $\varphi: \mathbb R ^n \to \mathbb R$, the transformations 
$$
\begin{aligned}
    \Delta [\varphi] _t & := |\varphi (X_t^\delta)  - \varphi(\hat X_t)|, \\
    \Gamma [\varphi] _t & :=  |\varphi(\hat X_t) - \delta \varphi (X_t^{\bar x}) - (1-\delta) \varphi (X_t^{x})|.
\end{aligned}
$$
Then, applying the monotone convergence theorem, we rewrite \eqref{eq proof GIorgio OX} as 
\begin{equation}\label{estimate A + B}
\begin{aligned} 
V^{\lambda, \varepsilon}& ( x^\delta, y) - \delta V^{\lambda, \varepsilon}(\bar x,  y) - (1-\delta)V^{\lambda, \varepsilon}( x,y) 
-\zeta \\
& \leq  \lim_{T \to \infty} \mathbb{E} \bigg[ \int_0^T e^{-\rho t} \Big( Y_t^{y,u} \Delta [\pi] _t + u_t \Delta [G] _t   + Y_t^{y,u} \Gamma [\pi] _t + u_t \Gamma [G]_t) \Big) \dd t\bigg].
\end{aligned}
\end{equation}
We next estimate each term on the right-hand side of \eqref{estimate A + B} separately.

First, using Condition \ref{condition costs} in Assumption \ref{assumption general}, Hölder inequality, and the SDE estimates \eqref{eq estimate SDE growth} and \eqref{eq estimate SDE semiconv}, we find
\begin{equation}\label{eq delta pi}
    \begin{aligned}
\mathbb{E} \bigg[&\int_0^T e^{-\rho t}   Y_t^{y,u} \Delta [\pi] _t \dd t \bigg] \\
& \leq C \mathbb{E} \bigg[ \int_0^T e^{-\rho t} 
\big(1+ | \hat X_t |^{p-1} + | X_t^\delta  |^{p-1} \big) \big| \hat X_t- X^\delta_t \big| \dd t \bigg] \\
& \leq  C\bigg(     \int_0^T e ^{-\rho t} \Big( 1 + \big( \mathbb E \big[ |\hat X_t|^{2(p-1)} \big] \big)^{\frac12} + \big( \mathbb E \big[ | X ^\delta_t|^{2(p-1)} \big] \big)^{\frac12}  \Big) \Big( \mathbb E  \big[ | \hat X_t -  X ^\delta_t |^2 \big]  \Big)^{\frac12}  \dd t \\
& \leq C \delta (1-\delta) ( 1 +|{x}|^{p-1}+ |\bar{x}|^{p-1} )  |\bar{x} - x|^2 \int_0^T e^{\big( - \rho + \frac{ \hat c _0  (2(p-1))}{2} + \frac{\hat c _1 (2)}{2} \big) t} \dd t \\
& \leq C \delta (1-\delta) ( 1 +|{x}|^{p-1}+ |\bar{x}|^{p-1} )  |\bar{x} - x|^2, 
    \end{aligned}
\end{equation}
where the last inequality follows by  Condition \ref{condition rho} in Assumption \ref{assumption general}.

Secondly, using Condition \ref{condition costs} in Assumption \ref{assumption general}, Hölder inequality, and the SDE estimates \eqref{eq estimate SDE growth} and \eqref{eq estimate SDE semiconv}, we find
\begin{equation}\label{eq delta G}
    \begin{aligned}
        \mathbb{E} \bigg[&
        \int_0^T e^{-\rho t} u_t \Delta [G] _t \dd t \bigg] \\
    & \leq  \mathbb{E} \bigg[
    \sup_{t \leq T} \big( e^{-\rho t}  \Delta [G] _t \big) \int_0^T u_t \dd t \bigg] \\
&\leq   C \mathbb{E} \bigg[ \sup_{t \leq T} \Big( e^{-\rho t} 
\big(1+ |  \hat X_t |^{p-1} + | X_t^\delta  |^{p-1} \big) \big|  \hat X_t- X_t^\delta \big| \Big) \bigg] \\
 &\leq   C   \Big( \mathbb E \Big[ \sup_{t \leq T} \Big( e ^{-\rho t} \big( 1 +   | \hat X_t|^{2(p-1)}  +  | X_t^\delta |^{2(p-1)} \big) \Big) \Big] \Big)^{\frac12} \Big( \mathbb E \Big[ \sup_{t \leq T} \big( e ^{-\rho t} |  \hat X_t - X_t^\delta |^2 \big) \Big]  \Big)^{\frac12}\\
 & \leq C \delta (1-\delta) ( 1 +|{x}|^{p-1}+ |\bar{x}|^{p-1} )  |\bar{x} - x|.
\end{aligned}
\end{equation} 
where the last inequality follows from the estimates \eqref{eq estimate SDE growth SUP} and \eqref{eq estimate SDE semiconv SUP}.

Thirdly, using Condition \ref{condition costs} in Assumption \ref{assumption general}, Hölder inequality, and the SDE estimates \eqref{eq estimate SDE growth} and \eqref{eq estimate SDE Lip}, we find
\begin{equation}
    \label{eq D pi}
    \begin{aligned}
        \mathbb{E} \bigg[ & \int_0^T e^{-\rho t} Y_t^{y,u} \Gamma [\pi] _t \dd t \bigg] \\
        & \leq C \delta (1-\delta) \mathbb E \bigg[ \int_0^T e^{-\rho t}  \Big( 
 1+     | X^{\bar x} _t |^{p-2}   +  | X^x _t |^{p-2} \Big) |  X^{\bar x} _t - X ^x_t \big|^2 \dd t \bigg] \\
        & \leq C \delta (1-\delta) \int_0^T e^{-\rho t}  \Big( 
 1+ \big( \mathbb{E}[ |X^{\bar x} _t |^{2(p-2)}] \big)^{\frac{1}{2}} + \big( \mathbb{E}[ | X^x _t |^{2(p-2)}] \big)^{\frac{1}{2}} \Big) 
\\
& \qquad \qquad \qquad \qquad \qquad \times \Big( \mathbb{E}[ | X^{\bar x} _t - X^x_t \big|^4 ] \Big)^{\frac12}  dt   \\
  &  \leq C \delta (1-\delta) ( 1 + |{x}|^{p-2} + |\bar{x}|^{p-2}) |\bar{x} - x|^2 \int_0^T e^{\big(-\rho  + \frac{ \hat c _0(2(p-2))}{2} + \frac{\hat c _1 (4)}{2} \big) t} \dd t \\
& \leq C \delta (1-\delta) ( 1 + |{x}|^{p-2} + |\bar{x}|^{p-2})  |\bar{x} - x|^2, 
    \end{aligned}
\end{equation}
where the finiteness of the integrals follows again from Condition \ref{condition rho} in Assumption \ref{assumption general}.

Fourthly, using Condition \ref{condition costs} in Assumption \ref{assumption general}, Hölder inequality, and the SDE estimates \eqref{eq estimate SDE growth} and \eqref{eq estimate SDE Lip}, we find
\begin{equation}
    \label{eq D G}
    \begin{aligned}
        \mathbb{E} \bigg[ & \int_0^T e^{-\rho t} u_t \Gamma [ G] _t   \dd t \bigg] \\
        & \leq  \mathbb{E} \bigg[ \sup_{t \leq T} \big( e^{-\rho t}   \Gamma [ G] _t \big)   \int_0^T  u_t \dd t \bigg] \\
        & \leq C \delta (1-\delta) \mathbb E \bigg[ \sup_{t \leq T} \Big( e^{-\rho t}  \big( 
         1+  | X^{\bar x} _t |^{p-2}   +  | X^x _t |^{p-2} \big) |  X^{\bar x} _t - X ^x_t \big|^2 \Big)  \bigg] \\
        & \leq C \delta (1-\delta) \Big( \mathbb E \Big[ \sup_{t \leq T} \Big( e^{-\rho t}  \big( 
        1+ |X^{\bar x} _t |^{2(p-2)}  + | X^x _t |^{2(p-2)} \big) \Big) \Big] \Big)^{\frac 12} \\
        & \qquad  \qquad  \qquad   \times         \Big( \mathbb E \Big[ \sup_{t \leq T} \Big( e^{-\rho t} | X^{\bar x} _t - X^x_t \big|^4 ] \Big)  \Big] \Big)^{\frac12}\\
        &  \leq C \delta (1-\delta) ( 1 + |{x}|^{p-2} + |\bar{x}|^{p-2}) |\bar{x} - x|^2 \int_0^T e^{\big(-\rho + \frac{\hat c _0 (2(p-2)}{2} + \frac{\hat c _1 (4)}2 \big) t} \dd t \\
& \leq C \delta (1-\delta) ( 1 + |{x}|^{p-2} + |\bar{x}|^{p-2})  |\bar{x} - x|^2, 
    \end{aligned}
\end{equation}
where we have used Condition \ref{condition rho} in Assumption \ref{assumption general}.

Plugging \eqref{eq delta pi}, \eqref{eq delta G}, \eqref{eq D pi} and \eqref{eq D G} back into \eqref{estimate A + B}, we obtain \eqref{eq semiconvexity}.

We will now use \eqref{eq semiconvexity} in the PDE \eqref{eq HJB penalized} in order to obtain an estimate for the derivative $D_x^2 V^{\lambda, \varepsilon}$.
Notice that \eqref{eq semiconvexity} implies that, for any $r>0$ and $B_r(0) := \{ x\in \mathbb R ^n \, |\, |x| < r \}$, one can find a constant $C_r >0 $ such that
\begin{equation}\label{eq semiconvexity derivatives}
\partial^2_{zz} V^{\lambda, \varepsilon}(x,y) \geq - C_r, \quad \text{for any $x \in B_r (0)$, $z \in \R^n$ with $|z|=1$, and $y \in [0,1]$.}
\end{equation}
Choose an orthonormal basis $z^1,...,z^n$ of $\R^n$ such that
$$
\tr (\sigma \sigma^* (x) D_x^2 V^{\lambda, \varepsilon}(x,y) ) = \sum_{i=1}^n c_{\sigma}^i (x) \partial^2_{z^i z^i} V^{\lambda, \varepsilon}(x,y), 
$$
where $c_{\sigma}^1,...,c_{\sigma}^n$ are the eigenvalues of the symmetric, nonnegative definite
matrix $\sigma \sigma^* (x)$.
Since $V^\varepsilon$ solves the PDE in \eqref{eq HJB penalized}, using the latter identity together with \eqref{eq semiconvexity derivatives} and the estimates in the previous steps, unless to take a larger constant $C_r$ (independently from $\varepsilon$) we obtain
\begin{equation*}
\begin{aligned}
    - C_r \leq &  \sum_{i=1}^n c_{\sigma}^i (x) \partial^2_{z^i z^i} V^{\lambda, \varepsilon}(x, y) \\
    =& \rho V^{\lambda, \varepsilon}(x, y)-b(x) D_{x} V^{\lambda, \varepsilon}(x, y)-\Pi(x, y)- \frac 1 {\varepsilon} \beta\left(G(x)-V_{y}^{\lambda, \varepsilon}(x, y)\right) \\
    \leq & \rho V^{\lambda, \varepsilon}(x, y)-b(x) D_{x} V^{\lambda, \varepsilon}(x, y)-\Pi(x, y)  \leq C_r,  \quad \text{for any $(x,y) \in B_r (0) \times (0,1)$}.
\end{aligned}
\end{equation*} 
This in turn implies that
\begin{equation*}
\begin{aligned}
    0\leq \frac 1 {\varepsilon} \beta\left(G(x)-V_{y}^{\lambda, \varepsilon}(x, y)\right)   \leq C_r ,  \quad \text{for any $(x,y) \in B_r (0) \times (0,1)$},
\end{aligned}
\end{equation*} 
so that, by using the uniform ellipticity of $\sigma$ ({\it cf.}\ Condition \ref{codition ellipticity} in Assumption \ref{assumption general}), by Theorem 9.11 p.\ 235 in \cite{gilbarg2001} we conclude that, for any $q>0$,  unless to enlarge again the constant $C_r$ (independently from $\varepsilon$), for any $y \in (0,1)$ we have
\begin{equation}
    \label{eq secon order derivative uniform estimate}
    \int_{B_{r/2} (0)} |D_x^2 V^{\lambda, \varepsilon}(x,y) |^q \dd x \leq C_r \bigg(1 + \int_{B_{r/2} (0)} \Big| \frac 1 {\varepsilon} \beta\left(G(x)-V_{y}^{\lambda, \varepsilon}(x, y)\right) \Big|^q \dd x \bigg) \leq C_r,
\end{equation}
thus giving the desired estimate.

\smallbreak\noindent
\emph{Step 5 (Lipschitzianity of $V^\lambda_y$ via optimal stopping).}
Following the arguments in \cite{Karatzas&Shreve84} (see in particular Theorem 3.4) we can show that, for any $(x, y) \in \mathbb{R}^{n} \times(0,1]$, one has the representation
\begin{equation*}
V^\lambda_y (x, y) =\sup _{\tau \in \mathcal T} \mathbb{E}\left[\int_{0}^{\tau} e ^{-\rho t} \Pi_{y}\left(X^x_{t}, y\right) \dd t+e^{-\rho {\tau}} G\left(X^x_{\tau}\right)\right],  
\end{equation*}
that we will now employ to prove the local Lipschitzianity of $V^\lambda_y$.

We first show the Lipschitzianity in $y$. 
Take $x \in \mathbb{R}^{n}$, $\bar{y}, y \in (0,1]$ with $\bar y \leq y$. 
By the monotonicity of $\Pi_y$ in $y$ we have $0 \leq  V^\lambda_y(x, \bar{y})-V^\lambda_y(x, y) $. 
Moreover, 
for $\zeta>0$ and  $\tau \in \mathcal T$ such that $ V^\lambda_y (x, \bar{y})  - \zeta 
  \leq \mathbb{E}\left[\int_{0}^{\tau} e ^{-\rho t} \Pi_{y}(X_{t}^x, \bar{y}) \dd t + e^{-\rho \tau} G(X^x_\tau ) \right]$, we find
$$
\begin{aligned}
 V^\lambda_y (x, \bar{y})-V^\lambda_y(x, y) - \zeta 
 & \leq \mathbb{E}\left[\int_{0}^{\tau} e ^{-\rho t}\left(\Pi_{y}\left(X_{t}^x, \bar{y}\right)-\Pi_{y}\left(X_{t}^x, y\right)\right) \dd t\right] \\
& = \lambda \mathbb{E}\left[\int_{0}^{\tau} e ^{-\rho t}(\log y - \log \bar{y}) \dd t\right] \\
& \leqslant \lambda \mathbb{E}\left[\int_{0}^{\tau} e ^{-\rho t} \dd t\right] \frac{1}{\bar{y}}(y - \bar y) 
\leqslant \lambda C  \frac{1}{\bar{y}} (y - \bar y) .
\end{aligned}
$$
From the arbitrariness of $\zeta$, we conclude that
$0 \leqslant V^\lambda_y(x, \bar{y})-V^\lambda_y(x, y) \leqslant C \frac{1}{\bar{y}} (y - \bar y)$, showing the Lipschitzianity in $y$. 

We next show the Lipschitzianity in $x$.
For $\bar{x}, x \in \mathbb{R}^{n}$ and $y \in (0,1]$, take $\zeta >0$ and $\tau \in \mathcal T$ such that $V^\lambda_y (\bar x, y)  - \zeta 
  \leq \mathbb{E}\left[\int_{0}^{\tau} e ^{-\rho t} \Pi_{y}(X_{t}^{\bar x}, {y}) \dd t + e^{-\rho \tau} G(X^{\bar x}_\tau ) \right]$. 
We observe that
\begin{equation}\label{eq Vy first}
\begin{aligned}
 V^\lambda_y (\bar x, y) & - V^\lambda_y (x, y)-\zeta \\
 & \leq \mathbb{E}\bigg[\int_{0}^{\tau} e^{-\rho {t}} \left(\Pi_{y}(X_{t}^{\bar{x}}, y )-\Pi_{y} (X_{t}^{x}, y )\right) \dd t + e^{-\rho \tau}\left( G(X_\tau^{\bar{x}})-G (X_\tau^{x})\right) \bigg] \\
 & = \mathbb{E}\bigg[\int_{0}^{\tau} e^{-\rho {t}} \left( \pi(X_{t}^{\bar{x}}  )- \pi (X_{t}^{x}  )\right) \dd t + e^{-\rho \tau}\left( G(X_\tau^{\bar{x}})-G (X_\tau^{x})\right) \bigg].
\end{aligned}
\end{equation}
Thus, by repeating the arguments that lead to \eqref{eq estimate u G lip in x} and \eqref{eq estimate pi Lip in x}, we obtain
$$
 V^\lambda_y (\bar x, y) - V^\lambda_y (x, y)-\zeta \leq C  (1 + |x|^{p-1} + |\bar x|^{p-1})  |\bar x - x|.  
$$
Hence, the local Lipschitzianity in $x$ follows by the arbitrariness of $\zeta$ and exchanging the role of $x$ and $\bar x$.

\smallbreak\noindent
\emph{Step 6 (Regularity of $V^\lambda$ and variational inequality).} 
We begin by noticing that the continuity of $V^\lambda$ in $y=0$ can be shown as in \eqref{eq continuity of V lambda epsilon in 0}.

By the estimates in Steps 2 and 3, for any $r, \delta > 0$ and any ball $B_r (0) := \{ x \in \mathbb R ^n | \,\, |x|<r \}$, we can find a subsequence $(\varepsilon_n)_n$ with $\varepsilon_n \to 0$ as $n \to \infty$ such that
\begin{equation}
   \label{eq limits Vx Vy} 
   \begin{aligned}
   & V^{\lambda,\varepsilon_n} \to V^\lambda  \text{ as $n \to \infty$, uniformly in $B_r(0) \times (\delta, 1) $,} \\
   & (D_x V^{\lambda,\varepsilon_n},V^{\lambda,\varepsilon_n}_y) \to (D_x V^\lambda,V^\lambda_y)  \text{ as $n \to \infty$, weakly in $\mathbb L ^2 (B_r(0) \times (\delta, 1) )$.}
   \end{aligned}
\end{equation}
Moreover, thanks to the estimate in Step 4, for any fixed $y \in (0,1]$ we can select a further subsequence $(\varepsilon^y_n)_n$ such that
\begin{equation} \label{eq limits DV Delta V}
\begin{aligned}
    & D_x V^{\lambda,\varepsilon^y_n} (\cdot,y) \to D_x V^\lambda(\cdot,y) \ \text{ as $n \to \infty$, uniformly in $B_r(0)$,} \\
   & D_x^2 V^{\lambda,\varepsilon^y_n}(\cdot,y) \to D^2_xV^\lambda(\cdot,y) \ \text{ as $n \to \infty$, weakly in $\mathbb L ^2 (B_r(0))$.}
\end{aligned}    
\end{equation}
Therefore, taking limits in \eqref{eq HJB penalized} %and using the continuity of $V^\lambda_y$ from Step 4, 
we obtain that
\begin{equation*}
\text{for any $y \in (0,1)$} \quad (\rho -\mathcal{L}_{x}) {V^\lambda} (\cdot,y)-\Pi(\cdot,y) \geq0 \quad \text{and} \quad V^\lambda_y(\cdot,y) -G \geqslant 0, \quad \text{$\dd x$-a.e.}
\end{equation*}
which yields that
\begin{equation}\label{eq HJB inequality}
    \max \big\{ \left(\mathcal{L}_{x} -\rho \right) V^\lambda(x, y) + \pi (x) y - \lambda y \log y, - V^\lambda_{y}(x, y) + G(x) \big\}\leq 0. 
\end{equation}
To complete the proof, if $(\bar x, \bar y)$ is such that $V^\lambda_y(\bar x,\bar y) -G (\bar x) > 0$, by continuity of $V^\lambda_y$ ({\it cf.}\ Step 5) we can find $r>0$ such that $V^\lambda_y ( x, \bar y) -G (x) > 0$ for any $x \in B_r(0) $. 
Thus, for a.a.\ $x \in B_r(0)$ there exists $\bar n_x$ such that ${V}^{\lambda,\varepsilon^{\bar y}_n}_y ( x, \bar y) -G (x) > 0$ for any $n \geq \bar n _x$, which gives $(\rho -\mathcal{L}_{x}) {V}^{\lambda,\varepsilon^{\bar y}_n} (x,\bar y)-\Pi(x,\bar y) = 0$ for any $n \geq \bar n _x$ so that
$$
\lim_n (\rho -\mathcal{L}_{x}) {V}^{\lambda,\varepsilon^{\bar y}_n} (\cdot,\bar y)-\Pi(\cdot,\bar y) = 0, \quad \text{a.e.\ in $x \in B_r(0)$.}
$$
Since $(\rho -\mathcal{L}_{x}) {V^\lambda} (\cdot,\bar y)-\Pi(\cdot,\bar y)  $ is the weak limit in $\mathbb L ^2 (B_r(\bar x))$ of $\big( (\rho -\mathcal{L}_{x}) {V}^{\lambda,\varepsilon^{\bar y}_n} (\cdot,\bar y)-\Pi(\cdot,\bar y) \big)_n$, the latter limits in turn implies that $(\rho -\mathcal{L}_{x}) {V^\lambda} (\cdot,\bar y)-\Pi(\cdot,\bar y) =0 $ a.e.\ in $ B_r(\bar x)$, thus showing that \eqref{eq HJB inequality} actually holds with equality and completing the proof.

\subsection{Proof of Theorem \ref{theorem value function}: uniqueness for the HJB}\label{subsection uniqueness HJB}
Uniqueness of the solution follows from a verification argument, that we prove in two steps.
\smallbreak\noindent
\emph{Step 1.}
Let $\tilde V$ be a solution of the HJB equation as in the statement of the theorem.
Fix $(x,y) \in \mathbb R ^n \times [0,1]$ and consider a generic $\xi \in \mathcal A (y)$.
For $R>0$, set $\tau_R:=\inf\{t \ge 0 \,| \,  |X^x_t| \ge R\}$. 
Thanks to the regularity of $\tilde V$, we can employ (a generalized) It\^o's formula to the process $\big( e^{-\rho t} \tilde V (X^x_t, Y^{y,\xi}_t) \big)_t$ (see \cite[Chapter 8, Section VIII.4, Theorem 4.1]{fleming2006controlled} and \cite[Theorem 4.2]{angelis2019solvable}),  up to the
stopping time $\tau_R \wedge T$, for some $T > 0$, to obtain
\begin{equation}\label{eq Ito verification}
\begin{aligned}
  \mathbb{E}\Big[e^{-\rho \tau_R \wedge T } \tilde V(X_{\tau_R\wedge T},Y^{y,\xi}_{\tau_R\wedge T})\Big] 
  =& \tilde V (x,y) 
     +\mathbb{E}\bigg[\int_0^{\tau_R\wedge T} e^{-\rho t} (\mathcal{L}_x - \rho) \tilde V (X^x_t,Y^{y,\xi}_{t-})    \dd t \bigg] \\ 
    & + \mathbb{E}\bigg[\int_0^{\tau_R\wedge T} e^{-\rho t} \, \tilde V _y (X^x_t,Y^{y,\xi}_{t-}) \dd Y_t^{y,\xi} \bigg]\\
    &+\mathbb{E}\bigg[\sum_{0\leq t\leq \tau_R\wedge T} \Big( \Delta \tilde V  (X^x_t,Y^{y,\xi}_{t}) - \tilde V _y (X^x_t,Y^{y,\xi}_{t-})\Delta Y^{y,\xi}_t \Big)\bigg], 
\end{aligned}
\end{equation}
with the notations $\Delta \tilde V(X_t,Y^{y,\xi}_{t}) =\tilde V(X_t,Y^{y,\xi}_{t})-\tilde V(X_t,Y^{y,\xi}_{t-})$ and $\Delta Y^{y,\xi}_t =Y^{y,\xi}_t - Y^{y,\xi}_{t-} $.
The concavity of $\tilde V$ in $y$ now implies that
$$
\sum_{0\leq t\leq \tau_R\wedge T} \Big( \Delta \tilde V  (X^x_t,Y^{y,\xi}_{t}) - \tilde V _y (X^x_t,Y^{y,\xi}_{t-})\Delta Y^{y,\xi}_t \Big) \leq 0.
$$
Moreover, since $\tilde V$ solves the HJB equation, we have $\mathbb P \otimes \dd t$-a.e.
$$
(\mathcal{L}_x - \rho) \tilde V (X^x_t,Y^{y,\xi}_{t-}) \leq - \Pi(X^x_t,Y^{y,\xi}_{t-})
\quad \text{and}\quad
- \tilde V_y(X^x_t, Y^{y,\xi}_{t-}) \leq - G(X^x_t).
$$
Plugging the latter two inequalities into \eqref{eq Ito verification} and then rearranging the terms,  we obtain
$$
\begin{aligned}
 \mathbb{E}\bigg[\int_0^{\tau_R\wedge T} e^{-\rho t} \Pi(X^x_t,Y^{y,\xi}_{t-})    \dd t  
     &+ \int_0^{\tau_R\wedge T} e^{-\rho t} G (X^x_t) \dd \xi_t \bigg]\\
     & \leq 
     \tilde V (x,y) -
      \mathbb{E}\Big[e^{-\rho \tau_R \wedge T } \tilde V(X_{\tau_R\wedge T},Y^{y,\xi}_{\tau_R\wedge T})\Big].
\end{aligned}
$$
Finally, thanks to the growth conditions of $\pi$ and $G$, to the growth of $\tilde V$, to the estimates \eqref{eq estimate SDE growth} and \eqref{eq estimate SDE growth SUP},
we can use the dominated convergence theorem to take limits, first as $T \to \infty$ and then as $R\to \infty$, in order to obtain 
$$
J^\lambda (x,y;\xi) \leq \tilde V(x,y).
$$
By arbitrariness of $\xi$, we deduce that $V^\lambda(x,y) \leq \tilde V(x,y)$.

\smallbreak\noindent
\emph{Step 2.}
We adapt the arguments in \cite{deangelis.ferrari.federico.2017optimal} to construct a control $\xi^\lambda$ which satisfies 
$$
J^\lambda (x,y;\xi^\lambda) = \tilde V(x,y).
$$
This identity, together with the previous step, implies that $V^\lambda = \tilde V$ and that $\xi^\lambda$ is optimal.

Define the function $\tilde g_\lambda : \R^n \to [0,1]$ by
$$
   \tilde g_\lambda (x) := \sup \big\{ y \in[0,1]  \, | \, - \tilde V _y (x,y) +  G(x) < 0 \big\},
$$
where we set $\tilde g_\lambda (x):=1$ if $\{ y \in[0,1] \, | \, - \tilde V _y (x,y) +  G(x) < 0 \} = \emptyset$.
The function $ \tilde g_\lambda$ is well defined by concavity of $\tilde V$.
Noticing that $\{ (x,y) \in \R^n \times [0,1] \, | \, y< \tilde g_\lambda(x) \} = \{ (x,y) \, | \, - \tilde V _y (x,y) +  G(x) < 0 \}$, by continuity of $\tilde V _y$ we deduce that the set $\{ (x,y) \in \R^n \times [0,1] \, | \, y<\tilde g_\lambda(x) \}$ is open, so that $\tilde g_\lambda$ is lower semi-continuous.

Next, for fixed $(x,y) \in \R ^n \times [0,1]$, define the process $\xi^\lambda$ by setting
$$
\xi^\lambda_t := \sup_{s\leq t } \big( y- \tilde g_\lambda(X^x_s) \big)^+, \quad t\geq 0. 
$$
Such a process is clearly monotone nondecreasing, so that it admits left limits.
To show its right-continuity, use the upper semi-continuity of $x \mapsto (y-\tilde g_\lambda(x))^+$ to obtain
$$
\begin{aligned}
    \lim_{s \downarrow t} \xi^\lambda_s 
    & = \xi^\lambda_t \lor \lim_{s \downarrow t} \sup_{t \leq r \leq s} \big( y- \tilde g_\lambda(X^x_r) \big)^+ \\
    & = \xi^\lambda_t \lor \limsup_{s \downarrow t} \big( y- \tilde g_\lambda(X^x_s) \big)^+ \leq \xi^\lambda_t \lor \big( y- \tilde g_\lambda(X^x_t) \big)^+ = \xi^\lambda_t.
\end{aligned}
$$
By monotonicity of $\xi^\lambda$, this implies that $\lim_{s \downarrow t} \xi^\lambda_s = \xi^\lambda_t$.
Thus, since this process is clearly adapted, it is progressively measurable, showing the admissibility of $\xi^\lambda$; i.e., $\xi^\lambda \in \mathcal A (y)$.

We will now prove the optimality of $\xi^\lambda$ by repeating, with some modifications, the arguments of Step 1.
For $Y^{y,\lambda} := Y^{y,\xi^\lambda}$, $\tau_R:=\inf\{t \ge 0 \,| \,  |X^x_t| \ge R\}$ and for $T>0$ we find
\begin{equation}\label{eq Ito verification 2}
\begin{aligned}
  \mathbb{E}\Big[e^{-\rho \tau_R \wedge T } \tilde V(X_{\tau_R\wedge T},Y^{y,\lambda}_{\tau_R\wedge T})\Big] 
  =& \tilde V (x,y) 
     +\mathbb{E}\bigg[\int_0^{\tau_R\wedge T} e^{-\rho t} (\mathcal{L}_x - \rho) \tilde V (X^x_t,Y^{y,\lambda}_{t-})    \dd t \bigg] \\ 
    & + \mathbb{E}\bigg[\int_0^{\tau_R\wedge T} e^{-\rho t} \, \tilde V _y (X^x_t,Y^{y,\lambda}_{t-}) \dd Y_t^{y,\lambda} \bigg]\\
    &+\mathbb{E}\bigg[\sum_{0\leq t\leq \tau_R\wedge T} \Big( \Delta \tilde V (X^x_t,Y^{y,\lambda}_{t}) - \tilde V_y (X^x_t,Y^{y,\lambda}_{t-})\Delta Y^{y,\lambda}_t \Big)\bigg], 
\end{aligned}
\end{equation}
with the notations $\Delta \tilde V(X_t,Y^{y,\lambda}_{t}) =\tilde V(X_t,Y^{y,\lambda}_{t})-\tilde V(X_t,Y^{y,\lambda}_{t-})$ and $\Delta Y^{y,\lambda}_t =Y^{y,\lambda}_t - Y^{y,\lambda}_{t-} $.

Notice that, by construction of $\xi^\lambda$, if the process $Y^{y,\lambda}$ jumps at time $t$, then we have $\tilde g_\lambda (X^x_t) < \lim_{s \uparrow t} \tilde g _\lambda (X^x_{s})$ and 
$$
\tilde V_y(X^x_t, \zeta) = G(X^x_t), \quad \text{for any $\zeta \in [Y^{y,\lambda}_t, Y^{y,\lambda}_{t-}]$}.
$$
Therefore, we obtain
$$
\Delta \tilde V (X^x_t,Y^{y,\lambda}_{t}) -\tilde V _y (X^x_t,Y^{y,\lambda}_{t-})\Delta Y^{y,\lambda}_t  = G(X^x_t) (Y^{y,\lambda}_{t}-Y^{y,\lambda}_{t-})  - G(X^x_t) ( Y^{y,\lambda}_t - Y^{y,\lambda}_{t-}) = 0. 
$$
Using the latter equality in \eqref{eq Ito verification 2}, together with the fact that the process $\xi^\lambda$ increases only at times $t$ in which $\tilde V _y (X^x_t,Y^{y,\lambda}_{t-})= G(X^x_t)$ and the fact that $(\mathcal L _x - \rho)\tilde V (X^x_t,Y^{y,\lambda}_{t-}) = \Pi (X^x_t,Y^{y,\lambda}_{t-}) \ \mathbb P \otimes \dd t$-a.e., we find
\begin{equation*}
\begin{aligned}
 \tilde V (x,y)  
    =& \mathbb{E}\bigg[\int_0^{\tau_R\wedge T} e^{-\rho t} \Pi(X^x_t,Y^{y,\lambda}_{t-})    \dd t  
     + \int_0^{\tau_R\wedge T} e^{-\rho t} G (X^x_t) \dd \xi_t^\lambda \bigg] \\ 
     &  + \mathbb{E}\Big[e^{-\rho \tau_R \wedge T } \tilde V(X_{\tau_R\wedge T},Y^{y,\lambda}_{\tau_R\wedge T})\Big].
\end{aligned}
\end{equation*}
%Noticing that $$\begin{aligned}   & \int_0^{\tau_R\wedge T} e^{-\rho t} |\Pi(X^x_t,Y^{y,\lambda}_{t-}) |   d t      + \int_0^{\tau_R\wedge T} e^{-\rho t} |G (X^x_t)| d \xi_t^\lambda\\    & \leq  \int_0^{ T} e^{-\rho t} |(1+|X^x_t|^p) |   d t     + \sup_{t \leq { T}} e^{-\rho t} |(1+|X^x_t|^p)| \xi_{ T}^\lambda\end{aligned}$$ 
Taking limits (first as $T \to \infty$ and then as $R\to \infty$) we conclude that
$$
\tilde V(x,y) = J^\lambda (x,y;\xi^\lambda),
$$
thus proving the optimality of $\xi^\lambda$.
This completes the proof.

\subsection{Proof of Theorem \ref{theorem optimal control}}\label{subsection proof optimal control}
Notice first that the process $\xi^\lambda$ is admissible. In particular, the càdlàg regularity easily follows from the upper semi-continuity of $g^\lambda$, by using arguments as in the proof of Proposition 5.8 in \cite{deangelis.ferrari.federico.2017optimal}. 
In order to show Theorem \ref{theorem optimal control}, it is enough to repeat the arguments in Step 2 in Subsection \ref{subsection uniqueness HJB} with the value function $V^\lambda$ and the free boundary $g_\lambda$ defined in \eqref{eq def free boundary}, and using the fact that $V^\lambda$ is a solution of the HJB equation ({\it cf.}\ Theorem \ref{theorem value function}).

\section{General algorithm}\label{sec:general_algorithm}

In this section, we propose an RL framework to learn the reflection boundary  $g_\lambda$ (namely the policy) and the corresponding value function $V^\lambda$, without prior knowledge on the form of $g_\lambda$. Our proposed RL framework has  two versions: a model-based numerical version (see Section \ref{sec:model-based}) and a model-free learning version (see Section \ref{sec:model-free}). In the first version, where all model parameters are known, we design a Policy Iteration algorithm to numerically find $g_\lambda$. In the second version, where the environment and model parameters are unknown, we provide a sample-based Actor-critic algorithm to learn $g_\lambda$, especially in high dimensions.

 In line with the RL literature, from now on the value function $V^\lambda$ will be referred to as \textit{optimal value function}, while the profit $J^\lambda (x,y;\xi)$ of a given policy $\xi \in \mathcal A (y)$ will be referred to as  \textit{value function} associated to the policy $\xi$.

\subsection{Model-based numerical schemes}\label{sec:model-based}
In this subsection, we present a numerical scheme to solve for the optimal boundary, utilizing full knowledge of the underlying system. As such, this method is termed a model-based analysis.

Motivated by the  results in the previous section, here we focus on a subclass of control policies that can be fully characterized by some reflection boundary $g\in C(\mathbb{R}^d)$.

For a function $g\in C(\mathbb{R}^d)$, we define the associated policy $\xi^g$  of the following form:
  \begin{eqnarray}\label{eq:parameterized_policy}
      \xi^g_t = \sup_{s\leq t} \big(y-g(X_s^x)\big)^+, t \ge 0.
     % \begin{cases}
     %     0, &(x,y): x\in (0,\infty), y \leq g(0)\\
      %    y- \inf_{0\leq s \leq t} g(X^x_s), & (x,y): x\in (0,\infty), y > g(0),
     % \end{cases}
  \end{eqnarray}
%  which is fully characterized by the function $g$.  %Denote $\hat{x}_g = \inf_{x\in \mathbb{R}_+}\{g(x)=1\}$, which is well-defined under the condition that $g$ is non-decreasing.

Define the associated value function as: %{\color{red}[change $\Pi$ to $\pi$; change $X_t$ to $X_t^x$;]}
 \begin{eqnarray} 
      V^{\lambda}_g (x,y) &:= &J^{\lambda}(x,y;\xi^g) \nonumber \\
      &=&\mathbb{E} \left[ \int_0^{\infty}  e^{- \rho t} \big( ( \pi (X^x_{t} )Y_{t}^{y,\xi^g}-\lambda Y_{t}^{y,\xi^g} \log  (Y_{t}^{y,\xi^g} ) \big) \dd t 
+  \int_0^{\infty}  e^{- \rho t} G(X^x_t) \dd\xi_t^g \right]\label{eq:V_g}
  \end{eqnarray}
 subject to $Y^{y,\xi^g}_t = y-\xi_t^g$, with $\xi_t^g$ defined in \eqref{eq:parameterized_policy}.

\iffalse
\begin{proposition}
    $V^{\lambda}_g (x,y)$ defined in  \eqref{eq:V_g} is concave in $y$.
\end{proposition}

\begin{proof}
    We aim to prove that for $y_1,y_2\in[0,1]$ and $\lambda\in[0,1]$, $$V^{\lambda}_g (x,y_3) \ge \lambda V^{\lambda}_g (x, y_1) + (1-\lambda)V^{\lambda}_g (x, y_2),$$
where $y_3:=\lambda y_1 +(1-\lambda)y_2$ .

    Define $\xi_t^{1}:=\sup_{s\leq t} \big(y_1-g(X_s^x)\big)^+$ and $\xi_t^{2}:=\sup_{s\leq t} \big(y_2-g(X_s^x)\big)^+$. Then it is easy to see that
\begin{eqnarray*}
    \xi_t^{3}:=\sup_{s\leq t} \big(y_3-g(X_s^x)\big)^+ &\leq& \lambda \sup_{s\leq t} \big( y_1  -g(X_s^x)\big)^+ + (1-\lambda) \sup_{s\leq t} \big( y_2  -g(X_s^x)\big)^+ \\
    &=& \lambda \xi_t^{1} + (1-\lambda)\xi_t^{2},
\end{eqnarray*}
as $\max\{x,0\}$ is convex and nondecreasing. Hence $Y_t^i = y_i-\xi_t^{i}$ ($i=1,2,3$)  satisfies:
\begin{eqnarray*}
   Y_t^3 \ge \lambda Y_t^{1} + (1-\lambda)Y_t^{2}
\end{eqnarray*}

{\color{red}[Need to have a discussion on this.]}
    
\end{proof}
\fi

The policy in \eqref{eq:parameterized_policy} defines two areas: the exploration area and the stopping area, both associated with $g$:
  \begin{eqnarray}
      \mathcal{E}(g) &:=&\Big\{(x,y)\,\,\Big|\,\, y \leq g(x) \Big\}, \\
      \mathcal{S}(g) &:=& \Big\{(x,y)\,\,\Big|\,\,   y> g(x)  \Big\}.
\end{eqnarray}

In order to show that the updated policy performs better that the original one, we need the following technical condition, which can be easily verified in examples of dimension one (see our companion paper \cite{dianetti2026reinforcement}).
 \begin{assumption}
     \label{thm:property_g} 
    Assume the following  boundary value problem has a unique ${C}^{1}(\mathbb{R}^d\times [0,1])\cap{C}^{2}\big(\overline{\mathcal{E}(g)}\big)$ solution:
     \begin{eqnarray}
      &&(\mathcal{L}_x-\rho) u + \pi(x)\,y - \lambda y \log y =0,                   \quad \text{on} \quad\mathcal{E}(g), \label{eq:HJB_g_1}\\
     && -u_y +G(x)= 0, \quad \text{on}\quad \mathcal{S}(g).
  \end{eqnarray}
In addition, the solution satisfies $\mathcal L _x u\in \mathbb{L}_{loc}^{\infty}(\mathbb{R}^d \times[0,1])$. 
\end{assumption}

 Under Assumption \ref{thm:property_g}, we have   
\begin{eqnarray}
   V^\lambda_g(x,y) \equiv u(x,y). \label{eq:verification_thm}
\end{eqnarray}

With the results in Theorem \ref{thm:property_g}, we update the new boundary as it follows:
  \begin{equation}\label{eq:updated_policy}
    \widetilde{g}(x) :=
      \begin{cases}
       & \max \Big\{0\leq y< g(x)\,\Big| \partial^-_{yy} V^\lambda_{g}(x,y) = 0\Big\} \,\,\text{if} \,\, \partial^-_{yy} V^\lambda_{g}(x,g(x))> 0, \text{ and } \\
    &   \widetilde{g}(x) = g (x)  \qquad \text{otherwise,}  
      \end{cases}
  \end{equation}
where, for a generic function $h$ with left and right second order derivatives, we denote by $\partial_{yy}^- h(y) :=\lim_{\varepsilon\rightarrow 0-} \frac{h(y+\varepsilon)-2h(y)+ h(y-\varepsilon)}{\varepsilon^2}$ the left second-order derivative of $h$. The idea behind the updating rule is to preserve the concavity of the value function $V^\lambda_{g}$ in $y$, while eliminating regions where the second derivative becomes positive.  This differs from  \cite{dianetti2026reinforcement}, where the updating rule is designed based on the cross derivative $V_{xy}$ which is natural when $x$ is one-dimensional. In the high-dimensional regime, however, it is more natural to consider the second-order derivative with respect to $y$. Both updating rules lead to policy improvement results.

We can show that the updating rule  \eqref{eq:updated_policy} always improves in terms of the value function.

\begin{theorem}[Policy improvement]\label{prop:policy_improvement}  Assume Assumption \ref{thm:property_g} hold and $ \widetilde{g}$ is updated according to \eqref{eq:updated_policy}. 
Then, it holds that
    \begin{eqnarray}\label{eq:policy_improvement}
        V^\lambda_{\widetilde{g}}(x,y)\ge V^\lambda_g(x,y),
    \end{eqnarray}
    where $V^\lambda_g, V^\lambda_{\widetilde{g}}$ are value functions associated with policies $\xi^g, \xi^{\widetilde{g}}$; see \eqref{eq:parameterized_policy}.
%\jo{Moreover, if $\tilde g (x) = g (x)$ for any $x \in \R^d$, then $ V^\lambda_{\widetilde{g}} = V^\lambda_{g}$.}
%    In addition, $g' \in \mathcal{C}^1([0,\hat{x}_g])$, where $\hat{x}_{g'} = \inf_{x\in \mathbb{R}_+} \{g'(x)=1\}$. {\color{red}[show non-decreasing]}
\end{theorem}

\begin{proof}
     First of all, $V^\lambda_g \in {C}^{2}\left(\overline{\mathcal{E}(g)}\right)$ according to Assumption \ref{thm:property_g}. %Therefore, by the design of the algorithm, $\widetilde{g}$ satisfies the conditions in Assumption \ref{ass:g}. Hence Theorem \ref{thm:property_g} also applies to $\widetilde{g}$.
    Denote $Y_t^{y, \xi^{\widetilde{g}}}$ as the process under control $\xi^{ \widetilde{g}}$, with $ \widetilde{g}$ defined in  \eqref{eq:updated_policy}. Fix the initial condition $(x,y)$ and take $R>0$. Set $\tau_R:=\inf\{t \ge 0: \|X^x_t\| \ge R\}$. Apply Ito's formula in the weak version to $V^\lambda_g$ (see \cite[Chapter 8, Section VIII.4, Theorem 4.1]{fleming2006controlled} and \cite[Theorem 4.2]{angelis2019solvable}) up to the
stopping time $\tau_R\wedge T$ for some $T > 0$. We then obtain
   {\allowdisplaybreaks \begin{eqnarray}
     &&\mathbb{E}\Big[e^{-\rho (\tau_R\wedge T)} V^\lambda_g(X^x_{\tau_R\wedge T},Y^{y, \xi^{\widetilde{g}}}_{\tau_R\wedge T})\Big] \nonumber\\
     &=& V^\lambda_g(x,y)  
     +\mathbb{E}\Big[\int_0^{\tau_R\wedge T} e^{-\rho t} \Big(\mathcal{L}_x V^\lambda_g(X^x_t,Y^{y, \xi^{\widetilde{g}}}_{t-})-\rho V^\lambda_g(X^x_t,Y^{y, \xi^{\widetilde{g}}}_{t-})\Big) \dd t\Big] \nonumber\\
      && + \mathbb{E}\Big[\int_0^{\tau_R\wedge T} e^{-\rho t} \, \partial_y V^\lambda_g(X^x_t,Y^{y, \xi^{\widetilde{g}}}_{t-}) \dd Y_t^{ \widetilde{g}} \Big]\nonumber\\
      &&+\mathbb{E}\Big[\sum_{0\leq t\leq \tau_R\wedge T}  e^{-\rho t}\Big( \Delta V^\lambda_g(X^x_t,Y^{y, \xi^{\widetilde{g}}}_{t}) - \partial_y V^\lambda_g(X^x_t,Y^{y, \xi^{\widetilde{g}}}_{t-})\Delta Y_t^{y, \xi^{\widetilde{g}}} \Big)\Big]\nonumber\\
      &=& V^\lambda_g(x,y) -\mathbb{E}\Big[\int_0^{\tau_R\wedge T} e^{-\rho t} \Big(\pi(X^x_t)Y_t^{y, \xi^{\widetilde{g}}} - \lambda  Y_t^{y, \xi^{\widetilde{g}}} \log\Big(Y_t^{y, \xi^{\widetilde{g}}}\Big) \Big) \dd t\Big] 
      \nonumber\\
      && +\mathbb{E}\Big[ \int_0^{\tau_R\wedge T} e^{-\rho t} \, \partial_y V^\lambda_g(X^x_t,Y^{y, \xi^{\widetilde{g}}}_{t-}) \dd (Y^{y, \xi^{\widetilde{g}}})_t^{\rm cont} \Big]\nonumber\\
      && +\mathbb{E}\Big[\sum_{0\leq t\leq {\tau_R\wedge T}} e^{-\rho t}\Big( \Delta V^\lambda_g(X^x_t,Y^{y, \xi^{\widetilde{g}}}_{t-})  \Big)\Big]\nonumber\\
      &\geq &V^\lambda_g(x,y) -\mathbb{E}\Big[\int_0^{\tau_R\wedge T} e^{-\rho t} \Big(\pi(X^x_t)Y_t^{y,\xi^{\widetilde{g}}} - \lambda  Y_t^{y, \xi^{\widetilde{g}}} \log\Big(Y_t^{y, \xi^{\widetilde{g}}}\Big) \Big)\dd t\nonumber\\
     && + {\int_0^{\tau_R\wedge T} e^{-\rho t}G(X_t^x)\dd \xi^{\tilde g}_t}\Big],\label{eq:V_g_ineq}
\end{eqnarray}}
\noindent with the notations $\Delta V^\lambda_g(X^x_t,Y^{y,\xi^{\widetilde{g}}}_{t}) =V^\lambda_g(X^x_t,Y^{y,\xi^{\widetilde{g}}}_{t})-V^\lambda_g(X^x_t,Y^{y,\xi^{\widetilde{g}}}_{t-})$ and $\Delta Y_t^{y,\xi^{\widetilde{g}}} =Y_{t}^{y,\xi^{\widetilde{g}}}-Y_{t-}^{y,\xi^{\widetilde{g}}}$. 
The second equality holds because 
    \begin{eqnarray*}
      &&   \mathbb{E}\Big[\int_0^{\tau_R\wedge T} e^{-\rho t} \Big(\mathcal{L}_x\,u(X^x_t,Y^{ y,\xi^g}_{t-})-\rho\, V^\lambda_g(X^x_t,Y^{ y,\xi^g}_{t-})\Big) \dd t \Big] \\
         &=& -\mathbb{E}\Big[\int_0^{\tau_R\wedge T} e^{-\rho t} \Big(\pi(X^x_t)Y_t^{ y,\xi^g} - \lambda  Y_t^{ y,\xi^g} \log\Big(Y_t^{ y,\xi^g}\Big) \Big)\dd t\Big],
    \end{eqnarray*} 
    as pure jumps could only possibly happen at time $0$,  $\xi^{\widetilde{g}}$ keeps $(X^x_t,Y_t^{y,\xi^{\widetilde{g}}})$ within $\overline{\mathcal{E}(g)}$ for $t >0$ and  \eqref{eq:HJB_g_1}. 
    The last inequality holds by the design of  \eqref{eq:updated_policy}. 
    In particular,  for $(x,y)\in \mathcal{S}( \widetilde{g})\cap \mathcal{E}(g)\subseteq \mathcal{E}(g)$, by mean value theorem,
\begin{eqnarray*}
    \partial_y V^\lambda_g(x,y) =  \partial_y V^\lambda_g(x,g(x)) -\partial_{yy} V^\lambda_g(x,\tilde{y})  \leq G(x),
\end{eqnarray*}
with $\tilde y\in[y,g(x)]$, where the last inequality holds due to the design of \eqref{eq:updated_policy}.
    
    In addition, on $\mathcal{S}( \widetilde{g})\cap \mathcal{E}(g)$,
\begin{eqnarray*}
    \Delta V^\lambda_g(X^x_t,Y^{y,\xi^{\widetilde{g}}}_{t}) &=&V^\lambda_g(X^x_t,Y^{y,\xi^{\widetilde{g}}}_{t})-V^\lambda_g(X^x_t,Y^{y,\xi^{\widetilde{g}}}_{t-}) \\
    &\ge & \int_0^{\Delta Y^{y,\xi^{\widetilde{g}}}_t} \partial_y V^\lambda_g\left(X^x_t,Y^{y,\xi^{\widetilde{g}}}_{t-}+u\right) {\rm d} u \ge G(X_t^x)\Delta Y^{y,\xi^{\widetilde{g}}}_t,
\end{eqnarray*}    
as $\Delta Y^{y,\xi^{\widetilde{g}}}_t \leq 0$.

When taking limits as $R \rightarrow\infty$ we have $\tau_R\wedge T\rightarrow T$, $\mathbb{P}$-a.s.  By standard properties of Brownian motion it is easy to prove that the integral terms in the last expression on the right-hand side of \eqref{eq:V_g_ineq} are uniformly bounded in $\mathbb{L}^2(\Omega, \mathbb{P})$, hence uniformly integrable. Moreover,
$V^\lambda_g$ (taking the form as in \eqref{eq:V_g}) has sublinear growth by straightforward calculation. Then we also take limits as $T\rightarrow \infty$ and it follows
that
\begin{eqnarray*}
    V^\lambda_{\widetilde{g}}(x,y) \ge V^\lambda_{g}(x,y),
\end{eqnarray*}
which completes the proof. 
%\jo{Notice that, if $\tilde g (x) = g (x)$ for any $x \in \R^d$, then all the inequalities become equalities, so that we obtain $ V^\lambda_{\widetilde{g}} = V^\lambda_{g}$.}
\end{proof}

Repeat the updating procedure \eqref{eq:updated_policy} iteratively, we have the general heuristic described in Algorithm \ref{alg1}. Algorithm \ref{alg1} falls under the category of Policy Iteration method \cite{sutton2018reinforcement,thomas2016data}, which contains two steps: (1) a Policy Evaluation step that calculates the value function of a given policy ({\it cf.} line \ref{line:33} of Algorithm \ref{alg1}) and (2) a Policy Improvement step that aims to update the policy in the direction of improving the value function ({\it cf.} line \ref{line:44} of Algorithm \ref{alg1}).
The idea behind our specific design is that, we start with an initial $g_0(x)$ that has a sufficiently large exploration region $\mathcal{E}(g_0)$. Then, in each iteration $k$, we update  $g_k$ function  ``downwards'' into $\mathcal{E}(g_k)$ according to the Hessian information ({\it cf.} line \ref{line:33} of Algorithm \ref{alg1}), in which $u_k$ has a better regularity ({\it cf.} line \ref{line:44} of Algorithm \ref{alg1}).

\begin{algorithm}[H]
\caption{Policy Iteration for Entropy-regularized Optimal Stopping (PI-$\lambda$-OS)} \label{alg1}
\begin{algorithmic}[1]
\State Initialize $g_0(x)$ (which is above the optimal $g^*$).
\For{$k=0,1,\cdots,K-1$}
\State\label{line:33} Find $u_k(x,y) $  a ${C}^1(\mathbb{R}^d\times [0,1])\cap{C}^2\left(\overline{\mathcal{E}(g_k)}\right)$ solution to the following equations:
  \begin{eqnarray}
       &&(\mathcal{L}_x-\rho) u_k + \pi(x)\,y - \lambda y \log y, \quad -u_y +G(x)\leq 0 \quad \text{on} \quad\mathcal{E}(g_k),\label{eq:uk-1}\\
     && -\partial_y\, u_k = G(x), \quad \text{on}\quad \mathcal{S}(g_k).\label{eq:uk-2}
  \end{eqnarray}
 \State \label{line:44} Update the strategy
 \begin{equation*}
      {g}_{k+1}(x) =
      \begin{cases}
       & \max \Big\{y< g_k(x)\,\Big| \partial_{yy} V^\lambda_{g}(x,g_k(x)) = 0 \Big\} \,\,\text{if} \,\, \partial^-_{yy} V^\lambda_{g_k}(x,g(x))>0, \text{ and } \\
    &  {g}_{k+1}(x) = g_k (x)  \qquad \text{otherwise,}  
      \end{cases}
 \end{equation*}
\EndFor
\end{algorithmic}
\end{algorithm}
 This iterative scheme is inspired by \cite{kumar2004numerical} for a one-dimensional singular control problem. However, the convergence result in that work requires a second-order condition to hold throughout the entire iteration process, which is difficult to verify. In addition, it is a generalization of the algorithm in \cite{dianetti2026reinforcement}, which is tailored towards the one-dimensional real option problem.

Note that Algorithm \ref{alg1} is a numerical iterative scheme, specifically a model-based numerical learning algorithm, as its implementation requires knowledge of $V^\lambda_g$. In reinforcement learning, however, we generally do not have access to model parameters or the value function for any given policy. Therefore, in the next section, we will develop a model-free learning algorithm. Additionally, to leverage the capabilities of neural networks and design a learning algorithm suitable for higher-dimensional spaces, we employ one neural network to parameterize the value function and another to parameterize the policy. This approach eliminates the need for exhaustive enumeration of discretized state and action spaces.

\subsection{Model-free deep-learning-based  learning algorithms in high-dimensions}
\label{sec:model-free}

We first specify the agent-environment interactions in the learning setting, highlighting what information is available to the learning agent and what is unknown to the agent. We then move to the discussion of the deep-learning-based learning algorithm, which is model-free.

 \subsubsection{Agent-environment interactions.}
 To highlight the distinction between model-based numerical solution and model-free learning, we describe the interactions between the agent and the unknown environment. Although the underlying dynamics are continuous, in practice rewards can only be observed at discrete time points \cite{jia2025accuracy}. The practical agent–environment interactions are summarized in Algorithm \ref{al:agent-environment}.

\begin{algorithm}[H]
\caption{Agent-Environment Interaction} \label{al:agent-environment}
\begin{algorithmic}[1]
\State Input: initial position $x$, policy described in terms of the boundary $g$, observation interval $\delta$ and horizon truncation $T$. Take $K=T/\delta$.
\State Output: for  $k=1,2,\cdots, K$ with $t_k=k\delta$,
\begin{eqnarray}
\int_{t_k}^{t_{k+1}}  e^{- \rho s} \left(\big( ( \pi (X_{s} )Y_{s}^{\xi^g}-\lambda Y_{s}^{\xi^g} \log  (Y_{t}^{\xi^g} ) \big) \dd s 
+   G(X_s) \dd\xi_s^g\right),  \label{eq:instantanuous_reward}
\end{eqnarray}
 and $(X_{t_k},Y_{t_k})$. \Comment{Here the policy $\xi^g$ is defined in  \eqref{eq:parameterized_policy} in reponse to the given boundary $g$.}
\end{algorithmic}

\end{algorithm} 

We have a few remarks for Algorithm \ref{al:agent-environment} in place. 
\begin{remark} \
\begin{enumerate}
    \item  Every time the agent acquires the running reward functions for a given policy $g$ under a single trajectory, \eqref{eq:instantanuous_reward} will be generated under \eqref{eq:x-process} with an independent Brownian motion $B$ and initialization $X_0=x$.

\item Note that the learning agent can only observe the realized instantaneous reward \eqref{eq:instantanuous_reward} at time $t_{k+1}$ ($0\leq k\leq K-1$), but does not have access to its functional form.  
\end{enumerate}
 
\end{remark}

Now we are ready to state the deep-learning algorithm.

\subsubsection{Deep learning algorithm.} Recall from \eqref{eq:V_g}, by dynamic programming principle, we have for $t'>0$
{\begin{eqnarray*} 
      V^{\lambda}_g (x,y) 
      &=&\mathbb{E} \left[ \int_0^{\infty}  e^{- \rho t} \big( ( \pi (X^x_{t} )Y_{t}^{y,\xi^g}-\lambda Y_{t}^{y,\xi^g} \log  (Y_{t}^{y,\xi^g} ) \big) \dd t 
+  \int_0^{\infty}  e^{- \rho t} G(X^x_t) \dd\xi_t^g \right]\\
&=&\mathbb{E} \left[ \int_0^{t'}  e^{- \rho t} \left(\big( ( \pi (X^x_{t} )Y_{t}^{y,\xi^g}-\lambda Y_{t}^{y,\xi^g} \log  (Y_{t}^{y,\xi^g} ) \big) \dd t 
+   G(X^x_t) \dd\xi_t^g\right) + e^{-\rho t'} V_g^\lambda(X_{t'},Y_{t'}) \right],
  \end{eqnarray*} 
  which implies that
  \begin{eqnarray*}
    0 &=&   \mathbb{E}\Bigg[\underbrace{ \int_0^{t'}  e^{- \rho t} \left(\big( ( \pi (X^x_{t} )Y_{t}^{y,\xi^g}-\lambda Y_{t}^{y,\xi^g} \log  (Y_{t}^{y,\xi^g} ) \big) \dd t 
+   G(X^x_t) \dd\xi_t^g\right)}_{\textrm{reward increment}} \\
&&\qquad \qquad + \underbrace{e^{-\rho t'} V^{\lambda}_g (X_{t'},Y_{t'})-V^{\lambda}_g (x,y)}_{\textrm{value difference}}\Bigg]=:\mathbb{E}\big[\delta_{[0,t']}(x,y)\big].
  \end{eqnarray*}
 In the RL literature, $\delta_{[0,t']}(x,y)$ is commonly referred to as the temporal-difference (TD) error, whose expectation is zero \cite{jia2022policy,sutton2008convergent,sutton2009fast}. The TD error consists of two components: the reward increment and the difference in the value function. The former can be directly observed from the environment along a single trajectory. When learning with a parameterized value function (such as those of a neural network), the objective is often formulated as minimizing the variance of the TD error, over admissible parameters.
 
More precisely, denote $\bar{V}^{\eta}(x,y)$ as the approximated value function under trainable parameter $\eta$. We use a model-free critic learning (TD(0) style). Define the TD target (under policy $\xi^g$) by
\begin{eqnarray*}
    {\rm TD}^g_{[t,t+\delta t]}(X_t,Y_t) &:=& \int_t^{t+\delta t}  e^{- \rho (s-t)} \left(\big( ( \pi (X_{s} )Y_{s}^{\xi^g}-\lambda Y_{s}^{\xi^g} \log  (Y_{s}^{\xi^g} ) \big) \dd s 
+   G(X_s) \dd\xi_s^g\right) \\
&&+e^{-\rho (\delta t)}\bar V^\eta(X_{t+\delta t},Y_{t+\delta t}),
\end{eqnarray*}
which is often approximated by the following in practical implementation
\begin{eqnarray*}
    \pi(X_t)Y^{\xi^g}_t\delta\, t-\lambda Y^{\xi^g}_t \log (Y^{\xi^g}_t) \delta\, t + G(X_t) (\xi^g_{t+\delta\,t}-\xi^g_t) + (1-\rho\delta)  \bar V^\eta(X_{t+\delta t},Y^{\xi^g}_{t+\delta t}).
\end{eqnarray*}
As a consequence, the critic loss function is defined as  
\begin{eqnarray}\label{eq:critic_loss}
    \mathcal{L}_{\textrm{critic}}(\eta) = \widehat{\mathbb{E}}\left[\sum_{k=0}^{K-1} \left(\bar V^\eta(X_{t_k},Y_{t_k})-{\rm TD}^g_{[t_k,t_{k+1}]}(X_{t_k},Y_{t_k})\right)^2\right],
\end{eqnarray}
where $\widehat{\mathbb{E}}$ denotes approximating the underlying expectation using independent trajectories from the environment (see Algorithm \ref{al:agent-environment}).
}

If we parameterize the boundary $g$ by trainable parameter $\theta$, then the gradient-based surrogate loss function follows:
\begin{eqnarray}\label{eq:actor_loss}
    \mathcal{L}_{\textrm{policy}} (\theta) = \widehat{\mathbb{E}}\Big[\sum_{k=0}^K\ell (X_{t_k};\theta)\Big],
\end{eqnarray}
with 
\begin{eqnarray*}
 \ell (x;\theta) = \begin{cases}
     \Big(\partial_{yy}\bar V^\eta (x,g^\theta(x))\Big)^2, &\textrm{ if } \partial_{yy} \bar V^\eta (x,g^\theta(x)) >0,\\
     0 & \textrm{otherwise}.
 \end{cases}   
\end{eqnarray*}

\begin{algorithm}[h]
\caption{Actor-Critic Algorithm} \label{alg:actor-critic}
\begin{algorithmic}[1]
\State Initialization of the critic parameter $\eta^0$ and the policy parameter $\theta^0$; The numbers of inner iterations $N$ and outer iterations $L$. Learning rates $\alpha_1>0$ and $\alpha_2>0$ for the critic and policy, respectively.
\For{$\ell=0, 1, \ldots, L-1$}
\State Update the  critic parameter:
\begin{eqnarray*}
    \eta^{\ell+1} = \eta^{\ell} -\alpha_1 \nabla \mathcal{L}_{\textrm{critic}}(\eta^\ell).
\end{eqnarray*}
$\qquad$ where $\mathcal{L}_{\textrm{critic}}$ is calculated using output from Algorithm \ref{al:agent-environment}.
\For{$n=0, 1, \ldots, N-1$}
\State Update the policy parameter:
\begin{eqnarray*}
    \theta^{n+1} = \theta^{n} -\alpha_2 \nabla \mathcal{L}_{\textrm{policy}}(\theta^n).
\end{eqnarray*}
\EndFor
\EndFor
\end{algorithmic}
\end{algorithm}

We have a  few remarks in place.
\begin{remark} \
\begin{enumerate}
    \item In the continuous-time RL literature, TD-type objectives have been investigated in the settings of regular control subject to deterministic dynamics \cite{doya2000reinforcement} as well as stochastic dynamics \cite{jia2022policy,jia2022policy2}. However, these developments have not addressed problems involving optimal stopping or singular control. Our framework provides the first result of this kind in the literature.
    \item Two time-scale update: For each actor (policy) update, we update the critic (value) multiple times. This two-time scale approach, with faster updates for the value function and slower updates for the policy, enables more stable and efficient learning in complex environments.
\end{enumerate}
\end{remark}

\section{Numerical Experiments}\label{sec:numerics}
In this section, we test the performance of our Actor-critic Algorithm \ref{alg:actor-critic} in two settings. The first setting is a one-dimensional toy case, which allows us to use a numerical PDE solution as a benchmark for comparison. The second case demonstrates the power of our algorithm in high dimensions.
\subsection{One dimensional case as a benchmark}
\medskip

\emph{Numerical set--up.}
The state process 
$X_t\in\mathbb{R}$ follows a discretized Ornstein--Uhlenbeck dynamics
\[
X_{t+\delta}
= X_t + \big(-\theta X_t\,\delta  + \sigma \sqrt{\delta }\,Z_t\big),
\]
with $\theta=0.2$, $\sigma=0.4$, $\delta =0.05$, and discount factor
$\gamma = e^{-\rho\Delta t}$ with $\rho=0.05$, and $Z_t$ are iid standard Gaussian random variables.
The control acts through a boundary function $g(x)\in[0,1]$ governing the auxiliary state $Y_t\in[0,1]$,
\[
Y_{t+\delta} = \min\big(Y_t, g(X_t)\big), \qquad \Delta_t = Y_t - Y_{t+\delta}\ge 0,
\]
so that $\Delta_t$ represents the amount of singular intervention.
The instantaneous cost combines a running term
$a_1 X_t^2 Y_t + \lambda Y_t\log Y_t$ and a singular cost $a_2 X_t^2\Delta_t$,
with parameters $(a_1,a_2,\lambda) = (2,1,0.3)$.

\medskip
\emph{Actor network and critic network.}
We approximate the value function with $V_\eta$, using a   fully connected feed-forward network $\eta$ ({\tt ValueNet}) with two hidden layers of width 64 and ReLU activations, that inputs $(x,y)$
and outputs a scalar. It consists of two hidden layers with 64 ReLU units each.
The control boundary  is approximated by $g^\theta(x)$, with a separate network ({\tt PolicyNet})
of the same architecture but with a sigmoid output layer mapping into $[0,1]$. %An exponential moving average copy of the policy network is maintained for stability during evaluation.

\medskip

\emph{Online learning scheme.}
At each iteration, we simulate $L=100$ trajectories of length $T=200$
under the current policy.
We update the critic via the TD(0) objective with $N=50$ critic gradient steps per actor update.
The learning rate for the actor and critic are both $\alpha_1=\alpha_2=10^{-5}$.

%\emph{(ii) Offline Monte Carlo refinement:} after online training, we retrain the critic on Monte Carlo value estimates computed over $1000$ randomly sampled states, using 200--step returns with 40 replicates, minimizing normalized MSE for 40 epochs.

%\emph{(iii) Monte Carlo policy improvement:} for each grid point $x$, we evaluate 32 candidate boundary levels and compute approximate returns via multi--replicate rollouts (400 steps, 30 replicates). The resulting maximizers $g_{\mathrm{MC}}(x)$ are fit as a supervised regression target over 200 epochs to yield an improved policy. A final Monte Carlo refinement phase is run under the improved policy, and diagnostics (boundary evolution, value function slices, MC error plots) are used to assess convergence. Random seeds are fixed for reproducibility and gradients are clipped to ensure numerical stability.

\medskip

\emph{HJB benchmark.} Note that for the numerical example we consider, there is no explicit solution. To validate our results, we compute a reference solution by numerically solving the  HJB equation on the $(x,y)$-grid using finite differences. The computational domain is $x\in[-5,5]$ discretized over $121$ grid points, and
$y\in[0,1]$ over $32$ points.  The optimal value function and free boundary extracted from this solver serve as benchmarks for assessing the accuracy of our learned policy and critic.

\medskip

\emph{Experiment performance.} The experiment results demonstrate that the proposed Actor-critic method (without knowing the underlying model) closely replicates the benchmark HJB solution. In Figure~\ref{fig:value}, the learned value function exhibits the same qualitative structure as the finite-difference HJB solution, with discrepancies primarily concentrated near the free boundary region where the value function transitions between active and inactive control. The relative error plot shows that these deviations remain small and spatially localized. Figure~\ref{fig:g_slice-v} further highlights this agreement: the learned boundary function 
$g(x)$ matches the shape and location of the numerically computed optimal boundary, while the one-dimensional slice of the value function at 
$y=0.5$ almost coincides with the HJB reference. Overall, these comparisons indicate that the actor–critic scheme not only captures the global geometry of the optimal value landscape, but also accurately resolves the free boundary that characterizes the optimal singular control policy.

\begin{figure}[H]
    \centering
    \includegraphics[width=0.3\linewidth]{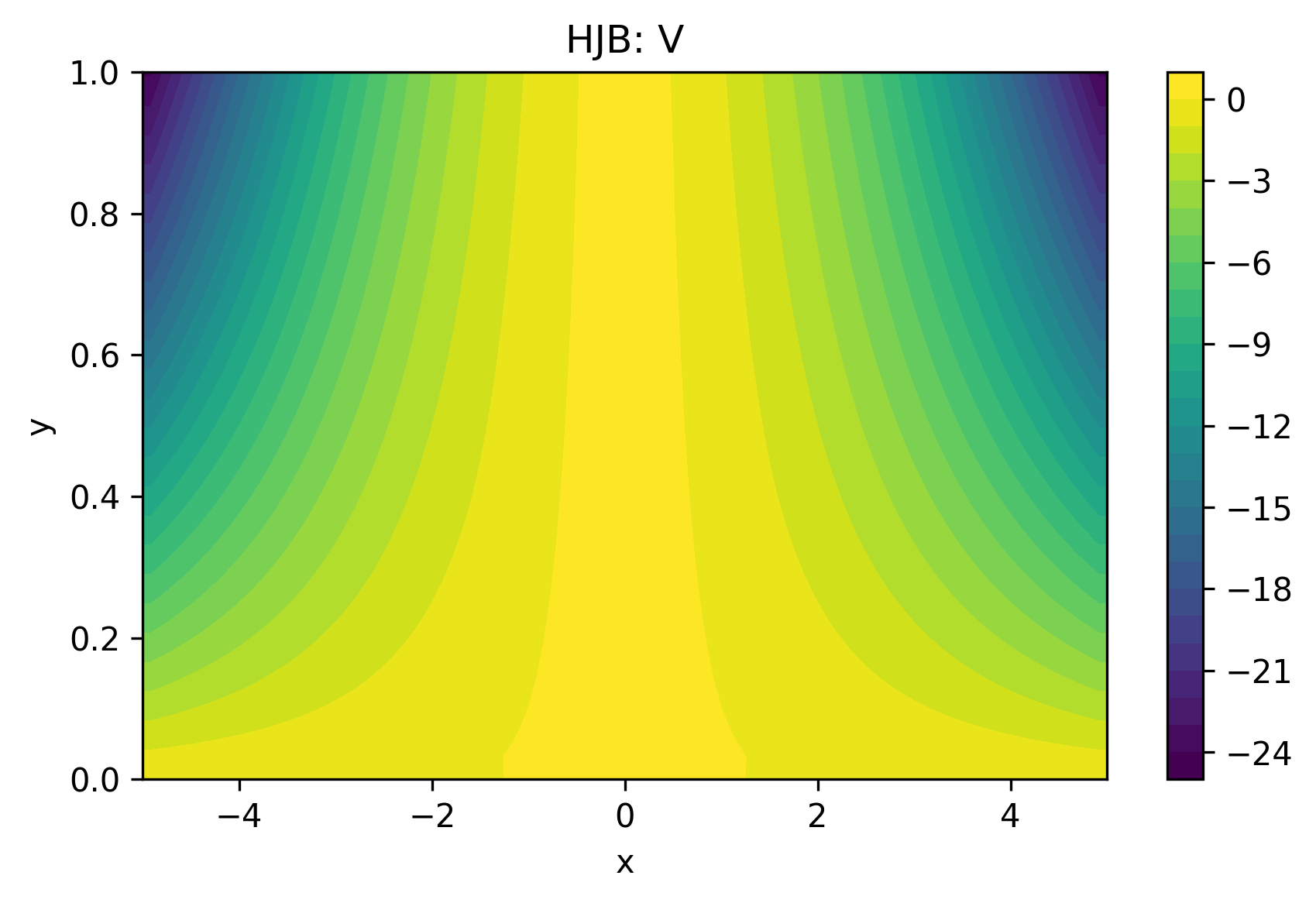}
    \includegraphics[width=0.3\linewidth]{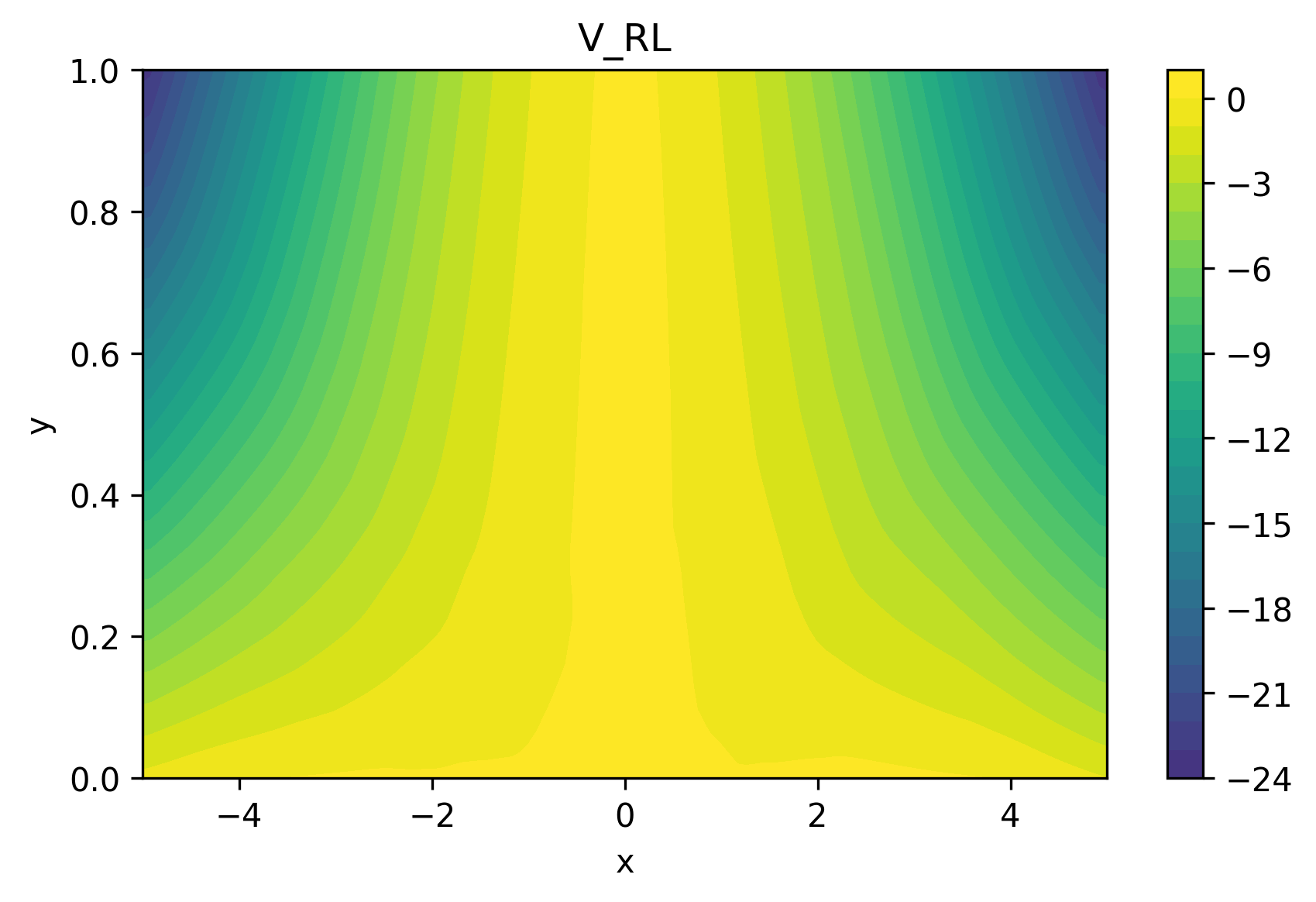}
    \includegraphics[width=0.3\linewidth]{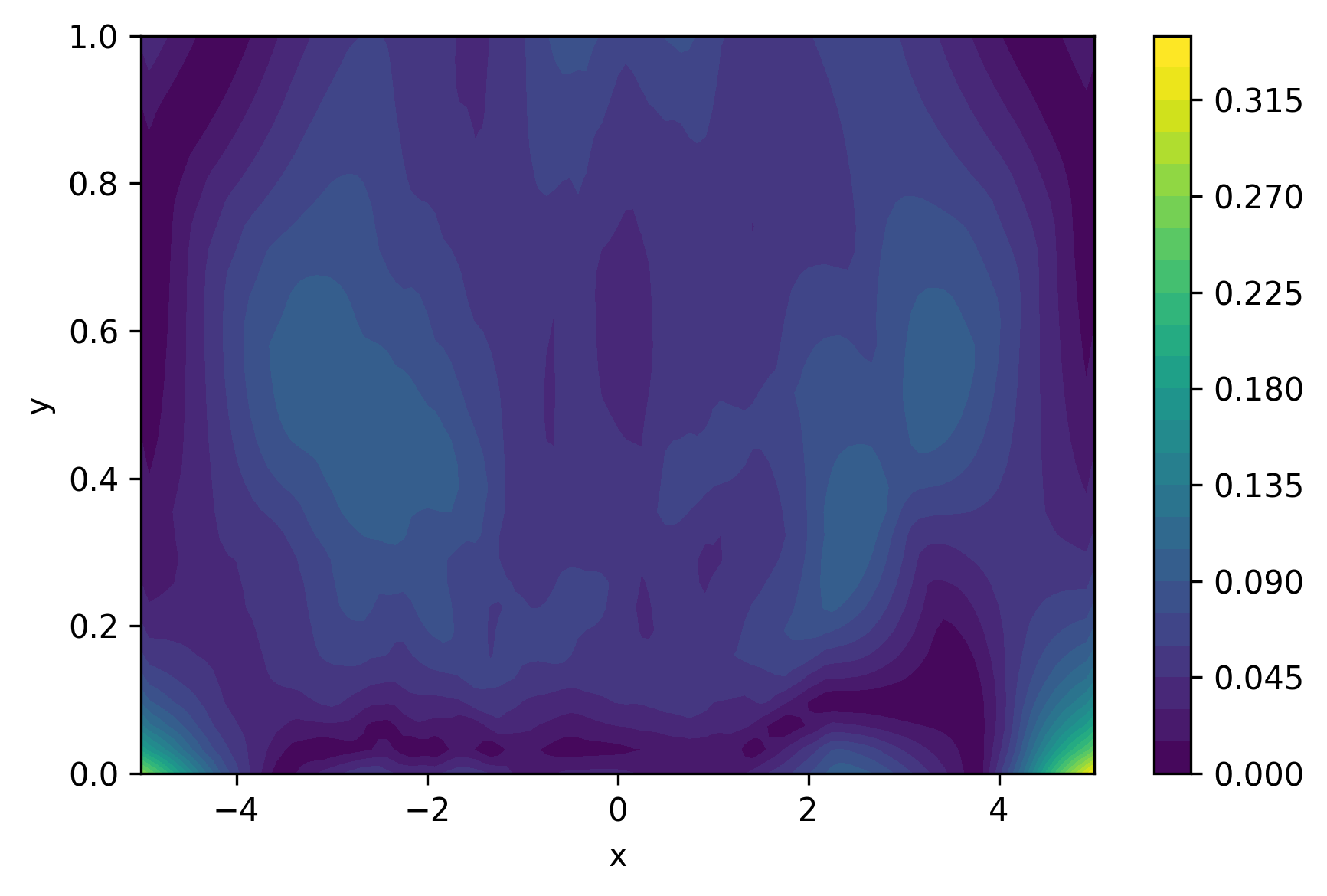}
    \caption{Value functions: (left) value function from numerical HJB solution. (middle) value function solved from our actor-critic method. (right) relative error.}
    \label{fig:value}
\end{figure}

\begin{figure}[H]
    \centering
    \includegraphics[width=0.35\linewidth]{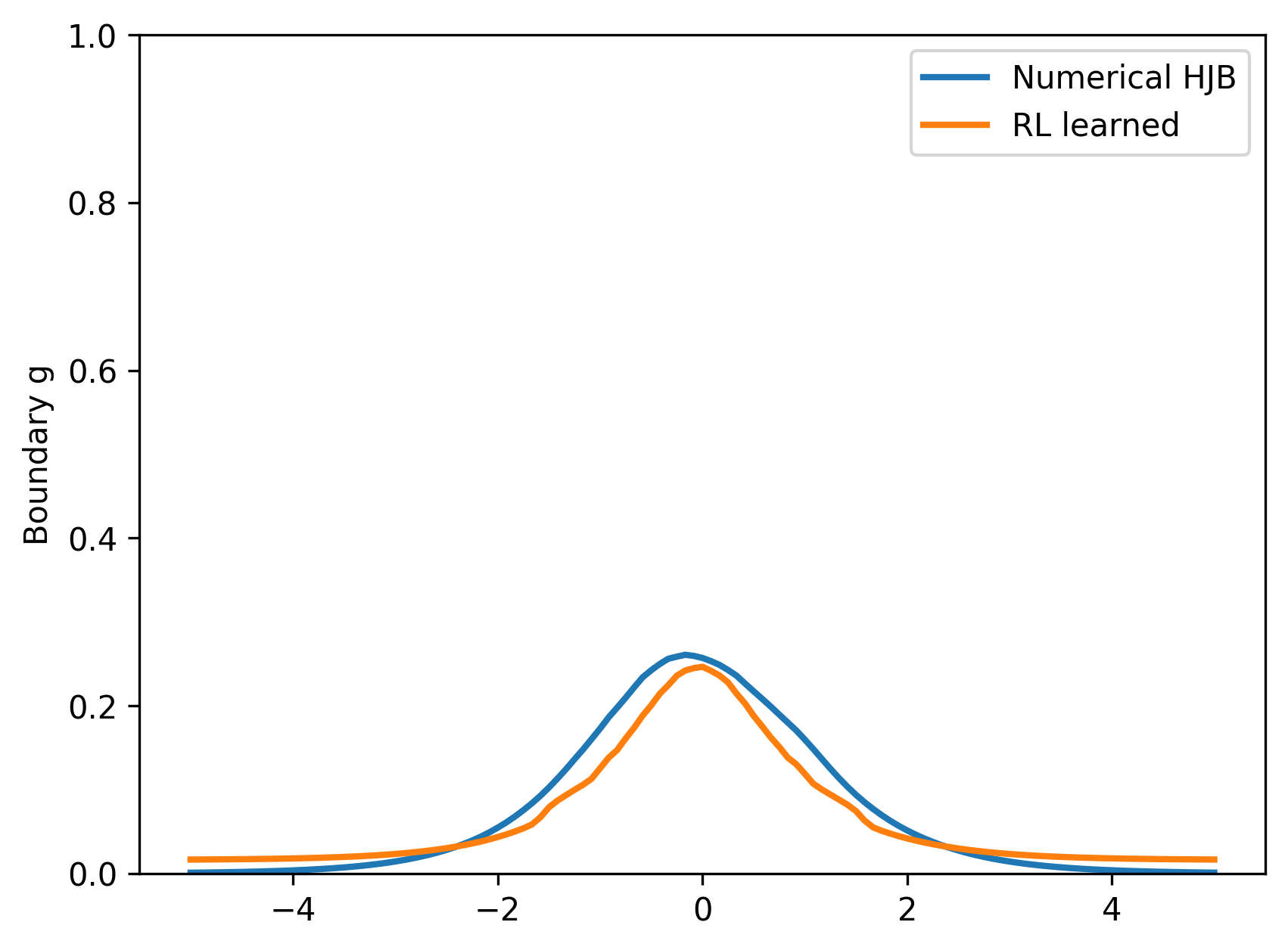}
     \includegraphics[width=0.45\linewidth]{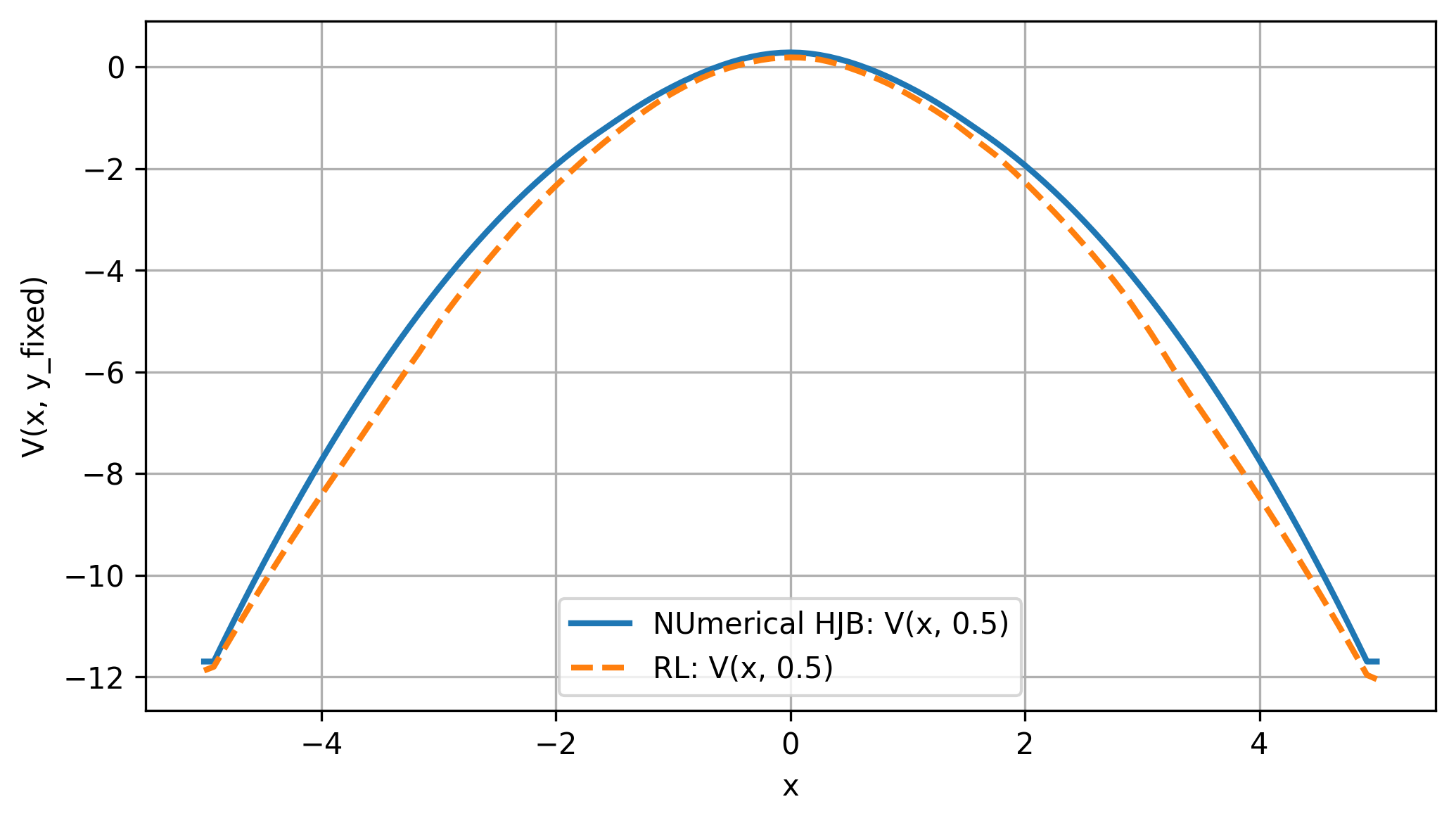}
    \caption{Comparison of the boundary and sliced value function. (left): comparison of the boundary $g$. (right): sliced value function at $y=0.5$.}
    \label{fig:g_slice-v}
\end{figure}

\subsection{High-dimensional case}

\emph{Numerical set-up.} We study a high-dimensional singular stochastic control problem posed on the state space $(X_t,Y_t)\in\mathbb{R}^{10}\times[0,1]$. The process $X_t=(X_t^{(1)},\dots,X_t^{(10)})$ evolves as a heterogeneous Ornstein--Uhlenbeck diffusion
\[
X_{t+\delta}^{(i)}=X_{t}^{(i)}+\Big(-\theta_i X_t^{(i)}\,\delta+\sigma\,\sqrt{\delta}Z_t^{(i)}\Big),\qquad i=1,\dots,10,
\]
where $\delta=0.05$, $\sigma=0.4$ and the drift parameters increase linearly from $0.1$ to $3.0$. In addition, $Z_t^{(i)}$ are iid standard Gaussian random variables. The controlled component $Y_t$ represents the remaining budget of exploration, implemented in discrete time as
\[
Y_{t+\delta}=\min\{Y_t,g(X_t)\}, \qquad \Delta_t:=Y_t-Y_{t+\delta}\ge0,
\]
where $g:\mathbb{R}^{10}\to [0,1]$ is the boundary. The instantaneous cost  is
\[
c(x,y,\Delta)=a_1(x)y+\lambda\, y\log y+a_2(x)\Delta,
\]
with $\lambda=0.3$. Here $a_1(x) = \log(1+e^{\tilde a_1(x)})$ and $a_2(x) = \log(1+e^{\tilde a_2(x)})$ are the smooth counterparts of $\tilde a_1$ and $\tilde a_2$, where 
\[
\tilde a_1(x)=\sum_{i=1}^{10}a_{1,i}x_i^2 +2.0 x_1-1.4x_5 ,\qquad
\tilde a_2(x)=\sum_{i=1}^{10}a_{2,i}x_i^2 +0.4 x_1+0.2x_5.
\]
The coefficients $a_{1,i}$ decrease linearly from $5.0$ to $0.5$, whereas $a_{2,i}$ increase from $0.1$ to $5.0$, inducing heterogeneous incentives across coordinates. The discount rate is set as $\rho=0.05$.

\begin{figure}[H]
    \centering
\includegraphics[width=0.5\linewidth]{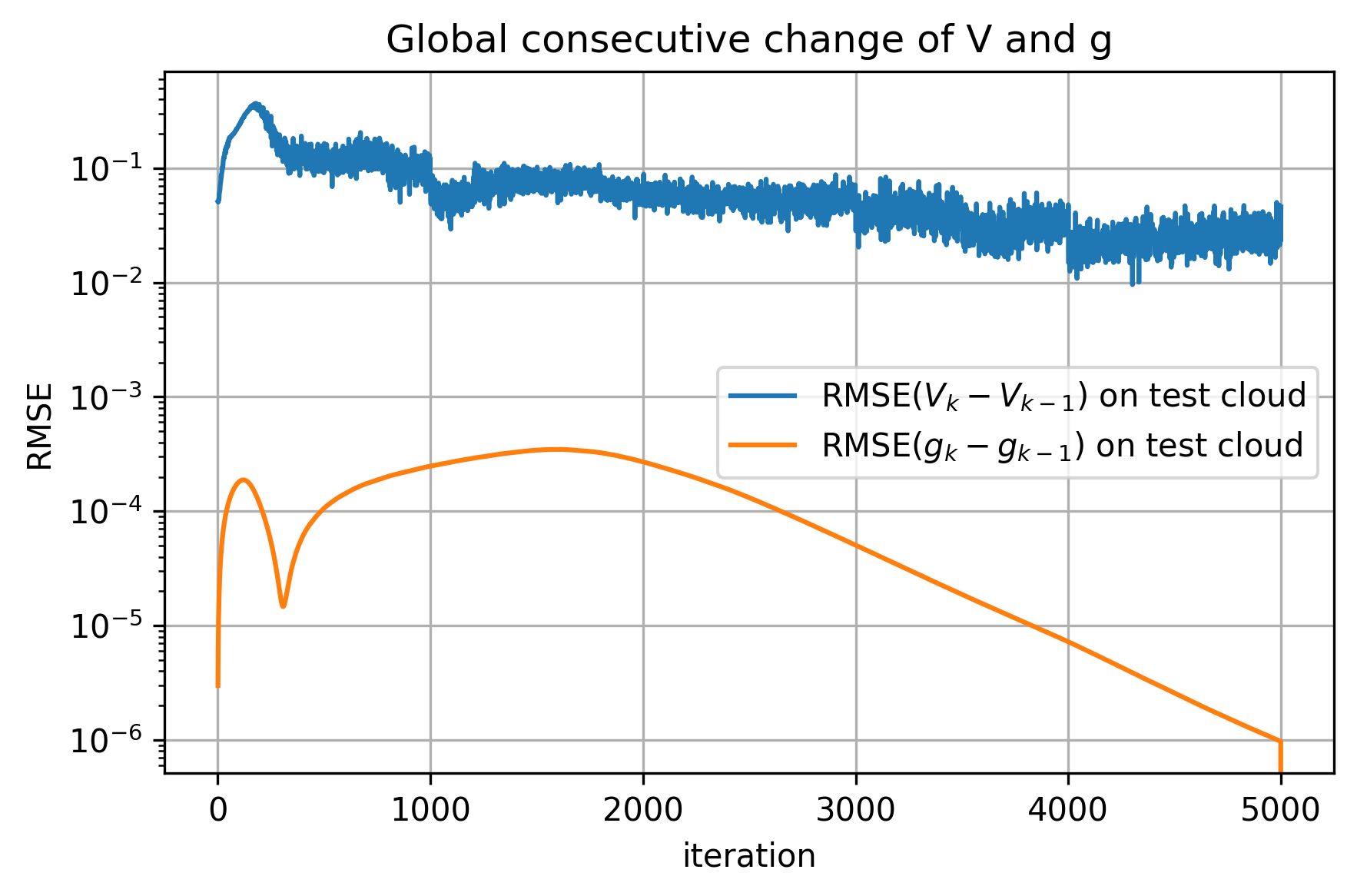}
    \caption{Convergece rate: consecutive changes of the value and policy updates (tested on a cloud of 4000 prefixed points).}
    \label{fig:consecutive_changes}
\end{figure}

\begin{figure}[H]
    \centering
 \includegraphics[width=0.9\textwidth]{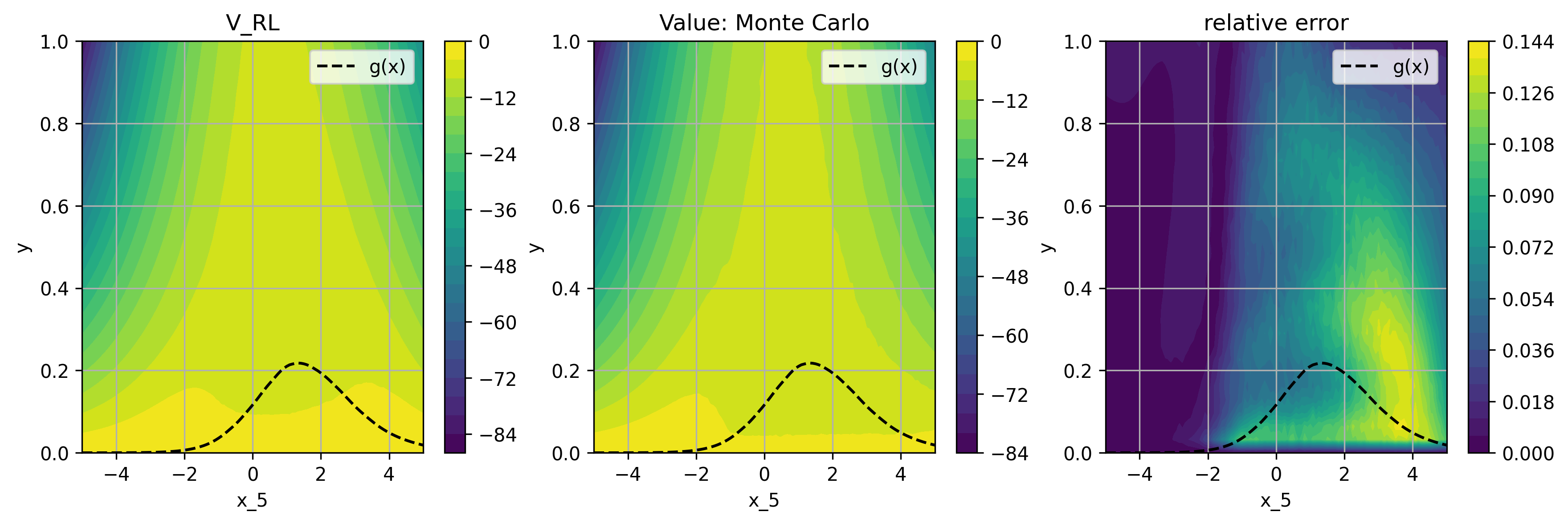}
  \includegraphics[width=0.9\textwidth]{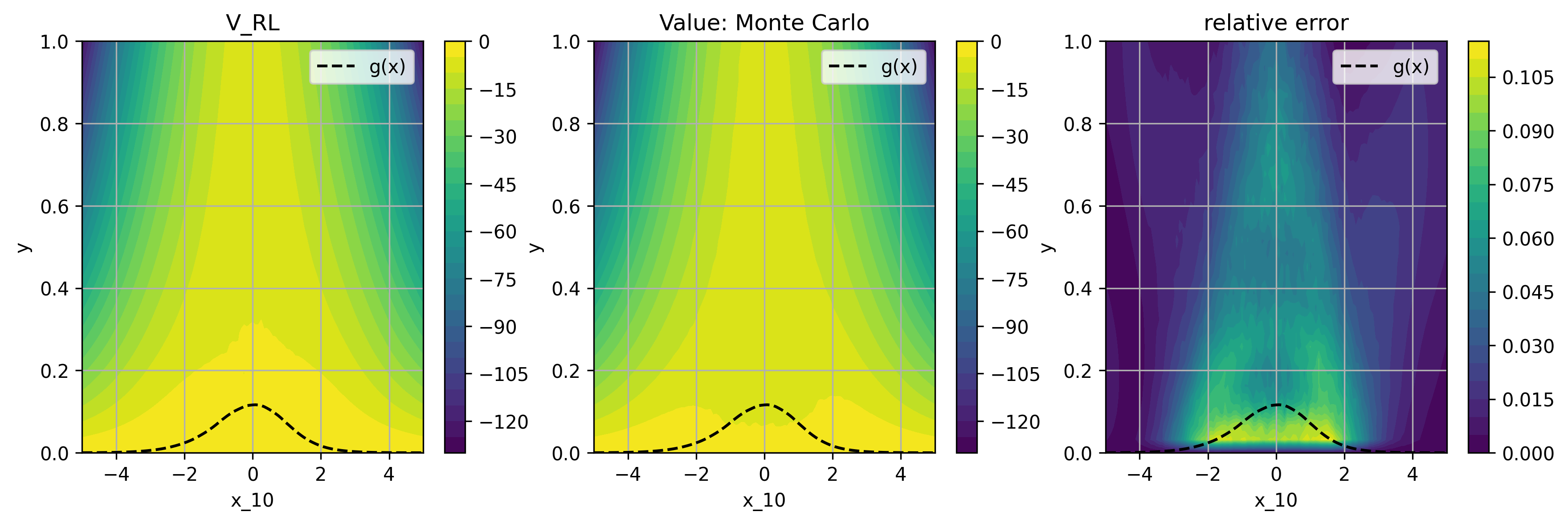}
    \caption{Sliced value function for visualization (top row for $x_5$,  and bottom row for $x_{10}$). Left columns are the learned value function, middle columns are the simulated value function by Monte Carlo (benchmark), and the right columns are the relative error.}
    \label{fig:value_big_comparison}
\end{figure}

\medskip

\emph{Actor network and critic network.} Similar as beofre,
 the critic $V^\eta(x,y)$ is a two-layer ReLU network with input dimension $11$ and scalar output. The actor $g^\theta(x)$ is a two-layer ReLU network with sigmoid output mapping $\mathbb{R}^{10}\to(0,1)$. Both networks have hidden width $64$, are trained with Adam (learning rate $10^{-5}$), and use gradient clipping.
\medskip

\emph{Online learning scheme.} The online phase consists of $L=5{,}000$ iterations. Each iteration simulates $1{,}000$ trajectories of length $T=200$ using the current actor. The critic is updated via TD$(0)$ using $N=50$ gradient steps per iteration. Convergence is monitored on a fixed test cloud of $2{,}000$ states, recording root mean-square error (RMSE) changes of successive actor and critic outputs.

\emph{Evaluation.} Note that a numerical HJB solver cannot achieve arbitrarily accurate results in high dimensions. Hence we do not compare with HJB benchmarks. Instead, we report the convergence rate and perform sanity checks on the output value functions and policies. Additionally, after training, we use the Monte Carlo method to approximate the value function of the output policy and compare it to our  output value function (parameterized by neural network).
\medskip

\emph{Results.}
The numerical results provide strong evidence that the proposed Actor-critic algorithm successfully learns both the value function and the optimal singular control boundary in a genuinely high-dimensional non-radial setting. Figure \ref{fig:consecutive_changes} reports the global RMSE of successive value and policy iterates on a fixed test cloud of sampled states. The policy error decays over three orders of magnitude and stabilizes to machine-precision fluctuations, while the value error plateaus at a small level consistent with the intrinsic variance of Monte Carlo sampling and the critic approximation capacity. Figure \ref{fig:value_big_comparison} compares two-dimensional slices of the learned value function with Monte Carlo benchmarks across two representative coordinates. In each case the learned surfaces replicate the qualitative shape and monotonic structure of the Monte Carlo reference, and the residual plots confirm that discrepancies are  small relative to the overall value scale. Note that we introduced an asymmetric term in $x_5$ in the cost function, which shifts the center of 
$g$ toward positive values, consistent with the learned policy in Figure~\ref{fig:value_big_comparison}-top. In contrast, since the cost is symmetric in 
$x_{10}$, the projection of the learned policy $g$ onto the $x_{10}$-coordinate in Figure~\ref{fig:value_big_comparison}-top is symmetric.

\newpage

\appendix
\section{Proof of Lemma \ref{lemma uniqueness stopping time}}\label{appendix proof uniqueness stopping time}

Setting $\hat V := V -G $, by It\^o's formula, problem \eqref{eq:value_classic_os} reduces to the OS problem
$
\hat V (x) = \sup_{\tau \in \mathcal T} \mathbb E \big[ \int_0^\tau e^{-\rho t} \hat \pi ( X_t^x) dt  \big],
$ 
for which the value function  $\hat V$ is the unique $W^{2,2}_{loc}(\R)$-a.e.\ solution to the HJB variational inequality (see. e.g., Section 5.2 in \cite{pham2009continuous})
\begin{equation}
\max \big\{ \left(\mathcal{L}_{x} -\rho \right) \hat V  + \hat \pi  , - \hat V \big\}=0.
\end{equation}
By Theorem 2.3.5 in \cite{lamberton.2009optimal}, the minimal stopping time $\tau^*$ is the first hitting time to the region  $\hat{\mathcal S} : = \{ \hat V = 0 \}$. 
Also, since $\hat \pi $ is increasing, $\hat V$ is nondecreasing so that $\hat{\mathcal S}$ is a closed interval (possibly empty). 
If $\hat{\mathcal S}$ is empty, then $\tau^* = \infty$ is the unique optimal stopping time, by minimality of $\tau^*$.
On the contrary, if  $\hat{\mathcal S}$ is nonempty, then we can find $x^* \in \R$ such that $\hat{\mathcal S} : = \{ \hat V = 0 \} = (-\infty, x^* ]$.
Notice moreover that necessarily we have $x^* \leq \bar x$, as otherwise we would have $(\mathcal L_x - \rho) \hat V + \hat \pi = \hat \pi > 0$ in the interval $(\bar x, x^*)$, thus contradicting the HJB equation.
In other words, 
\begin{equation} 
    \label{eq negative pi}
    \text{$(\mathcal L_x - \rho) \hat V + \hat \pi = \hat \pi < 0$ in $(-\infty,x^*]$ and $(\mathcal L_x - \rho) \hat V + \hat \pi < 0$, a.e.\ in $\mathbb R$.}
\end{equation}

Next, assume by contradiction that there exists another optimal stopping time $\tau$. 
By minimality of $\tau^*$ we have $\tau^* \leq \tau$.
Moreover, by the martingale optimality principle (see e.g.\ Theorem 2.2 at p.\ 29 in \cite{peskir.shiryav.2006optimal}), the process $M = (M_t)_t$ defined by
$$
M_t := e^{-\rho t \land \tau} \hat V (X^x_{t\land \tau}) + \int_0^{t\land \tau} e^{-\rho s} \hat \pi (X^x_s) \dd s,
$$
is an $\mathbb F ^W$-martingale and, using It\^o's formula (thanks to the regularity of $\hat V$) we have
$$
\E [M_t] := \hat V(x) +\E \bigg[ \int_0^{t\land \tau} e^{-\rho s} \big((\mathcal L_x - \rho) \hat V + \hat \pi \big) (X^x_s) \dd s  \bigg].
%+ \int_0^{t\land \tau} e^{-\rho s}  D_x \hat V (X^x_s) \sigma (X^x_s) d W_s.
$$
Notice that, since $\sigma^2(x) \geq c >0$,  after the time $\tau^*$ the Brownian motion pushes the state $X_t^x$ in the region $(-\infty, x^*)$ with probability 1.
Thus,  by \eqref{eq negative pi}  and the characterization of $\tau^*$ we have
$$
\begin{aligned}
\E [M_\tau] &= \hat V(x)+ \E \bigg[
\int_0^{  \tau} e^{-\rho s} \big((\mathcal L_x - \rho) \hat V + \hat \pi \big) (X^x_s) \dd s \bigg] \\
& =  \hat V(x)+ \E \bigg[ \mathds 1 _{ \{ \tau > \tau^* \} } \int_{\tau^*}^{\tau} e^{-\rho s} \big((\mathcal L_x - \rho) \hat V + \hat \pi \big) (X^x_s) \dd s  \bigg] < \hat V (x) = \E[M_0],
\end{aligned}
$$
where the last inequality follows from assuming that the event $\tau > \tau^*$ has positive probability. 
This contradicts the fact that $M$ is a martingale, and thus show that the optimal stopping time $\tau^*$ is unique.

\section{Proof of the auxiliary estimates}\label{appendix proof of estimates}
%\subsection{Proof of \eqref{eq estimate SDE growth} -- \eqref{eq estimate SDE semiconv SUP}}
We prove each of the estimate \eqref{eq estimate SDE growth}--\eqref{eq estimate SDE semiconv SUP} separately. 
\smallbreak \noindent
\emph{Proof of \eqref{eq estimate SDE growth}.}
Noticing that $|x| (1+|x|) \leq \frac32(1+|x|^2)$, we have
$$
\begin{aligned}
x b(x)+\frac{q-1}{2}|\sigma(x)|^{2} & \leq \left(\frac{3}{2} L+(q-1) L^{2}\right)\left(1+|x|^{2}\right).
\end{aligned}
$$
Hence, the estimate \eqref{eq estimate SDE growth} follows from Theorem 4.1 at p. 60 in \cite{mao2008}.

\smallbreak \noindent
\emph{Proof of \eqref{eq estimate SDE Lip}.}
We adopt the rationale in the proof of Theorem 4.1 at p. 60 in \cite{mao2008}.
Set $\Delta_{t}:=X_{t}^{\bar{x}}-X_{t}^{x}$. 
For any $\varepsilon>0$, using It\^o's formula together with elementary estimates, we find
\begin{equation}\label{eq useful Lip}
\begin{aligned}
\Big(\varepsilon+& \left|\Delta_{t}\right|^{2} \Big)^{\frac{q}{2}} \\
=&  \left(\varepsilon+|\bar{x}-x|^{2}\right)^{\frac{q}{2}} 
  + q \int_{0}^{t}\left(\varepsilon+\left|\Delta_{t}\right|^{2}\right)^{\frac{q-2}{2}} \Delta_{t}^{*}\left(\sigma\left(X_{t}^{\bar{x}}\right)-\sigma\left(X_{t}^{x}\right)\right) \dd W_{t} \\
& +q \int_{0}^{t}\left(\varepsilon+\left|\Delta_{t}\right|^{2}\right)^{\frac{q-2}{2}} \Delta_{t}^{*} \left(b\left(X_{t}^{\bar{x}}\right)-b\left(X_{t}^{x}\right)\right) \dd t \\
& +q \frac{q-2}{2} \int_{0}^{t}\left(\varepsilon+\left|\Delta_{t}\right|^2\right)^{\frac{q-4}{2}}\left|\Delta_{t}^* \left(\sigma\left(X_{t}^{\bar{x}}\right)-\sigma\left(X_{t}^{x}\right) \right) \right|^{2} \dd t  \\
& +\frac{q}{2} \int_{0}^{t}\left(\varepsilon+\left|\Delta_{t}\right|^{2}\right)^{\frac{q-2}{2}}\left|\sigma\left(X_{t}^{\bar{x}}\right)-\sigma\left( X_{t}^{x}\right)\right|^{2} \dd t \\
\leq & \left(\varepsilon+|\bar{x}-x|^{2}\right)^{\frac{q}{2}}
+q \int_{0}^{t}\left(\varepsilon+\left|\Delta_{t}\right|^{2}\right)^{\frac{q-2}{2}} \Delta_{t}^{*} \left(\sigma \left(X_{t}^{\bar{x}}\right)-\sigma \left(X_{t}^{x}\right)\right) \dd W_t \\
& +q \int_{0}^{t}\left(\varepsilon+|\Delta_t|^2\right)^{\frac{q-2}{2}}\left[\Delta_{t}^{*}\left(b\left(X_{t}^{\bar{x}}\right)-b\left(X_{t}^{x}\right)\right)+\frac{q-1}{2}\left|\sigma\left(X_{t}^{\bar{x}}\right)-\sigma\left(X_{t}^{x}\right)\right|^2 \right] \dd t  \\
\leq &
\left(\varepsilon+|\bar{x}-x|^{2}\right)^{\frac{q}{2}}  +q \int_{0}^{t}\left(\varepsilon+\left|\Delta_{t}\right|^{2}\right)^{\frac{q-2}{2}} \Delta_{t}^{*}\left(\sigma\left(X_{t}^{\bar{x}}\right)-\sigma\left(X_{t}^{x}\right)\right) \dd W_{t} \\
& +q\left(L+\frac{q-1}{2} L^{2}\right) \int_{0}^{t}\left(\varepsilon+|\Delta_t|^{2}\right)^{\frac{q}{2}} \dd t .
\end{aligned}
\end{equation}
Taking expectations and using Gr\"onwall inequality, we obtain
$$
\mathbb{E}\left[\left(\varepsilon+|\Delta_t|^{2}\right)^{\frac{q}{2}}\right] \leq
\left(\varepsilon+|\bar{x}-x|^{2}\right) ^{\frac q 2 }e^{q\left(L+\frac{q-1}{2} L^{2}\right) t},
$$
which completes the proof sending $\varepsilon \rightarrow 0$.

\smallbreak \noindent
\emph{Proof of \eqref{eq estimate SDE semiconv}.}
Set
$$
\begin{aligned}
& \hat{X}_{t}:=\delta X_{t}^{\bar{x}}+(1-\delta) X_{t}^{x}, \quad X_{t}^{\delta}:=X_{t}^{\delta \bar{x}+(1-\delta) x}, \quad \Gamma_{t}:=\hat{X}_{t}-X_{t}^{\delta}, \\
& \hat{b}_{t}:=\delta b\left(X_{t}^{\bar{x}}\right)+(1-\delta) b\left(X_{t}^{x}\right), \quad 
\hat{\sigma}_{t}:=\delta \sigma\left(X_{t}^{\bar{x}}\right)+(1-\delta) \sigma\left(X_{t}^{x}\right).
\end{aligned}
$$
By It\^o's formula, with elementary estimates we obtain 
$$
\begin{aligned}
\left(\varepsilon+\left|\Gamma_{t}\right|^{2}\right)^{\frac{q}{2}}=\varepsilon^{\frac{q}{2}} 
& 
+q\int_{0}^{t}\left(\varepsilon+\left|\Gamma_{t}\right|^{2}\right)^{\frac{q-2}{2}} \Gamma_{t}^{*}\left[\left(\hat{b}_{t}-b\left(X_t^\delta\right)\right) \dd t+\left(\hat{\sigma}_{t}-\sigma\left(X_t^\delta\right)\right) \dd W_{t}\right] \\
& +q \frac{q-2}{2} \int_{0}^{t}\left(\varepsilon+\left|\Gamma_{t}\right|^{2}\right)^{\frac{q-4}{2}}\left| \Gamma_{t}^* \left( \hat{\sigma}_{t}-\sigma\left(X_{t}^{\delta} \right) \right)\right|^{2} \dd t \\
& +\frac{q}{2} \int_{0}^{t}\left(\varepsilon+\left|\Gamma_{t}\right|^{q-2}\right)^{\frac{q-2}{2}}\left|\hat{\sigma}_{t}-\sigma\left(X_t^\delta\right)\right|^{2} \dd t \\
%\leq \varepsilon^{\frac{q}{2}} & +q \int_{0}^{t}\left(\varepsilon+\left|\Gamma_{t}\right|^{2}\right)^{\frac{q-2}{2}} \Gamma_{t}^{*}\left[\left(\hat{b}_{t}-b\left(X_t^\delta\right)\right) d t+\left(\hat{\sigma}_{t}-\sigma\left(X_t^\delta\right)\right) d W_{t}\right] \\& +q \frac{q-1}{2} \int_{0}^{t}\left(\varepsilon+\left| \Gamma_{t}^{2}\right| \right)^{\frac{q-2}{2}}\left|\hat{\sigma}_{t}-\sigma\left(X_t^\delta\right)\right|^{2} d t \\
\leq \varepsilon^{\frac{q}{2}} & +q \int_{0}^{t}\left(\varepsilon+\left|\Gamma_{t}\right|^{2}\right)^{\frac{q-2}{2}}\left[\Gamma_{t}^{*}\left(\hat{b}_{t}-b\left(X_t^\delta\right)\right)+\frac{q-1}{2}\left|\hat{\sigma}_{t}-\sigma\left(X_{t}^{\delta}\right)\right|^{2}\right] \dd t \\
& +q \int_{0}^{t}\left(\varepsilon+\left|\Gamma_{t}\right|^2 \right)^{\frac{q-2}{2}} \Gamma_{t}^{*}\left(\hat{\sigma}_{t}-\sigma\left(X_{t}^{\delta}\right)\right) \dd W_{t}.
\end{aligned}
$$
Using the properties of $b$ and $\sigma$ in Condition \ref{condition b sigma} in Assumption \ref{assumption general}, we continue with
$$
\begin{aligned}
 \Big( \varepsilon+&\left|\Gamma_{t}\right|^{2}\Big)^{\frac{q}{2}}  \\
 \leq& 
 \varepsilon^{\frac q2} +q \int_{0}^{t}\left(\varepsilon+\left|\Gamma_{t}\right|^2 \right)^{\frac{q-2}{2}} \Gamma_{t}^{*}\left(\hat{\sigma}_{t}-\sigma\left(X_{t}^{\delta}\right)\right) \dd W_{t} \\
& +q \int_{0}^{t}\left(\varepsilon+\left|\Gamma_{t}\right|^{2}\right)^{\frac{q-2}{2}}\left[\Gamma_{t}^{*}\left(b\left(\hat{X}_{t}\right)-b\left(X_{t}^{\delta}\right)\right)+(q-1)\left|\sigma\left(\hat{X}_{t}\right)-\sigma\left(X_{t}^{\delta}\right)\right|^{2}\right] \dd t \\
& +q  \int_{0}^{t}\left(\varepsilon+\left|\Gamma_{t}\right|^{2}\right)^{\frac{q-2}{2}}\left[\Gamma_{t}^{*}\left(\hat{b}_{t}-b\left(\hat{X}_{t}\right)\right)+(q-1)\left|\hat{\sigma}_{t}-\sigma\left(\hat{X}_{t}\right)\right|^{2}\right] \dd t  \\
\leq& \varepsilon^{\frac{q}{2}} +q \int_{0}^{t}\left(\varepsilon+\left|\Gamma_{t}\right|^2 \right)^{\frac{q-2}{2}} \Gamma_{t}^{*}\left(\hat{\sigma}_{t}-\sigma\left(X_{t}^{\delta}\right)\right) \dd W_{t} \\
& +q  \int_{0}^{t}\left(\varepsilon+\left|\Gamma_{t}\right|^{2}\right)^{\frac{q-2}{2}}\left(L+(q-1) L^{2}\right)\left|\Gamma_{t}\right|^{2} \dd t \\
& +q  \int_{0}^{t}\left(\varepsilon+\left|\Gamma_{t}\right|^{2}\right)^{\frac{q-2}{2}}\Big(\left|\Gamma_{t}^{*}\right| \delta(1-\delta) L\left|X_{t}^{\bar{x}}-X_{t}^{x}\right|^{2} \\
& \qquad \qquad \qquad \qquad \qquad \quad +(q-1) \delta^{2}(1-\delta)^{2} L^{2}\left|X_{t}^{\bar{x}}-X_{t}^{x}\right|^{4}\Big) \dd t  .
\end{aligned}
$$
Next, since $y \bar y \leq \frac12 (y^2 +\bar y^2)$, after rearranging some of the terms we arrive at the expression
$$
\begin{aligned}
\left(\varepsilon+\left|\Gamma_{t}\right|^{2}\right)^{\frac{q}{2}}  \leq& \varepsilon^{\frac{q}{2}} +q \int_{0}^{t}\left(\varepsilon+\left|\Gamma_{t}\right|^2 \right)^{\frac{q-2}{2}} \Gamma_{t}^{*}\left(\hat{\sigma}_{t}-\sigma\left(X_{t}^{\delta}\right)\right) \dd W_{t}\\
& +q  \int_{0}^{t}\left(\varepsilon+\left|\Gamma_{t}\right|^{2}\right)^{\frac{q-2}{2}}\left(L+(q-1) L^{2}\right)\left|\Gamma_{t}\right|^{2} \dd t  \\
& + q  \int_{0}^{t}\left(\varepsilon+\left|\Gamma_{t}\right|^{2}\right)^{\frac{q-2}{2}}
\bigg(\frac{ L}{2}\left(\left|\Gamma_{t}\right|^{2}+\delta^2(1-\delta)^2\left|X_{t}^{\bar{x}}-X_{t}^{x}\right|^{4}\right) \\
& \qquad \qquad \qquad \qquad \qquad \quad +(q-1)  L^{2} \delta^2(1-\delta)^2 \left|X_{t}^{\bar{x}}-X_{t}^{x}\right|^{4}\bigg) \dd t   \\
\leq&  \varepsilon^{\frac{q}{2}} +q \int_{0}^{t}\left(\varepsilon+\left|\Gamma_{t}\right|^2 \right)^{\frac{q-2}{2}} \Gamma_{t}^{*}\left(\hat{\sigma}_{t}-\sigma\left(X_{t}^{\delta}\right)\right) \dd W_{t} \\
& +q  \int_{0}^{t}\left(\varepsilon+\left|\Gamma_{t}\right|^{2}\right)^{\frac{q-2}{2}}
\bigg[\left(\frac{3}{2} L+(q-1) L^{2}\right)\left|\Gamma_{t}\right|^{2} \\
& \qquad \qquad \qquad \qquad \qquad \quad 
+\delta^2(1-\delta)^2 \left(\frac{L}{2}
+(q-1) L^{2}\right)\left|X_{t}^{\bar{x}}-X_{t}^{x}\right|^{4}\bigg] \dd t.
\end{aligned}
$$
%Set $\hat{L}_{1}:=q \left(\frac{3}{2} L+(q-1) L^{2}\right), \ \hat{L}_{2}:=q\left(\frac{L}{2}+(q-1) L^{2}\right), \ \alpha:=\frac{q}{q-2}, \ \alpha^{\prime}:=\frac{q}{2} \quad$ and notice that $\frac{1}{\alpha}+\frac{1}{\alpha}=1$.
Next, using Young inequality with exponent $\frac{q}{q-2}$ and its conjugate $\frac{q}{2}$, we obtain
\begin{equation}\label{eq useful semiconc SDE estimate}
\begin{aligned}
 \Big(\varepsilon+&\left|\Gamma_{t}\right|^{2}\Big)^{\frac{q}{2}}  \\
 \leq  &  
 \varepsilon^{\frac q 2}   +q \int_{0}^{t}\left(\varepsilon+\left|\Gamma_{t}\right|^2 \right)^{\frac{q-2}{2}} \Gamma_{t}^{*}\left(\hat{\sigma}_{t}-\sigma\left(X_{t}^{\delta}\right)\right) \dd W_{t} \\
& + q \left(\frac{3}{2} L+(q-1) L^{2}\right) \int_{0}^{t} \left(\varepsilon+\left|\Gamma_{t}\right|^{2}\right)^{\frac{q}{2}}  \dd t \\
& + \delta^2(1-\delta)^2 q\left(\frac{L}{2}+(q-1) L^{2}\right) \int_0^t \left( \frac{q-2}{q}  \left(\varepsilon+\left|\Gamma_{t}\right|^{2}\right)^{\frac{q}{2}}+\frac{2}{q} \left|X_{t}^{\bar{x}}-X_{t}^{x}\right|^{2 q}\right) \dd t  \\
%= & \varepsilon^{\frac q 2}+\left(q\left(\frac{3}{2} L+(q-1) L^{2}\right)+(q-2)\left(\frac{L}{2}+(q-1) L^{2}\right)\right) \int_{0}^{t} \mathbb{E}\left[\left(\varepsilon+\left|\Gamma_{t}\right|^{2}\right)^{\frac{q}{2}}\right] d t \\& \quad+\left(L+2(q-1) L^{2}\right) \int_{0}^{t} \mathbb{E}\left[\left|X_{t}^{\bar{x}}-X_{t}^{x}\right|^{2 q}\right] d t \\
\leq & \varepsilon^{\frac q 2} +q \int_{0}^{t}\left(\varepsilon+\left|\Gamma_{t}\right|^2 \right)^{\frac{q-2}{2}} \Gamma_{t}^{*}\left(\hat{\sigma}_{t}-\sigma\left(X_{t}^{\delta}\right)\right) \dd W_{t} \\ 
& +\left((2 q-1) L+2(q-1)^{2} L^{2}\right) \int_{0}^{t}  \left(\varepsilon+\left|\Gamma_{t}\right|^{2}\right)^{\frac{q}{2}} \dd t \\ 
& \quad \quad + \delta^2(1-\delta)^2 \left(L+2(q-1) L^{2}\right) \int_{0}^{t}  \left|X_{t}^{\bar{x}}-X_{t}\right|^{2 q}  \dd t.
\end{aligned}
\end{equation}
Taking expectations and %from which we conclude that
%$$\begin{aligned}\mathbb{E}\left[\left(\varepsilon+\left|\Gamma_{t}\right|^{2}\right)^{\frac{q}{2}}\right] \leq& \varepsilon^{\frac q 2}+\left((2 q-1) L+2(q-1)^{2} L^{2}\right) \int_{0}^{t} \mathbb{E}\left[\left(\varepsilon+\left|\Gamma_{t}\right|^{2}\right)^{\frac{q}{2}}\right] d t \\ & \quad \quad +\left(L+2(q-1) L^{2}\right) \int_{0}^{t} \mathbb{E}\left[\left|X_{t}^{\bar{x}}-X_{t}^x\right|^{2 q}\right] d t.\end{aligned}$$
using Grönwall,  \eqref{eq estimate SDE Lip} implies that
$$
\mathbb{E}\left[\left(\varepsilon+\left|\Gamma_{t}\right|^{2}\right)^{\frac{q}{2}}\right] \leq\left(\varepsilon^{\frac q 2}+C \delta^2(1-\delta)^2 |\bar{x}-x|^{2 q} e^{2 q\left(L+\frac{2 q-1}{2} L^{2}\right) t}\right) e^{\left((2 q-1) L+2(q-1)^{2} L^{2}\right) t},
$$
which in turn gives \eqref{eq estimate SDE semiconv} by sending $\varepsilon \rightarrow 0$. 
%, we conclude that$$\mathbb{E}\left[\left|\Gamma_{t}\right|^{q}\right] \leq  C|\bar{x}-x|^{2 q} e^{\left(\left(4 q-1\right) L+\left(4 q^{2}-5 q+2\right) L^{2}\right) t}.$$
\smallbreak\noindent
\emph{Proof of \eqref{eq estimate SDE growth SUP}.}
%Set $Z_{t}:=e^{-\rho t}\left(1+\left(X_{t}^{x}\right)^{2}\right)$
First, by It\^o's formula we have
$$
\begin{aligned}
e^{-\rho t}& \left(1+\left|X_{t}^{x}\right|^{2}\right)^{\frac{q}{2}}\\
=& \left(1+|x|^{2}\right)^{\frac{q}{2}}  +q \int_{0}^{t} e^{-\rho s}\left(1+\left| X_t^x \right|^{2}\right)^{\frac{q}{2}-1} (X_{t}^x)^{*} \sigma (X_t^x ) \dd W_{t}   \\
& +\int_{0}^{t} e^{-\rho t}\left(1+\left|X_{t}^{x}\right|^{2}\right)^{\frac{q}{2}}
\bigg[-\rho +q\bigg(\frac{\left(X_{t}^{x}\right)^{*}}{1+\left|X_{t}^{x}\right|^{2}} b\left(X_{s}^x\right) \\
& \qquad \qquad \qquad \qquad \qquad \qquad
+ \left(\frac{q}{2}-1\right) \frac{|( X_{t}^{x})^* \sigma( X_{t}^x ) | ^2} {
\left( 1+| X_t^x|^2 \right)^{2}} + \frac{1}{2} \frac{|\sigma  ( X_t^x)|^2}{  1+ | X_t^x|^{2} } \bigg) \bigg] \dd t \\
\leq & (1+|x|^2)^{\frac{q}{2}}+ q \int_{0}^{t} e^{-\rho t} \left( 1+\left|X^x_t \right|^{2}\right)^{\frac{q}{2}-1} (X^x_{t})^{*} \sigma\left(X^x_{t} \right) \dd W_{t} \\
& + C \int_{0}^{t} e^{-\rho t}\left(1+\left|X^x_{t}\right|^{2}\right)^{\frac{q}{2}} \dd t 
.
\end{aligned}
$$
Next, for $\tau \in \mathcal T$ and $T >0$, Burkholder-Davis-Gundy inequality gives
$$
\begin{aligned}
 \mathbb{E}\left[\sup _{t \leq \tau \land T } e^{-\rho t}\left(1+\left|X_{t}^{x}\right|^{2}\right)^{\frac{q}{2}}\right] \leq&  \left(1+|x|^{2}\right)^{\frac{q}{2}}+C \int_{0}^{T} e^{-\rho t} \mathbb{E}\left[\left(1+\left|X_{t}^{x}\right|^{2}\right)^{\frac{q}{2}}\right] \dd t \\
&+C \mathbb{E}\left[ \left( \int_{0}^{T} e^{-2 \rho t}\left(1+\left|X_{t}^{x}\right|^{2}\right)^{q-2}\left|X_{t}^{x}\right|^{2} | \sigma (X_{t}^{x})|^{2} \dd t\right)^{\frac{1}{2}} \right]. 
\end{aligned}
$$
Thus, using \eqref{eq estimate SDE growth} and Jensen inequality we find
$$
\begin{aligned} 
\mathbb{E}\left[\sup _{t \leq \tau \land T} e^{-\rho t}\left(1+\left|X_{t}^{x}\right|^{2}\right)^{\frac{q}{2}}\right] \leq& C \left(1+|x|^{q}\right) + C \left(1+|x|^{q}\right) \int_{0}^{\infty} e^{\left(-\rho + \hat{c}_{0}(q) \right) t} \dd t \\
& 
+ C \left(\int_{0}^{\infty} e^{-2 \rho t} \mathbb{E} \left[ \left(1+\left|X_{t}^{x}\right|^{2}\right)^{q}\right] \dd t\right)^{\frac{1}{2}} \\
& \leq C\left(1+|x|^{q}\right)
+C\left(1+|x|^{q}\right) \left(\int_{0}^{\infty} e^{(-2 \rho + \hat c _0 (2q) )t} \dd t\right)^{\frac{1}{2}} \\
& \leq C \left(1+|x|^{q}\right), 
\end{aligned}
$$
where the finiteness of the integrals follows from the conditions $\rho > \hat c _0(q)$ and $2 \rho > \hat c _0 (2q)$ and the particular choice of $q$ in \eqref{eq estimate SDE growth SUP} (which are part of Condition \ref{condition rho} in Assumption \ref{assumption general}).
The estimate \eqref{eq estimate SDE growth SUP} follows taking limits as $T \to  \infty$ and using the dominated convergence theorem.

\smallbreak\noindent
\emph{Proof of \eqref{eq estimate SDE Lip SUP}.}
With the notation $\Delta_{t}:=X_{t}^{\bar{x}}-X_{t}^{x}$, 
similarly to \eqref{eq useful Lip} we find
$$
\begin{aligned}
\left(\varepsilon+\left|\Delta_{t}\right|^{2}\right)^{\frac{q}{2}} \varepsilon\left(\varepsilon+|\bar{x}-x|^{2}\right)^{\frac{q}{2}} & +\left(-\rho+q\left(L+\frac{q-1}{2} L^{2}\right)\right) \int_{0}^{t} e^{-\rho t}\left(\varepsilon+\left|\Delta_{t}\right|^{2}\right)^{\frac{q}{2}} \dd t \\
& +q \int_{0}^{t} e^{-\rho t}\left(\varepsilon+\left|\Delta_{t}\right|^{2}\right)^{\frac{q-2}{2}} \Delta_{t}^{*}\left(\sigma\left(X_{t}^{\bar{x}}\right)-\sigma\left(X_{t}^{x}\right)\right) \dd W_{t}.
\end{aligned}
$$
For $\tau \in \mathcal T$ and $T >0$, by using Burkholder-Davis-Gundy inequality  and Jensen inequality, thanks to \eqref{eq estimate SDE Lip} we obtain
$$
\begin{aligned}
\mathbb{E}\left[\sup _{t \leq \tau \land T}\left(\varepsilon+\left|\Delta_{t}\right|^{2}\right)^{\frac{q}{2}}\right] 
\leq& 
C (\varepsilon^{\frac q 2}+|\bar{x}-x|^{q}) +C \int_{0}^{T} e^{-\rho t} \mathbb{E}\left[|\Delta_t|^{q} \right] \dd t \\
& +C\left(\int_{0}^{T} e^{-2 \rho t } \mathbb{E}\left[  |\Delta _t|^{2 q} \right] \dd t\right)^{\frac{1}{2}} \\
 \leq& C \left( \varepsilon^{\frac q 2} + |\bar{x}-x|^{q} \left( 1 + \int_0^\infty \left(  e^{(-\rho + \hat c_1(q) )t} +  e^{(-2\rho + \hat c_1(2q) )t}\right) \dd t \right) \right) \\
 \leq & C \left( \varepsilon^{\frac q 2} + |\bar{x}-x|^{q}\right),
\end{aligned}
$$
where the finiteness of the integrals follows from the conditions $\rho > \hat c _1(q)$ and $2 \rho > \hat c _1 (2q)$ and the choice of $q$ in \eqref{eq estimate SDE Lip SUP} (which are part of Condition \ref{condition rho} in Assumption \ref{assumption general}).

From the latter estimate, take limits as $\varepsilon \to 0$ and $T \to \infty$ and use monotone convergence theorem to derive  \eqref{eq estimate SDE Lip SUP}.

\smallbreak\noindent
\emph{Proof of \eqref{eq estimate SDE semiconv SUP}.}
With the same notations as in the proof of \eqref{eq estimate SDE semiconv} above, by repeating the steps leading to \eqref{eq useful semiconc SDE estimate}, we find
$$
\begin{aligned}
e^{-\rho t}\left(\varepsilon+\left|\Gamma_{t}\right|^{2}\right)^{\frac{q}{2}} \leqslant  \varepsilon^{\frac{q}{2}} & +C \int_{0}^{t} e^{-\rho t}\left(\varepsilon+\left|\Gamma_{t}\right|^{2}\right)^{\frac{q}{2}} \dd t+C  \delta^2(1-\delta)^2 \int_{0}^{t} e^{-\rho t}\left|X_{t}^{\bar{x}}-X_{t}^{x}\right|^{2 q} \dd t \\
& +q \int_{0}^{t} e^{-\rho t} \left(\varepsilon+\left|\Gamma_{t}\right|^{2}\right)^{\frac{q-2}{2}} \Gamma_{t}^{*}\left(\hat{\sigma}_{t}-\sigma\left(X_t^\delta\right)\right) \dd W_{t}
\end{aligned}
$$
Hence, for $\tau \in \mathcal T$ and $T >0$, by using Burkholder-Davis-Gundy inequality  and Jensen inequality, thanks to \eqref{eq estimate SDE Lip} and \eqref{eq estimate SDE semiconv} we get
\begin{equation}\label{eq estimate original semconc}
\begin{aligned}
\mathbb{E}\bigg[\sup _{t \leq \tau \land T} & e^{-\rho t}\left(\varepsilon+\left|\Gamma_ t \right|^{2}\right)^{\frac{q}{2}}\bigg]  \\
\leq & C \varepsilon^{\frac{q}{2}} 
+C \int_{0}^{T} e^{-\rho t} \mathbb E \left[  \left|\Gamma_{t}\right|^{q} \right] \dd t \\
& + C  \delta^2(1-\delta)^2 \int_{0}^{T} e^{-\rho t}\mathbb E \left[ \left|X_{t}^{\bar{x}}-X_{t}^{x}\right|^{2 q} \right] \dd t \\
& +C \bigg( \int_{0}^{T} e^{-2 \rho t} 
\mathbb E \left[ \left(\varepsilon+\left|\Gamma_{t}\right|^{2}\right)^{q-2} |\Gamma_{t}|^{2}\left| \hat{\sigma}_{t}-\sigma\left(X_t^\delta\right)\right|^2 \right] \dd t \bigg)^{\frac{1}{2}}.
\end{aligned}
\end{equation}
To estimate the last term in the right-hand side, we first use the regularity of $\sigma$ to obtain
$$
\begin{aligned}
\int_{0}^{T} &  e^{-2 \rho t} 
\mathbb E \left[ \left(\varepsilon+\left|\Gamma_{t}\right|^{2}\right)^{q-2} |\Gamma_{t}|^{2}\left| \hat{\sigma}_{t}-\sigma\left(X_t^\delta\right)\right|^2 \right] \dd t \\
& \leq C \int_{0}^{T} e^{-2 \rho t} 
\mathbb E \left[ \left(\varepsilon+\left|\Gamma_{t}\right|^{2}\right)^{q-1} \left( \left| \hat{\sigma}_{t}  - \sigma \left( \hat X _t \right) \right|^2 + \left| \sigma \left( \hat X _t \right) - \sigma\left(X_t^\delta\right)\right|^2 \right) \right] \dd t \\
& \leq C \int_{0}^{T} e^{-2 \rho t} 
\mathbb E \left[ \left(\varepsilon+\left|\Gamma_{t}\right|^{2}\right)^{q-1} \left( \delta^2 (1-\delta)^2  \left|   \hat X _t  - X_t^\delta \right|^4 + \left| \Gamma_t \right|^2 \right) \right] \dd t, 
\end{aligned}
$$
and then employ Young inequality with exponent $\frac{q}{q-1}$ and its conjugate $q$ to conclude that
$$
\begin{aligned}
\int_{0}^{T} & e^{-2 \rho t} 
\mathbb E \left[ \left(\varepsilon+\left|\Gamma_{t}\right|^{2}\right)^{q-2} |\Gamma_{t}|^{2}\left| \hat{\sigma}_{t}-\sigma\left(X_t^\delta\right)\right|^2 \right] \dd t \\
& \leq C \left( \varepsilon^{\frac q 2} 
+\int_{0}^{T} e^{-2 \rho t} \mathbb E \left[ \left| \Gamma_t \right|^{2q}  \right] \dd t 
+ \delta^2 (1-\delta)^2  \int_0^T e^{-2 \rho t} \mathbb E \left[ \left|   \hat X _t  - X_t^\delta \right|^{4q} \right] \dd t \right). 
\end{aligned}
$$
By plugging the latter inequality into \eqref{eq estimate original semconc}, and then using \eqref{eq estimate SDE Lip} and \eqref{eq estimate SDE semiconv}, we find
$$
\begin{aligned}
\mathbb{E} &\left[\sup _{t \leq \tau \land T} e^{-\rho t}\left(\varepsilon+\left|\Gamma_ t \right|^{2}\right)^{\frac{q}{2}}\right] \\
& \leq  C \left(  \varepsilon^{\frac{q}{2}} 
+\int_{0}^{T} e^{-\rho t} \mathbb E \left[  \left|\Gamma_{t}\right|^{q} \right] d t + \delta^2 (1-\delta)^2  \int_{0}^{T} e^{-\rho t}\mathbb E \left[ \left|X_{t}^{\bar{x}}-X_{t}^{x}\right|^{2 q} \right] \dd t \right.\\
& \quad \qquad \ \ \left. + \bigg(\varepsilon^{\frac q 2} 
+\int_{0}^{T} e^{-2 \rho t} \mathbb E \left[ \left| \Gamma_t \right|^{2q}  \right] \dd t 
+ \delta^2 (1-\delta)^2  \int_0^T e^{-2 \rho t} \mathbb E \left[ \left|   \hat X _t  - X_t^\delta \right|^{4q} \right] \dd t    \bigg)^{\frac{1}{2}} \right) \\
& \leq C \left(  \varepsilon^{\frac{q}{2}} 
+ \delta^2(1-\delta)^2 |\bar{x}-x|^{2 q} \int_{0}^{\infty} \left( e^{(-\rho+ \hat c_2(q)) t} + e^{(-\rho+ \hat c_1(2q)) t} \right) \dd t  \right.\\
& \quad \qquad \ \ \left. + \bigg(\varepsilon^{\frac q 2} 
+  \delta^2(1-\delta)^2 |\bar{x}-x|^{4 q} \int_{0}^{\infty} \left( e^{(-2 \rho + \hat c_2 (2q)) t} + e^{(-2 \rho + \hat c_1 (4q)) t} \right) \dd t \bigg)^{\frac{1}{2}} \right).
\end{aligned}
$$ 
Finally, by sending $\varepsilon \to 0$, we conclude that
$$
\begin{aligned}
\mathbb{E} \left[\sup _{t \leq \tau \land T} e^{-\rho t} \left|\Gamma_ t \right|^{q} \right] 
\leq&  C \delta(1-\delta) |\bar{x}-x|^{2 q}   \int_{0}^{\infty} \left( e^{(-\rho+ \hat c_2(q)) t} + e^{(-\rho+ \hat c_1(2q)) t} \right) \dd t \\
&+  C \delta(1-\delta) |\bar{x}-x|^{2 q}  \bigg(\int_{0}^{\infty} \left( e^{(-2 \rho + \hat c_2 (2q)) t} + e^{(-2 \rho + \hat c_1 (4q)) t} \right) \dd t \bigg)^{\frac{1}{2}} \\
\leq &  C \delta(1-\delta) |\bar{x}-x|^{2 q}, 
\end{aligned}
$$
where the finiteness of the integrals follows from taking $q=2$ and from the conditions $\rho > \hat c _1(4q)$ and $2 \rho > \hat c _2 (2q)$ (which are part of Condition \ref{condition rho} in Assumption \ref{assumption general}).
The estimate \eqref{eq estimate SDE Lip SUP} follows taking limits as $T \to  \infty$ and using the dominated convergence theorem.

\vspace{10pt}

\noindent \textbf{Acknowledgements.} 
Supported by 
the project E83C25000470005 financed by University of Rome Tor Vergata, by the Deutsche Forschungsgemeinschaft (DFG, German Research Foundation) - Project-ID 317210226 - SFB 1283 and by the NSF CAREER award DMS-2339240.

\bibliographystyle{siam}
\bibliography{main.bib,references}

\newpage

\end{document}